\documentclass[a4paper,10pt]{article}
\usepackage[margin=1in]{geometry}
\usepackage{textcomp}

\usepackage{amssymb}
\usepackage{amsmath}

\usepackage{amsthm}

\usepackage[hyperfootnotes=false]{hyperref} 
\hypersetup{colorlinks=false}

\usepackage[utf8]{inputenc}

\newtheorem{theorem}{Theorem}[section]
\newtheorem{lemma}[theorem]{Lemma}

\newtheorem{cor}[theorem]{Corollary}
\newtheorem{prop}[theorem]{Proposition}

\hypersetup{colorlinks=false}
\usepackage{enumerate}
\usepackage{bbm}
\usepackage{listings}   
\usepackage{bm}
\usepackage{dsfont}
\usepackage{mathtools}
\usepackage{etextools}

\DeclarePairedDelimiter\floor{\lfloor}{\rfloor}
\begin{document}

\title{Vinogradov's Theorem with Fouvry-Iwaniec Primes}
\author{Lasse Grimmelt}
%\email{l.p.grimmelt@uu.nl}
%\address{}
%
\date{}
%\dedication{}
%\classification{11P32, (11N36)}
%\keywords{Goldbach-type theorems, Transference Principle}
%\thanks{}
\maketitle

\begin{abstract}
We show that every sufficiently large $x\equiv 3(4)$ can be written as the sum of three primes, each of which is a sum of a square and a prime square. The main tools are a transference version of the circle method and various sieve related ideas. In particular, a majorant of the set of primes of interest is constructed that overestimates it by a factor of less than $3$ and for which we have good control of the Fourier transform.
\end{abstract}

\section{Introduction}
In this paper we study a ternary additive problem involving a subset of the primes.  The result that every sufficient large odd integer can be represented as the sum of three primes was first proved by Vinogradov \cite{vin} in 1937. By the work of Helfgott \cite{hel}, we now know that this holds true for every odd $x\geq 7$. Hence, the weak Goldbach conjecture is solved. We consider primes $p$ of the form
\begin{align}
p=k^2+l^2, \label{foiprime}
\end{align}
where $k$ is an integer and $l$ is prime. Fouvry and Iwaniec \cite{foi} proved in 1997 that there are infinitely many such primes and we call them Fouvry-Iwaniec primes. Their result is an application of the asymptotic prime sieve, strikingly breaking the parity barrier of sieve theory. Our goal is a version of Vinogradov's Theorem restricted to these primes. As only primes $p\equiv 1(4)$ can be Fouvry-Iwaniec primes, if $x$ is the sum of three of those it necessarily holds that $x\equiv 3(4)$.  We show that this condition is sufficient for large $x$.
\begin{theorem}\label{Goal}
Every sufficiently large integer $x\equiv 3(4)$ can be written as the sum of three Fouvry-Iwaniec primes.
\end{theorem}
In contrast to the work of Fouvry and Iwaniec, who can handle \eqref{foiprime} in more general cases than $l$ being prime, we make crucial use of the restriction to this case. 

A natural approach to prove Theorem \ref{Goal} is to apply the Circle Method. If we have sufficient understanding of
\begin{align}\label{foifou}
\sum_{n\leq x} \Lambda^\Lambda(n) e(\alpha n),
\end{align}
where
\begin{align}\label{LLdef}
\Lambda^\Lambda(n):=\Lambda(n)\sum_{n=k^2+l^2}\Lambda(l)
\end{align}
is a weighted indicator of our set of interest, we can solve the ternary counting problem. This would resolve in an asymptotic of the form
\begin{align}\label{eq_expasym}
 \sum_{x=n_1+n_2+n_3}\Lambda^\Lambda(n_1)\Lambda^\Lambda(n_2)\Lambda^\Lambda(n_3) \sim \mathfrak{S}(x)x^2,
\end{align}
where the singular series $\mathfrak{S}(x)$ encodes the density of local solution. To apply the Circle Method, we need two different types of results: First, for $\alpha$ lying on the so called \emph{major arcs}, we need to understand the distribution of \eqref{LLdef} in arithmetic progressions to a certain range of moduli. While this is not done explicitely in \cite{foi}, such a result can be deduced without much problem, see Theorem \ref{GFI}. Second, for the remaining $\alpha$ we would require a \emph{minor arc bound} and this is where the approach seems to be stuck. The reason for this lies in the occurrence of two prime indicators. Applying the standard tool of combinatorical decompositions on both of them gives four combinations of Type I / II sums. All of these can be bound sufficiently, except the case of two Type II sums. Sieve theory tells us that we cannot reach primes without Type II information. Consequently it seems that new ideas would be required to prove the minor arc bounds for \eqref{foifou}. One of the central point of this paper is the following: We can still prove Theorem \ref{Goal} even though we don't have sufficient Fourier information for the set we are interested in.

The proof of Theorem \ref{Goal} is made possible by the \emph{transference Circle Method}. This concept originates in Green's work \cite{gre} in which the existence of infinitely many three term arithmetic progression in dense subsets of the primes is proved. We use a version that can be found in recent work of Maynard, Matomäki, and Shao on Vinogradov's Theorem with almost equal summands \cite{mms}. The advantage of the transference Circle Method is that we require exponential sum information not for the Fouvry-Iwaniec primes (see \eqref{foifou}), but for some sufficiently tight majorant of them. As mentioned above, we are able to handle all combinations of Type I/II sums in which there is at least one Type I sum. As sieves are Type I sums, this naturally leads us to using them for our majorant. One disadvantage of the transference Circle Method is that it does not give us the expected asymptotic \eqref{eq_expasym}. Instead of the pure existence of solutions given by Theorem \ref{Goal}, the method we use could be sharpened to give an explicit lower bound. However, it is clear that it would give the wrong leading constant. In addition, without further modification of the arguments given here, it would also miss the expected order of magnitude by at least $(\log x)^2$, due to the $W$-trick.

The majorant is constructed by using the linear sieve in conjunction with Chen-Iwaniec sieve switching, see \cite{che} and \cite{iwa}. For estimating minor arc exponential sums associated to the majorant we use the combinatorical dissection into Type I and II sums that was described by Duke, Friedlander, and Iwaniec in \cite{dfi}. This is necessary since the commonly used identities of Vaughan and Heath-Brown are insufficient for our purpose. The more difficult Type II sums are approached by following the idea in \cite{foi} to step into the Gaussian integers. This results in sums over lattices and by choosing a suitable basis we can catch cancellation caused by the minor arc phase in a adequate large range. Since the requirements for finding three term arithmetic progressions in dense subsets are weaker than what we need for Theorem \ref{Goal}, we also obtain the following result.
\begin{theorem}\label{3apcor}
Every subset of the Fouvry-Iwaniec primes with positive relative density contains infinitely many three term arithmetic progressions.
\end{theorem}
Both Theorem \ref{Goal} and Theorem \ref{3apcor} are deduced in the last section.

We now outline the basics of the transference principle. For a complete and accessible description we refer the reader to section 5 of Green and Tao's paper \cite{gt}. We write $[N]:=\{1, \ldots N\}$. For any function $f:[N]\mapsto\mathbbm{C}$ we denote its Fourier transform by
\begin{align*}
\hat{f} \text{: } \mathbbm{R}/\mathbbm{Z}&\to \mathbbm{C}\\
\gamma &\mapsto \sum_{n}f(n)e(\gamma n)
\end{align*}
Let now $f,\nu \text{: }  [N]\to \mathbbm{R}$ be two functions with the majorisation condition
\begin{align*}
0\leq f(n)\leq \nu(n) \text{ for all } n\in[N].
\end{align*}
Two properties that are central for the transference principle. 
\begin{enumerate}[I.]
\item The function $\nu$ behaves pseudorandom, that is
\begin{align} \label{dpseudorandom}
\bigl| \hat{\nu}(\gamma)-\widehat{1_{n\in[N]}}(\gamma) \bigr|\leq \eta N \text{ for all } \gamma\in \mathbbm{R}/\mathbbm{Z}
\end{align}
for some sufficiently small $\eta$.
\item The function $f$ fulfills for some $2<q<3$ and fixed $K\geq 1$ the $L^q$ estimate
\begin{align}\label{Lqrestrctionintro}
\int_0^1 |\hat{f}(\gamma)|^q d\gamma \leq K N^{q-1}.
\end{align}
\end{enumerate}
Green showed that if both of these are met, one can transfer the existence of ternary additive patterns in $f$ to similar ones for a function $\tilde{f}\text{: }[N]\to \mathbbm{R}$. This new function is bounded, it holds that
\begin{align*}
0\leq \tilde{f}(n) \leq 1+\epsilon
\end{align*}
for all $n$ and some small $\epsilon$ that depends on $\eta$ and $K$. Furthermore its density in the integers is the same as the relative density of $f$ in $\nu$, i.e. it holds that
\begin{align*}
\frac{\sum_{n\leq N}\tilde{f}(n)}{N} =\frac{\sum_{n}f(n)}{\sum_n \nu(n)}.
\end{align*}
As $\tilde{f}$ behaves set-like, one can use different techniques to ensure the existence of solutions to a given problem in $\tilde{f}$. The property \eqref{dpseudorandom} is closely related to major arc asymptotics and minor arc bounds of the classical Circle Method. After dealing with the effect of small moduli that cause major arc spikes with help of the $W$-trick, \eqref{dpseudorandom} can be deduced for 
\begin{align*}
&\nu(n)=\begin{cases}\Lambda(n), &\text{ if } n\leq N\\
0, &\text{ else,}
\end{cases}
\end{align*}
where $\Lambda$ is the von Mangoldt function, by using Vinogradov's minor arc bound. In the original work, Green used the concept to reduce the statement 
\begin{quote}\emph{Every positive density subset of the \emph{primes} contains infinitely many 3-term arithmetic progressions}
\end{quote}
to
\begin{quote}\emph{Every positive density subset of the \emph{integers} contains infinitely many 3-term arithmetic progressions.}\end{quote} As the second one is known as Roth's Theorem, Green proved Roth's Theorem in the primes in this way.

By replacing Roth's Theorem with a suitable statement about sums of three integers lying in some dense set, the transference principle can be used to obtain lower bounds for
\begin{align*}
r(x,f)=\sum_{n_1+n_2+n_3=x}f(n_1)f(n_2)f(n_3).
\end{align*}
For example in the works of Li and Pan \cite{lpd} and Shao \cite{sh1}, it was used to consider Goldbach's weak problem in dense subsets of the primes. 

This concept is useful for other aspects than just considering arbitrary dense subsequences of sequences $\nu$ for which $r(x,\nu)$ is understood. Besides the mentioned result \cite{mms} it was used in another work of Shao \cite{sh2} for an alternative proof of Vinogradov's Theorem that does not rely on the theory of L-functions and is as such notably independent of Siegel's Theorem. In these works a given interesting $f$ is majorised by some $\nu$ to obtain a lower bound for $r(x,f)$. As an advantage over the classical Circle Method one does then not need the full strength of results for $\hat{f}$ but for $\hat{\nu}$ instead. This is the starting point for the proof of Theorem \ref{Goal}. Compared to \cite{sh2} and \cite{mms} in which the major arc case is the one with insufficient information for the classical Circle Method, we use the transference principle to overcome lacking minor arc bounds.

In contrast to finding three term arithmetic progressions, our problem is not translation invariant. This means that the results are sensitive to the relative density. The result most suitable for us in this aspect is stated in \cite{mms} and roughly means the following. If \eqref{dpseudorandom} and \eqref{Lqrestrctionintro} are met, we can get a lower bound for $r(x,f)$ as long as 
\begin{align*}
\frac{\sum \nu(n)}{\sum f(n)}< 3
\end{align*}
and $f$ is not too sparse in arithmetic progressions. The precise statement of the transference principle we use is given by Theorem \ref{MT} in the next subsection. It contains three conditions, the first two of which are variants of \eqref{dpseudorandom} and \eqref{Lqrestrctionintro}. The third combines the restrictions on the density of the Fouvry-Iwaniec primes in arithmetic progressions and relative to the majorant. For the Fouvry-Iwaniec primes, the density in progressions follows from the aforementioned generalisation of \cite{foi}, Theorem \ref{GFI}. Furthermore the $L^q$ estimate \eqref{Lqrestrctionintro} can be proved in a straightforward manner by using ideas of Bourgain \cite{bou}. Consequently the main work consists in finding a suitable majorant. 

To describe the construction of the majorant we recall the weighted indicator function $\Lambda^\Lambda(n)$ given in \eqref{LLdef}. We construct a majorant function $\Lambda^+(n)\geq \Lambda^\Lambda(n)$ that, to prove Theorem \ref{Goal}, has to fulfill two properties. First, the majorant cannot overestimate by too much. We have to show that
\begin{align}\label{alpha+1}
\lim_{x\to \infty}\frac{\sum_{n\leq x}\Lambda^+(n)}{\sum_{n\leq x}\Lambda^\Lambda(n)}<3.
\end{align}
Second, $\Lambda^+$ should behave pseudorandomly \eqref{dpseudorandom}. To achieve these conditions, we use the following steps.
\begin{enumerate}
\item In the construction of $\Lambda^+$ we majorise the inner occurrence of $\Lambda$ by the linear sieve.
\item We decompose the remaining outer von Mangoldt function of $\Lambda^+$ into Type I and II sums.
\item We prove minor arc bounds for these Type I and II sums.
\end{enumerate}
These three steps interact and restrict each other in the following manner.

We start with the minor arc bound, step 3, since it restricts the range of Type I/II sums and the level of the sieve. Type I sums in our case are of the form
\begin{align}\label{T1Intro}
\sum_{d\leq D_I} \bigl|\sum_{\substack{dn\leq x\\ dn=k^2+l^2}}e(\gamma dn)\omega(l) \bigr|,
\end{align}
for some weight $\omega$. We estimate them successfully as long as 
\begin{align}\label{typeIrangecondint}
D_I \leq x^{1/2}(\log x)^{-B}
\end{align}
for some sufficiently large $B$. The Type II sums are more intricate and for them the level of the employed sieve from step 1 that we denote by $D_S$ plays a role. After opening the sieve we have to estimate
\begin{align}\label{T2Intro}
\sum_{\substack{c\leq D_S\\ U_1\leq d \leq U_2}} \lambda_c \beta_d \sum_{\substack{dn\leq x \\ \substack{dn=k^2+l^2 \\ c|l}}}\alpha_n e(\gamma dn)
\end{align}
for some $U_1$ and $U_2$ that depend on the decomposition used. Here, as usual for Type II minor arc sums, $U_1$ cannot be too small. This condition is harmless and we are mostly concerned with the size of $U_2$. To estimate \eqref{T2Intro}, we follow the arguments presented in section 8 of \cite{foi} to step into the Gaussian integers and afterwards apply Cauchy's inequality. In the resulting object the innermost sums are of the form
\begin{align}\label{einllat}
\sum_{\substack{|m|^2\sim M\\ \substack{ c_1|\Re(m\overline{l_1})\\ c_2|\Re(m\overline{l_2})}}}e(\gamma[|l_1|^2-|l_2|^2] |m|^2),
\end{align}
where the sum runs over Gaussian integers $m$, $c_i$ are square free integers less or equal than $D_S$, and $l_i$ are primitive Gaussian integers. The range of $M$ we need to consider depends on $U_1$ and $U_2$. The summation condition in \eqref{einllat} means that we sum over $m$ lying in some lattice depending on $l_i$ and $c_i$, $i\in \{1,2\}$. Since the lattice is two dimensional the first and second successive minima of it form a basis. Using this basis we can sum nontrivially as long as the discriminant of the lattice is somewhat smaller than $M$. This results in the following restriction for the level of the sieve from step 1 and the Type II range $U_2$ from step 2. We succeed if
\begin{align}\label{typeIIrangedcondint}
D_S^2U_2\leq x(\log x)^{-B}.
\end{align}

Step 2, the dissection into Type I and II sums, is done with the above range conditions in mind. For the Type I sums \eqref{T1Intro} it is enough that \eqref{typeIrangecondint} is fulfilled. In the Type II case we have the restriction \eqref{typeIIrangedcondint}. Consequently we can increase $D_S$ and so make the majorant tighter, if $U_2$ is small. The commonly used identities of Vaughan and Heath-Brown are not suitable to reach the required relative density \eqref{alpha+1}. Our tool of choice is the combinatorical dissection described in \cite{dfi}. It gives us admissible Type I sums and Type II sums with
\begin{align}\label{U_2condint}
U_2=x^{1/3}.
\end{align}
Since our sequence is sparser than the primes, when applying the dissection we use an additional sieve not to lose a crucial logarithmic factor.
 
We conclude this subsection by giving details for step 1. By \eqref{typeIIrangedcondint} and \eqref{U_2condint} we can prove minor arc estimates if
\begin{align*}
D_S\leq x^{1/3}(\log x)^{-B}
\end{align*}
for some sufficiently large $B>0$. We observe that in the sum
\begin{align*}
\sum_{n\leq x}\Lambda(n)\sum_{n=k^2+l^2}\Lambda(l)
\end{align*}
we have $l\leq \sqrt{x}$. So by the theory of the linear sieve this range of $D_S$ means a straightforward application of an upper bound sieve on the inner $\Lambda$ closely misses the required relative density \eqref{alpha+1}. We get the necessary saving by sieve switching. This idea was used by Chen \cite{che} to show that there are infinitely many primes $p$ such that $p+2$ has at most two prime factors. Since we construct a majorant and two prime indicators occur, our problem is related to finding good upper bounds for the number of twin primes. This upper bound twin prime problem was considered, among others, by Pan \cite{pan}. We follow his work and use Buchstab's identity to get a majorant for the inner von Mangoldt function that takes the form
\begin{align*}
\Lambda(l)\leq f_1(l)-f_2(l)+f_3(l).
\end{align*}
Here the functions $f_i$ are weighted indicator functions for integers having prime divisors in certain ranges only. We apply weighted upper and lower bound linear sieves $\omega^\pm$ on $f_1$ and $f_2$ respectively to arrive at
\begin{align*}
\Lambda(l)\leq \omega^+(l)-\omega^-(l)+f_3(l).
\end{align*}
For the contribution of $f_3$ we switch to sieve the outer occurrence of $\Lambda$. So the majorant is given by
\begin{align*}
\Lambda^+(n)=\Lambda(n)\sum_{n=k^2+l^2}\bigl(\omega^+(l)-\omega^-(l)\bigr)+\Omega^+(n)\sum_{n=k^2+l^2}f_3(l),
\end{align*}
where $\Omega^+$ is an upper bound linear sieve.

We remark that there is another type of transference Circle Method that is used in the result on Vinogradov's Theorem with almost twin primes by Matomäki and Shao \cite{ms}, and the related work of Teräväinen \cite{te}. It has some similarities to the  strategy of \cite{mms} that we use here. In both cases convolution with Bohr sets is used to make some weighted indicator set-like. However, the underlying motivation is different. The result of Fouvry and Iwaniec gives an asymptotic for the sum over \eqref{LLdef}, but in many cases one relies on sieve methods to get lower bounds. One the one hand, our problem is that the Fourier transform of $\Lambda^\Lambda$ is difficult. By the parity problem, any minorant will require the Type II / II case that we are unable to handle, so sieve minorants do not help us. However, since the factor $3$ in \eqref{alpha+1} is larger than $2$, we do not run into the parity barrier and so have a chance of constructing a majorant that does not use the bad Type II / II case. On the other hand, in cases where one uses a sieve minorant, it usually is not a non negative function. For that reason it cannot be used in a straight forward manner in an additive context. If for example $\nu(n)\leq f(n)$ for any $n$, it may not necessarily hold that 
\begin{align*}
\sum_{n_1+n_2+n_3=x}\nu(n_1)\nu(n_2)\nu(n_3)\leq \sum_{n_1+n_2+n_3=x}f(n_1)f(n_2)f(n_3),
\end{align*}
if $\nu$ takes negative values. In \cite{ms} and \cite{te} the convolution with Bohr sets makes the functions set-like and in particular non negative, so that this problem can be circumvented. See also the author's paper \cite{gr19} for further development of this strategy.

\subsection{Outline and Notation}
We define for any arithmetic function $\omega: \mathbbm{N}\to \mathbbm{C}$
\begin{align*}
\Lambda^{\omega}(n)=\Lambda(n)\sum_{n=k^2+l^2}\omega(l),
\end{align*} so that as in the previous subsection $\Lambda^\Lambda$ becomes a weighted indicator function closely related to our prime set of interest. Throughout this paper we fix $\chi$ to be the nontrivial character to the modulus $4$ and set 
\begin{align}\label{Hdef}
H=2\prod_{p}\Bigl(1-\frac{\chi(p)}{(p-1)(p-\chi(p))} \Bigr).
\end{align}
By (1.5) of \cite{foi} we have for any fixed $A>0$ the asymptotics
\begin{align}\label{foiasymp}
\sum_{n\leq x}\Lambda^\Lambda(n)= Hx+O_A(x(\log x)^{-A}).
\end{align}
Error terms of this form and related bounds for the length of arithmetic progressions appear on various occasions. We reserve the letter $A$ for them.

We want to apply the transference principle as used in \cite{mms}. For this we need, among other things, to know something about the distribution of $\Lambda^\Lambda$ in arithmetic progressions. This is achieved by proving a Siegel-Walfisz like theorem. 
The result is also needed at several other occasions, so we state it at this point. Related to $H$ and similar as in \cite{foi} there is the Euler factor
\begin{align*}
\psi(l):=\prod_{p\nmid l}\Bigl(1-\frac{\chi(p)}{p-1} \Bigr).
\end{align*}
The distribution in residue classes further depends on the local density
\begin{align*}
\varrho_l(q,a):=\bigl|\{k \mod (q)\text{ : }k^2+l^2\equiv a(q) \}\bigr|.
\end{align*}
By $\tau(n)$ we denote the number of divisors of $n$. With this notation we can state the following generalization of \eqref{foiasymp} and Theorem 1 in \cite{foi}.
\begin{theorem}
\label{GFI}
Let $A>0$, $|\omega(l)|\ll (\log 2l)^{c_1}\tau(l)^{c_2}$ for some $c_1,c_2$.  We have uniformly in $q\leq (\log x)^A$ and $a$ with $(a,q)=1$
\begin{align*}
\sum_{\substack{n\leq x \\n \equiv a (q)}}\Lambda^\omega(n)&=\frac{1}{\varphi(q)}\sum_{\substack{k^2+l^2\leq x\\}}\omega(l)\varrho_l(q,a)\psi(lq)+O_{A,c_1,c_2}(x(\log x)^{-A}).
\end{align*}
\end{theorem}
Besides the residue class condition, this result also generalises \cite{foi} by allowing more general types of weights. In our case the weight $\omega$ is either an indicator function closely related to the primes or some sieve (in which case it may depend on $x$). In both cases we expect uniformly in $T\leq \sqrt{x}$ and $q\leq (\log x)^A$ the asymptotic
\begin{align}
\sum_{\substack{l\leq T\\l\equiv a (q)}}\omega(l)\psi(ql)=\Theta(q,a)\psi'(q)\sum_{l\leq T}\omega(l)\psi(l)+O_A(\sqrt{x}(\log x)^{-A}), \label{1}
\end{align}
where
\begin{align*}
\psi'(q)=\prod_{p|q}\Bigl(1-\frac{\chi(p)}{p-1}\Bigr)^{-1}
\end{align*} 
and
\begin{align}
\Theta(q,a)=\begin{cases}\frac{1}{\varphi(q)}, & \text{ if } (q,a)=1 \\
0, & \text{ else.} \end{cases}\label{2}
\end{align}
Setting for $(a,q)=1$
\begin{align}\label{xidef}
\Xi(q,a)=\frac{\psi'(q)}{\varphi(q)}\sum_{c(q)^*}\varrho_c(q,a)
\end{align}
and $\Xi(q,a)=0$ else, we infer the following special case of Theorem \ref{GFI}. It is deduced after the proof of Theorem \ref{GFI} in subsection \ref{corosubs}.

\begin{cor}\label{cor} 
Let $\omega(l)$ and $A$ be as required in Theorem \ref{GFI}. If further $\omega(l)$ obeys \eqref{1}, we have uniformly in $q\leq (\log x)^A$ and $a \mod (q)$ that 
\begin{align*}
\sum_{\substack{n\leq x \\n \equiv a (q)}}\Lambda^\omega(n)
=\frac{\Xi(q,a)}{\varphi(q)}\sum_{k^2+l^2\leq x}\omega(l)\psi(l)+O_A(x(\log x)^{-A}).
\end{align*}
Here $\Xi(q,a)$ is given by \eqref{xidef}. $\Xi(q,a)$ is multiplicative in $q$, $q$-periodic in $a$, and for odd primes $p$ and $r\in \mathbbm{N}$ we have
\begin{align}\label{Xipowerblindeq}
\Xi(p^r,a)=\Xi(p,a).
\end{align}
Furthermore it holds that
\begin{align*}
\Xi(2,1)&=1\\
\Xi(4,1)&=2\\
\Xi(4,3)&=0,
\end{align*}
and for $r\in \mathbbm{N}_{\geq 2}$
\begin{align*}
\Xi(2^r,a)=\Xi(4,a).
\end{align*}
\end{cor}
The function $\Xi(q,a)$ encodes biases of Fouvry-Iwaniec primes to certain residue classes. Note that $\Xi(q,a)$ can be quite far away from $1$. In fact there are arbitrary large $q$ such that
\begin{align*}
\Xi(q,a_1)&\ll (\log \log q)^{-1} \\
\Xi(q,a_2)&\gg \log \log q
\end{align*}
for suitable $a_1$ and $a_2$. This can be seen by considering that 
\begin{align*}
\sum_{c(p)}\rho_c(p,a)
\end{align*}
is for a fixed prime $p$ constant for all $a$ with $p\nmid a$. By removing the contribution of $c\equiv 0(p)$ to this sum, a short calculation shows the following. If $a$ is a square modulo $p$, then
\begin{align*}
\Xi(p,a)=1-\frac{1}{p+O(1)}
\end{align*}
and if it is not, then
\begin{align*}
\Xi(p,a)=1+\frac{1}{p+O(1)}.
\end{align*}
By choosing $q=\prod_{p\leq z} p$ and $a_1$ such that it is not a square modulo any of these primes, we get
\begin{align*}
\Xi(q,a_1)&=\prod_{p\leq z}\bigl(1-\frac{1}{p+O(1)}\bigr)\\
&\ll (\log z)^{-1}\\
&\ll (\log \log q)^{-1}.
\end{align*}
Similarly, we can construct $a_2$ by choosing it to be a square residue modulo primes up to $z$.

To mask local effects to small moduli Green (see \cite{gre}) uses a device that is now mostly called $W$-trick. Instead of considering for example the von Mangoldt function $\Lambda(n)$ it is useful to look at
\begin{align*}
\Lambda(Wn+b),
\end{align*}
where $W$ is the product of small primes and $b$ a primitive residue class to that modulus. We need to handle different $W$-tricked functions related to our problem, so we fix the following notation. Assume we are given parameters $x$, $W$, a residue class $b\in [W]=\{1,\ldots, W\}$, and $f:[x]\to \mathbbm{C}$. We set $N=\floor{x/W}$ and define $f_{W,b}:[N]\to \mathbbm{C}$, the normalised $W$-tricked function, by
\begin{align}\label{normal}
f_{W,b}(n):=\frac{ \varphi(W)}{\Xi(W,b)W H}f(Wn+b).
\end{align}
The normalisation is chosen in this way to ensure that we have for admissible $b$
\begin{align*}
\sum_{n\leq N}\Lambda^{\Lambda}_{W,b}(n)= N\bigl(1+o(1)\bigr).
\end{align*}
Using this notation we can state the transference result. It gives us three conditions that are sufficient for proving Theorem \ref{Goal}. The statement and its proof are based on Theorem 2.1 of \cite{mms}. 
\begin{theorem} \label{MT}
Let $x\equiv 3(4)$ be a large integer. Let $\Lambda^+$ be a function with $\Lambda^{\Lambda}(n)\leq \Lambda^+(n)$ for every $n\in [x]$. Let $\alpha^+>0.1$ and $\alpha^-,\eta>0$. Let $W=2\prod_{p\leq w}p$, where $w=0.1 \log \log x$ and let $N=\floor{x/W}.$ Assume that the following conditions hold for every residue class $b (W)$ with $(b,W)=1$, $b\equiv 1(4)$:
\begin{enumerate}[I.]
\item (Pseudorandomness of the majorant) For all $\gamma\in \mathbbm{R}$ we have
\begin{align*}
\Bigl|\sum_{n\leq N}\Lambda^+_{W,b}(n)e(\gamma n)-\alpha^+ \sum_{n\leq N}e(\gamma n)|\leq \alpha^+ \eta N. 
\end{align*}
\item ($L^q$ estimate) For some $2<q<3$ we have 
\begin{align*}
\int_0^1 \bigl|\sum_{n\leq N} \Lambda^{\Lambda}_{W,b}(n)e(\gamma n)\bigr |^q d\gamma\ll N^{q-1}, 
\end{align*}
where the implied constant does not depend on $W$ or $b$.
\item (Density in progressions) For each arithmetic progression $P\subset [N]$ with $|P|\geq \eta N$ we have \begin{align*}
\sum_{n\in P}\Lambda^{\Lambda}_{W,b}(n)\geq \alpha^{-}|P|.
\end{align*}
\end{enumerate}
If $\alpha^+<3\alpha^-$ and $\eta$ is small enough in terms of $3\alpha^--\alpha^+$, then $x$ can be written as the sum of three Fouvry-Iwaniec primes.
\end{theorem}
In our case $\alpha^-$ is arbitrary close to one and $\alpha^+=\alpha^+(x)$ is a function of $x$ that fulfills
\begin{align*}
\alpha^+(x)<3-\delta+o(1)
\end{align*}
for some fixed $\delta>0$. 

The structure of the paper is as follows. In the next section we prove Theorem \ref{MT}, Theorem \ref{GFI} and preliminary results for the majorant. Theorem \ref{MT} is proved by using the transference principle. We do not go in the details, like dissecting the related functions into uniform and antiuniform parts, but instead make use of a proposition in \cite{mms} that can readily be applied and incorporates much of the necessary work. This also means that we do not deal with the combinatorical work that gives the crucial density ratio of $3$. The proof of Theorem \ref{GFI} requires only a slight extension of the results in \cite{foi}. Afterwards we apply the result to confirm that condition III. of Theorem \ref{MT} holds for $\alpha^-$ close to $1$. The preliminary results for the majorant in subsection \ref{secbuch} are of two kinds. First we state a strong asymptotics for rough numbers. This problem was first considered by Buchstab and we employ a result of De Bruijn \cite{deB} to gain a better error term. Further we state the sieve we intend to use and importing the required results.

In section \ref{sec3} we construct the majorant $\Lambda^+$. The goal is to show that the normalised $W$-tricked version of it fulfills condition I. of Theorem  \ref{MT} for some a suitable $\alpha^+<3-\delta$.
\begin{prop} \label{prmajo} Let $x$, $N$, $W$, $b$, $\eta$ as in Theorem \ref{MT}. Let further $\Lambda^+$ as in \eqref{L+def}. There exists a function $\alpha^+(x)$ with $\alpha^+(x)<2.9739$ for all sufficiently large $x$, such that the following statement holds. We have for all $N>N_0(\eta)$ and all $\gamma \in \mathbbm{R}$ 
\begin{align*}
|\sum_{n\leq N}\Lambda^+_{W,b}(n,x)e(\gamma n)-\alpha^+(x)\sum_{n\leq N}e(\gamma n)|\leq \alpha^+(x) \eta N.
\end{align*}
\end{prop}
The construction of the majorant is done by the use of sieves and the switching principle as described in the introduction. Furthermore we need to check that the majorant is not overestimating by too much, see \eqref{alpha+1}. These results are stated in Lemma \ref{majolem} and proved thereafter. We conclude section 3 by applying the combinatorical decomposition of \cite{dfi} on $\Lambda^+$ to split it into Type I and II sums with suitable ranges.

In section \ref{sec4} we show that $\Lambda^+_{W,b}$ is pseudorandom for $\gamma$ lying in the major arcs, a partial result towards Proposition \ref{prmajo}. Let $A_\mathfrak{M}=10^4$. For a given $x$ we set
\begin{align}\label{Mqadef}
\mathfrak{M}(q,a)=\{\gamma\in[0,1]: |\gamma-a/q|\leq (\log x)^{A_\mathfrak{M}}x^{-1} \}.
\end{align}
The set of major arcs is given by
\begin{align}\label{Mganzdef}
\mathfrak{M}=\bigcup_{q\leq (\log x)^{A_\mathfrak{M}}}\bigcup_{\substack{a(q)^*}}\mathfrak{M}(q,a).
\end{align}
The size of $A_\mathfrak{M}$ plays no role for the major arc analysis, as long as it is fixed. For the minor arcs, however, we require it to be sufficiently large and above choice is admissible. We consider the major arcs by using Lemma \ref{majolem} together with summation and integration by parts.

In section \ref{sec5}, that can be seen as the core piece for the proof of Theorem \ref{Goal}, we prove minor arc pseudorandomness for the majorant. The minor arcs are given as the complement of the major arcs 
\begin{align*}
\mathfrak{m}=[0,1]\setminus \mathfrak{M}.
\end{align*}
We consider separately Type I and II sums. The harder Type II case is handled as described in the introduction. This reduces the minor arc Type II bound to well known sums over certain minima, see for example chapter 13 of \cite{ik}. The Type I bounds cause no problem and together with the results of the previous section we deduce Proposition \ref{prmajo}.

In section \ref{sec6} we use the major and minor arc bounds together with a different majorant to show the $L^q$ restriction estimate. This is condition II. of Theorem \ref{MT} and the final ingredient. We use ideas of Bourgain \cite{bou} to prove the restriction estimate.

We conclude the paper by combining the results to prove Theorem \ref{Goal} and Theorem \ref{3apcor}.

\section{Preliminary Results}
In this section we show preliminary results. We start by proving Theorems \ref{MT} and \ref{GFI} that were stated in the previous subsection. Afterwards we show results that are needed for constructing and understanding the majorant. 
\subsection{The Transference Result}
The proof of Theorem \ref{MT} is based on the following result that incorporates the whole transference concept.
\begin{lemma}\label{transfer}
Let $\epsilon,\eta\in(0,1)$. Let N be a positive integer and let $f_1,f_2,f_3:[N]\to \mathds{R}_{\geq 0}$ be functions, with each $f\in \{f_1,f_2,f_3\}$ satisfying the following assumptions:
\begin{enumerate}[i.]
\item There exists a majorant $\nu:[N]\to \mathds{R}_{\geq 0}$ with $f\leq \nu$ pointwise, such that for all $\gamma\in \mathbbm{R}$ we have
\begin{align*}
|\sum_{n\leq N}\nu(n)e(\gamma n)-\sum_{n\leq N}e(\gamma n)|\leq \eta N.
\end{align*}
\item We have 
\begin{align*}
\int_0^1 \bigl|\sum_{n\leq N} f(n)e(\gamma n)\bigr |^q d\gamma\leq K^q N^{q-1}
\end{align*}
for some fixed $q$, $K$ with $K\geq 1$ and $2<q<3$.
\item For each arithmetic progression $P\subset [N]$ with $|P|\geq \eta N$ we have 
\begin{align*}
|P|^{-1}\sum_{n\in P}f(n)\geq 1/3+\epsilon.
\end{align*}
\end{enumerate}
Then for each $n\in [N/2,N]$ we have \begin{align*}
\sum_{n_1+n_2+n_3=n}f_1(n_1)f_2(n_2)f_3(n_3)\geq (c(\epsilon)-O_{\epsilon,K,q}(\eta))N^2,
\end{align*}
where $c(\epsilon)>0$ depends only on $\epsilon$.
\end{lemma}
\begin{proof}
This is Proposition 3.1. in \cite{mms}.
\end{proof}
To prove Theorem \ref{MT} we follow closely the proof of \cite[ Theorem 2.1]{mms}. 
 \begin{proof}[Proof of Theorem \ref{MT}]
Let $x$, $W$, $N$, $\Lambda^+$, $\alpha^-$, and $\alpha^+$ be as stated in Theorem \ref{MT}. Choose $b_1$, $b_2$, $b_3\in [W]$, such that
\begin{align*}
b_1+b_2+b_3&\equiv x (W)\\
(b_i,W)&=1\\
b_i&\equiv 1(4),
\end{align*}
which can always be done be the Chinese Remainder Theorem. Consider now the functions
\begin{align*}
f_i(n)&=\frac{\Lambda^{\Lambda}_{W,b_i}(n)}{\alpha^+}
\end{align*}
and
\begin{align*} 
\nu_i(n)&=\frac{\Lambda^+_{W,b_i}(n)}{\alpha^+}.
\end{align*}
As $\Lambda^{\Lambda}\leq \Lambda^+$ we have $f_i\leq \nu_i$. Condition i. for $\nu_i$ is now an immediate consequence of condition I. for $\Lambda^+$. Similarly ii. of the lemma follows from II in the theorem, since $\alpha^+$ is bounded away from $0$. Let now $P$ be as in iii., we then have by III. 
\begin{align*}
\sum_{n\in P}f_i(n)\geq \frac{\alpha^{-}}{\alpha^+}|P|.
\end{align*}
As $\alpha^+<3 \alpha^{-}$ there is an $\epsilon>0$ such that $\alpha^-/\alpha^+>1/3+\epsilon$. 

Let now $n=(x-b_1-b_2-b_3)/W$. Since $n\in [N-3,N]$, we can apply Lemma \ref{transfer} and get 
\begin{align*}
\sum_{n_1+n_2+n_3=n}f_1(n_1)f_2(n_2)f_3(n_3)\geq (c(\epsilon)-O_{\epsilon,K,q}(\eta))N^2.
\end{align*}
The contribution of higher prime powers in either von Mangoldt function to
\begin{align*}
\sum_{n_1+n_2+n_3=n}f_1(n_1)f_2(n_2)f_3(n_3)
\end{align*}
is $O(N^{7/4+\epsilon})$. Recall that $\alpha^+$ is bounded away from $0$. So, if $\eta$ is small enough, for sufficiently large $n$, there are $n_i$ with $n_1+n_2+n_3=n$ such that there are Fouvry-Iwaniec primes $p_i$ with $p_i=Wn_i+b_i$. We conclude
\begin{align*}
x&=Wn+b_1+b_2+b_3\\
&=p_1+p_2+p_3.
\end{align*}
\end{proof}

\subsection{Generalised Fouvry-Iwaniec}
\begin{proof}[Proof of Theorem \ref{GFI}]
We now prove the proposed generalization of the result of Fouvry and Iwaniec. We start by closely following section 7 of \cite{foi} until we identify the main term and are left with estimating linear and bilinear remainder terms. We then show how the linear remainder can be sufficiently bounded by an application of \cite[Lemma 5']{foi}. The bilinear remainder requires only slight modification of the original argument. 

Let $\omega$, $q$ and $a$ be as required in Theorem \ref{GFI}. We write for $n\equiv a(q)$
\begin{align*}
a_n=\sum_{\substack{n=k^2+l^2}}\omega(l)
\end{align*}
and and $a_n=0$ else. We further set
\begin{align*}
P(x)=\sum_{n\leq x}\Lambda(n)a_n,
\end{align*}
which we want to evaluate. In the same way as in the original paper this is done by an application of Vaughan's identity. A central role, see \cite[(7.2)]{foi}, is played by
\begin{align}\label{Acongsum1}
A_d(x,q)=\sum_{\substack{n\equiv 0 (d)\\n\leq x}}a_n. 
\end{align}
Let $z$ and $y$ be parameters. As in \cite[(7.5)]{foi} we have
\begin{align}\label{gfie1}
P(x)=A(x,q;y,z)+B(x,q;y,z)+P(z),
\end{align}
where, as in \cite[(7.6)]{foi} and \cite[(7.8)]{foi},
\begin{align*}
A(x,q;y,z)&=\sum_{b\leq y}\Bigl(A'_b(x,q)-A_b(x,q)\log x-\sum_{c\leq z}\Lambda(c)A_{bc}(x,q) \Bigr)\\
|B(x,q;y,z)|&\leq \sum_{z<d<x/y} \log d \Bigl|\sum_{y<b\leq x/d}\mu(b)a_{bd} \Bigr|.
\end{align*}
Here $A'_b(x)$ denotes the sum in \eqref{Acongsum1} with $a'_n=a_n \log n$. Using \cite[(3.13')]{foi} instead of \cite[(3.13)]{foi} we expect  $A_d(x)$ to be approximated by main terms of the form
\begin{align}\label{mapprox}
M_d(x,q)=\frac{1}{dq}\sum_{n\leq x}\sum_{k^2+l^2=n}\omega(l)\varrho_l(d;q,a).
\end{align}
Here $\varrho_l(d;q,a)$ is the number of solutions to the system of congruences
\begin{align*}
\nu^2+l^2&\equiv 0 (d)\\
\nu^2+l^2&\equiv a (q).
\end{align*}
We define
\begin{align*}
\varrho_k(q,a)&=\#\{m \mod (q):m^2+k^2\equiv a (q) \}\\
\varrho_k(q)&=\varrho_k(q,0)\\
\end{align*}
and note that either $(d,q)=1$ or $M_d(q,a)$ and $A_d(q,a)$ both vanish. In the former case we have, just as in \cite[(3.19)]{foi},
\begin{align*}
\varrho_l(d;q,a)=\varrho_l(d)\varrho_l(q,a).
\end{align*}
The error of the approximation is defined as
\begin{align*}
R_d(x,q)=A_d(x,q)-M_d(x,q),
\end{align*}
which gives us the linear remainder
\begin{align*}
R(x,q;D)=\sum_{d\leq D}|R_d(x,q)|.
\end{align*}
We continue as in \cite[(7.13)]{foi} and get
\begin{align}\label{gfie2}
A(x,q;y,z)=M(x,q;y,z)+R(x,q;y,z),
\end{align}
where
\begin{align*}
&M(x,q;y,z)=\\
&\frac{1}{q}\sum_{n\leq x}\sum_{\substack{b\leq y\\ (b,q)=1}}\frac{\mu(b)}{b}\Bigl[(\log \frac{n}{b})\sum_{n=k^2+l^2}\omega(l)\varrho_l(b)\varrho_l(q,a)-\sum_{\substack{c\leq z\\(c,q)=1}}\frac{\Lambda(c)}{c}\sum_{n=k^2+l^2}\omega(l)\varrho_l(bc)\varrho_l(q,a)\Bigr],
\end{align*}
and $R(x,q;y,z)$ is some function that can be estimated with the help of the linear remainder function by
\begin{align*}
|R(x,q;y,z)|\leq R(x,q;yz)\log x+\int_1^x\frac{R(t,q;y)}{t}dt.
\end{align*}
The next step is the evaluation of the main term and here the effect of $q$ and $a$ becomes visible. We have 
\begin{align*}
M(x,q;y,z)=\frac{1}{q}\sum_{n\leq x}\sum_{k^2+l^2=n}\omega(l)\varrho_l(q,a)\sigma_l(n,q;y,z)
\end{align*}
with
\begin{align*}
\sigma_l(n,q;y,z)=\sum_{\substack{b\leq y\\ (b,q)=1}}\frac{\mu(b)}{b}\Bigl[\varrho_l(b)\log\frac{n}{b}-\sum_{\substack{c\leq z\\ (c,q)=1}}\frac{\Lambda(c)}{c}\varrho_l(bc) \Bigr].
\end{align*}
The completion of this to an infinite series can be done as in \cite{foi} and using their bound we get
\begin{align*}
\sigma_l(n,q;y,z)=\psi(l,q)+O_A(\tau(q)\tau(l)\log(2lnz)(\log y)^{-A}).
\end{align*}
Here $\psi(l,q)$ differs slightly from \cite[(7.19)]{foi}, we instead have (note that all appearing infinite sums converge)
\begin{align*}
\psi(l,q)&=-\sum_{\substack{b\\(b,q)=1}}\frac{\mu(b)}{b}\varrho_l(b) \log b\\
&=-\sum_{d|q}\mu(d) \sum_{\substack{b \\d|b}}\frac{\mu(b)}{b}\varrho_l(b) \log b \\
&=-\sum_{d|q}\frac{\mu^2(d)\varrho_l(d)}{d}\sum_{\substack{b\\(b,q)=1}}\frac{\mu(b)}{b}\varrho_l(b) \log b-\sum_{d|q}\frac{\mu^2(d)\varrho_l(d)\log d}{d}\sum_{\substack{b\\(b,q)=1}}\frac{\mu(b)}{b}\varrho_l(b) \\
&=\sum_{d|q}\frac{\mu^2(d)\varrho_l(d)}{d}\psi(l,d)+\sum_{d|q}\frac{\varrho_l(d)\log d}{d}\psi'(l,d),
\end{align*}
say. We have further
\begin{align*}
\psi'(l,d)&=\sum_{d'|d}\mu(d')\sum_{\substack{b \\d'|b}}\frac{\mu(b)}{b}\varrho_l(b)\\
&=\sum_{d'|d}\frac{\mu^2(d')\varrho_l(d')}{d}\psi'(l,d')
\end{align*}
and so
\begin{align*}
\psi'(l,d)&=\psi'(l,1)\prod_{p|d'}\Bigl(1-\frac{\varrho_l(p)}{p}\Bigr)^{-1}\\
&=0.
\end{align*}
Thus, 
\begin{align*}
\psi(l,q)=\sum_{d|q}\frac{\mu^2(d)\varrho_l(d)}{d}\psi(l,d)
\end{align*}
and similarly to the case of $\psi'$ we have
\begin{align*}
\psi(l,q)&=\psi(l,1)\prod_{p|q}\Bigl(1-\frac{\varrho_l(p)}{p}\Bigr)^{-1}\\
&=\psi(l,1)\prod_{p|q}\Bigl(1-\frac{1}{p}\Bigr)^{-1} \prod_{p|q}\Bigl(1-\frac{1}{p}\Bigr)\Bigl(1-\frac{\varrho_l(p)}{p}\Bigr)^{-1} \\
&=\frac{q}{\varphi(q)}\prod_{p\nmid l}\Bigl(1-\frac{\chi(p)}{p-1} \Bigr)\prod_{\substack{p|q\\ p\nmid l}}\Bigl(1-\frac{\chi(p)}{p-1} \Bigr)^{-1}\\
&=\frac{q}{\varphi(q)}\psi(ql).
\end{align*}
Here the one variable $\psi(q)=\psi(q,1)$ coincides with the one in \cite[(7.19)]{foi}. This enables to identify the proposed main term by 
\begin{align*}
M(x,q;y,z)=\frac{1}{\varphi(q)}\sum_{k^2+l^2\leq x}\omega(l)\varrho_l(q,a)\psi(ql)+O_A\bigl(x(\log x)^{-A+O_{c_1,c_2}(1)}\bigr),
\end{align*}
if $y>x^\epsilon$ for a fixed $\epsilon>0$. Here the implied constants are independent of $q$ and $a$. We note that $P(z)\leq x^{1-\epsilon}$ as long as $z\leq x^{1-2\epsilon}$. Considering \eqref{gfie1} and \eqref{gfie2}, to complete the proof we have to show that the two remainder terms $R(x,q;y,z)$ and $B(x,q;y,z)$ are small enough.

The bound for $R(x,q;y,z)$ follows essentially from Lemma 5' in \cite{foi}, but the implied constant there depends on $q$, which makes it not suitable for our application. This is a minor problem as Lemma 3' in \cite{foi} still has explicit dependence on $q$. We get for $(q,a)=1$ and $1\leq D\leq x$ the following variant of \cite[(3.16')]{foi}
\begin{align*}
\sum_{d\leq D}|R_d(f;q,a)|\ll_\epsilon q^3 ||\omega||\Delta D^{1/2}x^{5/4+\epsilon}.
\end{align*}
This leads to
\begin{align}\label{linrem}
R(x,q;D)\ll_\epsilon q^3||\omega|| D^{1/4}x^{1/2+\epsilon}.
\end{align}
We conclude that for $yz\leq x^{1-2\epsilon}$ we have
\begin{align*}
R(x,q;y,z)\ll_\epsilon x^{1-\epsilon/5},
\end{align*}
uniformly in $q\ll (\log x)^A$.

For estimating $B(x,q;y,z)$ we note that the additional congruence condition in the sequence $a_n$ and the fact that $\omega$ is not bounded are the only differences to the case considered in \ref{foi}. We open the congruence condition with Dirichlet characters and estimate estimate
\begin{align*}
B(x,q;y,z)\leq (\log x)\sup_{\chi(q)}\sum_{z<d<x/y}\Bigl|\sum_{y<b\leq x/d}\mu(b)\chi(b) \sum_{bd=k^2+l^2}\omega(l) \Bigr|.
\end{align*}
The steps in \cite{foi} section 8 and 9 can be followed almost verbatim, $\mu \chi$ taking the role of $\mu$ and, as there, $\lambda=\omega$. The increased range of $\omega$ loses us finitely many logarithms in \cite[Corollary 12]{foi} and so we get sufficient saving if $j$ is large and as long as
\begin{align*}
\alpha(z)=\mu(r|z|^2)\chi(r|z|^2)
\end{align*}
still fulfills Hypothesis \cite[(9.28)]{foi}. That is
\begin{align*}
D_d(\alpha)\ll ||\alpha||^2 (\log A')^{-j}
\end{align*}
for (see \cite[(9.23)]{foi})
\begin{align*}
D_d(\alpha)=2 \pi {A'}^{-2}\sum_{z_1\equiv z_2 (d)}\alpha(z_1)\overline{\alpha}(z_2)\text{exp}(-2\pi |z_1-z_2|)  {A'}^{-1}.
\end{align*}
As the final choice of $y$ means $y>x^\epsilon$ for some fixed $\epsilon>0$, we may assume that $q\leq (\log A')^j$. Consequently the additional character only requires a slight extension of the range of Grossencharacters considered and the proof completes as in section 9 \cite{foi}.
\end{proof}

The distribution in congruence classes of $\Lambda^\omega(n)$ to the modulus $q$ depends on the distribution of $\omega(l)$ to the same modulus. Theorem \ref{GFI} encodes this dependence in the main term. We intend to apply the result only for weights $\omega$ that are essentially supported and equidistributed on residue classes coprime to $q$. In that case Corollary \ref{cor} gives a more explicit statement. We now show how it follows from Theorem \ref{GFI}.
\begin{proof}[Proof of Corollary \ref{cor}]
The main term in Theorem \ref{GFI} is of the form
\begin{align*}
\sum_{m^2+l^2\leq x}\omega(l)\varrho_l(q,a)\psi(ql).
\end{align*}
To simplify notation we again write
\begin{align*}
\psi'(q)=\prod_{p|q}\Bigl(1-\frac{\chi(p)}{p-1} \Bigr)^{-1}.
\end{align*}
We recall the assumed asymptotics 
\begin{align*}
\sum_{\substack{l\leq T\\l\equiv b(q)}}\omega(l)\psi(ql)=\Theta(q,b)\psi'(q)\sum_{l\leq T}\omega(l)\psi(l)+O_A(\sqrt{x}(\log x)^{-A})
\end{align*}
with
\begin{align*}
\Theta(q,a)=\begin{cases}\frac{1}{\varphi(q)}, & \text{ if } (q,a)=1 \\
0, & \text{ else.} \end{cases}
\end{align*}
By sorting into residue classes and the definition of $\Xi$ in \eqref{xidef}, it follows that
\begin{align*}
\sum_{k^2+l^2\leq x}\omega(l)\varrho_l(q,a)\psi(ql)&=\frac{\psi'(q)}{\varphi(q)}\sum_{b(q)^*}\varrho_b(q,a)\sum_{k^2+l^2\leq x}\omega(l)\psi(l)+O_A(x(\log x)^{-A})\\
&=\Xi(q,a)\sum_{k^2+l^2\leq x}\omega(l)\psi(l)+O_A(x(\log x)^{-A}).
\end{align*}
The multiplicativity of $\Xi$ is clear. Let now $q=p^r$ with $p\neq 2$. By Hensel's Lemma we have
\begin{align*}
{\sum_{d(p^r)^*}}\varrho_d(p^r,a)&=\sum_{\substack{d(p^r)^*\\ }}\varrho_d(p,a).
\end{align*}
To prove \eqref{Xipowerblindeq}, we note that
\begin{align*}
\Xi(p^r,a)&=\Bigl(1-\frac{\chi(p)}{p-1}\Bigr)^{-1}\sum_{d(p^r)^*}\frac{\varrho_d(p^r,a)}{\varphi(p^r)}\\
&=\Bigl(1-\frac{\chi(p)}{p-1}\Bigr)^{-1}\sum_{d(p^r)^*}\frac{\varrho_d(p^r,a)}{p^{r-1}\varphi(p)}\\
&=\Bigl(1-\frac{\chi(p)}{p-1}\Bigr)^{-1}\sum_{d(p)^*}\frac{\varrho_d(p,a)}{\varphi(p)}\\
&=\Xi(p,a),
\end{align*}
as stated in the Corollary. For $q=2^r$ the proposed statements follow from first calculating $\Xi(2^r,b)$ for $r\leq 3$ by hand and then using Hensel's Lemma.
\end{proof}

We end this subsection by concluding that a suitable version of condition III. in Theorem \ref{MT} holds for $\alpha^-$ close to $1$.
\label{corosubs}
\begin{cor}\label{condc}
Let $x$, $N$, $W$, $b$, and $\eta$ as in Theorem \ref{MT}. Let further  $P\subset [N]$ be an arithmetic progression with $|P|\geq \eta N$. If $x$ is sufficiently large in terms of $\eta$ it holds that
\begin{align*}
\sum_{n\in P}\Lambda^{\Lambda}_{W,b}(n)\geq 0.999 |P|.
\end{align*}
\end{cor}
\begin{proof}
Let $P=\{kd+a: k\in \{1,\ldots |P|\}\}$. Since $|P|\geq \eta N$, we have $d\leq \eta^{-1}$. Further, 
\begin{align*}
\sum_{n\in P}\Lambda^{\Lambda}_{W,b}(n)&=\frac{\varphi(W)}{\Xi(W,b)WH}\sum_{n\in P}\Lambda^\Lambda(Wn+b)\\
&=\frac{\varphi(W)}{\Xi(W,b)WH}\sum_{k\leq |P|}\Lambda^\Lambda(dWk+Wa+b).
\end{align*}
By the Siegel-Walfisz Theorem and since $x$ is chosen large in terms of $\eta$, we can apply Corollary \ref{cor} to get
\begin{align*}
&\frac{\varphi(W)}{\Xi(W,b)WH}\sum_{k\leq |P|}\Lambda^\Lambda(dWk+Wa+b)\\
=&\frac{\varphi(W)}{\Xi(W,b)W}\frac{\Xi(Wd,Wa+b)}{\varphi(Wd)}dW|P|+O_A\Bigl(\frac{x}{(\log x)^A}\Bigr)
\end{align*}
Again if $x$ is sufficiently large in terms of $\eta$, we have that every prime divisor of $d$ also divides $W$. Consequently $Wa+b$ is coprime to $Wd$, $Wa+b\equiv 1(4)$, and $\varphi(Wd)=\varphi(W)d$. This gives us
\begin{align*}
=&\frac{\varphi(W)}{\Xi(W,b)W}\frac{\Xi(Wd,Wa+b)}{\varphi(Wd)}dW|P|+O_A\Bigl(\frac{x}{(\log x)^A}\Bigr)\\
=&\frac{\varphi(W)}{\Xi(W,b)W}\frac{\Xi(W,b)}{\varphi(W)d}dW|P|+O_A\Bigl(\frac{x}{(\log x)^A}\Bigr)\\
=&|P|+O_A\Bigl(\frac{x}{(\log x)^A}\Bigr)\\
\geq& 0.999 |P|,
\end{align*}
as long as $x$ is large enough in terms of the implied constant.
\end{proof}

\subsection{Rough Numbers and a Sieve}\label{secbuch}
We define for real numbers $n$ and $z$
\begin{align}\label{rhodef}
\rho(n,z,z_0)=\begin{cases}1, &\text{ if } n\in \mathbbm{N}\text{ and } p|n \Rightarrow p>z \text{ or } p\leq z_0    \\
0, &\text{ else. }
\end{cases}
\end{align}
and write $\rho(n,z)=\rho(n,z,0)$. Our application of sieve switching is done with the help of Buchstab's identity. Let $z<w$, then for square-free $n$ we have
\begin{align}\label{buchidentity}
\rho(n,z)=\rho(n,w)+\sum_{z<p\leq w}\rho(n/p,p).
\end{align}
Summing over $\rho(n,z)$ is related to Buchstab's function that, to avoid confusion with the appearing weights, we call $\mathcal{B}(u)$. We define it by
\begin{align*}
\mathcal{B}(u)=\begin{cases}0 &\text{ if }u < 1 \\
u^{-1}, & \text{ if }1\leq u \leq 2,
\end{cases}
\end{align*}
and by induction for $k\leq u\leq k+1$ with $k\in \mathbbm{N}_{\geq 2}$
\begin{align*}
\mathcal{B}(u)=\frac{1}{u}\Bigl(k\mathcal{B}(k)+\int_k^u \mathcal{B}(v-1)dv\Bigr).
\end{align*}
This differs slightly from the usual definition in that we allow $u<1$. In this way we can state the following results in a more natural way. Note that for $u> 1$, $u\neq 2$ the function fulfills the delayed differential equation
\begin{align}\label{bdde}
(u\mathcal{B}(u))'=\mathcal{B}(u-1).
\end{align}
As a special case of Buchstab's work \cite{buc} we have the well known result
\begin{align*}
\sum_{n\leq T}\rho(n,T^{1/u})= \frac{u \mathcal{B}(u)T}{\log T}+O_u(T(\log T)^{-2}).
\end{align*}
To allow thick enough major arcs we require an asymptotic with better error. We set
\begin{align*}
B(t,z)=\frac{\mathcal{B}(\frac{\log t}{\log z})}{\log z}
\end{align*}
and have the following result that is related to Lemma 1.2 of \cite{kum}.
\begin{lemma}\label{buch}
Let $A>0$, $z>T^\delta$ for some $\delta>0$. Then it holds uniformly in $q\leq (\log x)^A$ and $(a,q)=1$ that
\begin{align*}
\sum_{\substack{n\leq T\\ n\equiv a(q)}}\rho(n,z)=\frac{1}{\varphi(q)}\int_0^T B(t,z)dt+O_{\delta,A}\bigl(T(\log T)^{-A}\bigr).
\end{align*}
Furthermore $B(t,z)$ is differentiable in $t$ outside the points $t=z^j$ for $j\in \{1,2\}$ and it holds that
\begin{align}\label{derivbound}
\bigl|\frac{\partial}{\partial t}B(t,z)\bigr|\leq \frac{1}{t (\log t)(\log z)}.
\end{align}
\end{lemma}
\begin{proof}
We start by considering the derivative
\begin{align*}
\frac{\partial}{\partial t}B(t,z)=\frac{\partial}{\partial t}\frac{\mathcal{B}(\frac{\log t}{\log z})}{\log z}.
\end{align*}
From the definition we see immediately that $\mathcal{B}(u)$ is smooth except if $u\in \mathbbm{N}$. Furthermore $\mathcal{B}(u)$ is differentiable at $u\in \mathbbm{N}_{\geq 3}$. We now prove the derivative bound \eqref{derivbound}. We note that \eqref{bdde} implies 
\begin{align*}
\mathcal{B}(u)'=\frac{\mathcal{B}(u-1)-\mathcal{B}(u)}{u}.
\end{align*}
So we get
\begin{align*}
\frac{\partial}{\partial t}B(t,z)&=\frac{\mathcal{B}'(\frac{\log t}{\log z})}{t (\log z)^2}\\
&=\frac{\mathcal{B}(\frac{\log t}{\log z}-1)-\mathcal{B}(\frac{\log t}{\log z})}{t (\log t) (\log z)}.
\end{align*}
The bound \eqref{derivbound} now follows from the fact that
\begin{align*}
0\leq \mathcal{B}(u)\leq 1
\end{align*}
for all $u$. 

The stated asymptotics follows from the Siegel-Walfisz Theorem together with finitely many applications of \eqref{buchidentity}, since not square-free $n$ are negligible. For a rigorous proof see the works of de Bruijn \cite[(1.13)]{deB} and Xuan \cite[Corollary 1]{xua}.
\end{proof}

A central role in the construction of our majorant for the Fouvry-Iwaniec primes is played by the so called beta-sieve. For a complete treatment of the sieves we refer to chapters 6.5, 6.8 and 11 of Friedlander and Iwaniec' book \cite{cri}. We call a function $\theta^+(n;D,P)$ an upper bound sieve of level $D$ and range $P$, if
\begin{align*}
\theta^+(n;D,P)=\sum_{d|n}\lambda_d
\end{align*}
for some coefficients $\lambda_d$ supported only on $d\leq D$, and
\begin{align}\label{ubsieve}
\theta^+(n;D,P)\geq 1_{(n,P)=1}(n).
\end{align}
Similarly $\theta^-(n;D,P)$ is a lower bound sieve if instead of \eqref{ubsieve} we have
\begin{align*}
\theta^-(n;D,P)\leq 1_{(n,P)=1}(n).
\end{align*}
When applying these sieves on a sequence $a_n$, the congruence sums and their approximations
\begin{align*}
\sum_{d|n}a_n=g(d)X+R_d
\end{align*}
appear naturally. We assume that $g(d)$ is multiplicative and that for some $\kappa\geq 0$ and $K\geq 1$ we have the bound
\begin{align}\label{sieveaxiom}
\prod_{w\leq p< z}(1-g(p))^{-1}\leq K\bigl(\frac{\log z}{\log w} \bigr)^\kappa
\end{align}
for any $2\leq w<z$. We call $\kappa$ the dimension of the sieve and are mainly concerned with the case $\kappa=1$.

We now construct upper and lower bound sieves that allow us to understand 
\begin{align*}
\sum_{n\equiv a(q)}a_n \theta^\pm(n),
\end{align*}
where $a_n$ is a sequence that fulfills the above sieve axiom with $\kappa=1$.  Without the congruence condition the optimal choice is the beta-sieve with $\beta=2$. To deal with the congruence condition we do an additional preliminary sieving. 

Recall that the beta-sieve for level $D$ and sifting range $P$ is defined as 
\begin{align*}
\sum_{\substack{d|n\\ \substack{d|P\\ d\in \mathcal{D}^\pm(D,\beta)}}}\mu(d)=\sum_{\substack{d|n\\ d|P}}\lambda_d^\pm(D,\beta),
\end{align*}
where 
\begin{align*}
\mathcal{D}^+(D,\beta)&=\{d=p_1\ldots p_n: p_1>\ldots>p_n, p_1\ldots p_mp_m^{\beta}<D \text{ for all } m \text{ odd} \}\\
\mathcal{D}^-(D,\beta)&=\{d=p_1\ldots p_n: p_1>\ldots>p_n, p_1\ldots p_mp_m^{\beta}<D \text{ for all } m \text{ even} \}.
\end{align*}
When applying the sieve, naturally 
\begin{align*}
V^\pm(D,P):=\sum_{d|P}\lambda^\pm_d(D,\beta) g(d)
\end{align*}
arises. We require two results as to how it is related to
\begin{align*}
V(P):=\prod_{p| P}\bigl(1-g(p)\bigr).
\end{align*}
The first is a fundamental lemma.
\begin{lemma}\label{fundlem}
Let $\lambda^\pm_d(D,\beta)$ be the $\beta$-sieve of level $D$ for $\beta=10$. Then for any multiplicative function $g(d)$ fulfilling \eqref{sieveaxiom} with $\kappa=1$ and for $z=\max\{p\in P\}$ we have
\begin{align*}
V^+(D,P)&=(1+E^+(s))V(P)\\
V^-(D,P)&=(1-E^-(s))V(P),
\end{align*}
where $s=\frac{\log D}{\log z}$and $E^\pm(s)$ are some functions depending on $g$, $D$, and $z$ with
\begin{align*}
0\leq E^\pm(s) \ll  e^{-s}.
\end{align*}
Here the implied constant depends only on $K$.
\end{lemma}
\begin{proof}
This is Lemma 6.8 of \cite{cri} in the case $\kappa=1$ together with the observation that
\begin{align*}
V^+(D,P(z))&\geq V(z)\\
V^-(D,P(z))&\leq V(z).
\end{align*}
\end{proof}
The second result we need is the main theorem of the beta-sieve.
\begin{lemma}\label{ls}
Let $g(d)$ fulfill \eqref{sieveaxiom} with $\kappa=1$. Let $\beta=2$. We have
\begin{align*}
V^+(D,z)&\leq \bigl(F(s)+O((\log D)^{-1/6}) \bigr)V(z) \text{ if } s\geq 1\\
V^-(D,z)&\geq \bigl(f(s)+O((\log D)^{-1/6}) \bigr)V(z) \text{ if } s\geq 2,
\end{align*}
where $s=\frac{\log D}{\log z}$, and $F(s)$, $f(s)$ are the functions of the linear sieve that are defined chapter 11.3 of \cite{cri}.
\end{lemma}
\begin{proof}
This is Theorem 11.12 in \cite{cri}.
\end{proof}
We now define the sieve that we use. Assume we are given $z_0$, $z$, $D_0$ and $D$ with
\begin{align*}
z_0&<z,\\
z_0&<D_0,\\
z&<D.
\end{align*}
To define the sifting range we write
\begin{align*}
P(z,z_0)&=\prod_{z_0<p\leq z}p,\\
P(z_0)&=P(z_0,0).
\end{align*}
We deal with small primes by a preliminary sieving with range $P(z_0)$ and level $D_0$. Set
\begin{align*}
\theta_I^\pm(n;D_0,P(z_0))=\sum_{\substack{d|n\\ d|P(z_0)}}\lambda_{I}^\pm(d),
\end{align*}
where $\lambda_I^\pm(d)$ are the upper / lower bound beta-sieve weights with $\beta=10$ and level $D_0$. The main sieving is done by the linear sieve. We set
\begin{align*}
\theta_{II}^\pm(n;D,P(z,z_0))=\sum_{\substack{d|n\\ d|P(z,z_0)}}\lambda_{II}^\pm(d),
\end{align*}
where $\lambda_{II}^\pm(d)$ are the upper/lower bound beta-sieve with $\beta=2$ and level $D$. Finally, we define the composed sieve by
\begin{align}\label{mainsieve}
\theta^\pm(n)=\theta^\pm(n;D,D_0,z,z_0)=\theta_I^+(n;D_0,P(z_0))\theta_{II}^\pm(n;D,P(z,z_0)).
\end{align}
Note that the first component is always an upper bound sieve.

\section{The Majorant}\label{sec3}
In this section we consider our majorant of choice for the Fouvry-Iwaniec primes that we call $\Lambda^+$. The two goals we have are proving Proposition \ref{prmajo} and achieving sufficiently little overcounting \eqref{alpha+1}. After the construction we show two types of results for $\Lambda^+$. The first is a a precise statement about the asymptotic behavior of $\Lambda^+$ in arithmetic progressions that is later employed to deal with the case $\gamma\in \mathfrak{M}$ of Proposition \ref{prmajo}. The second is a combinatorical dissection of $\Lambda^+$ into Type I and II sums that is crucial for the complementary case $\gamma \in \mathfrak{m}$  of Proposition \ref{prmajo}.

\subsection{Construction}\label{subsecconst}
We start by recalling that 
\begin{align*}
\Lambda^\Lambda(n)=\Lambda(n)\sum_{n=k^2+l^2}\Lambda(l).
\end{align*}
Throughout this section we assume we are given a fixed value of $x$ and construct a majorant $\Lambda^+(n,x)$ that depends on $x$ and in the range $n\leq x$ fulfills
\begin{align*}
\Lambda^\Lambda(n)\leq \Lambda^+(n,x).
\end{align*}
 The problem of constructing a majorant in this case is related to finding good upper bounds for the number of twin primes. To use the idea of switching sieves we follow the considerations of Fouvry and Grupp in \cite[Section V.]{fgs}, which in turn are based on weights appearing in Pan's paper \cite{pan}. 

We consider the inner occurrence of $\Lambda$ first. We have $l\leq \sqrt{x}$, which motivates the following notation. Let 
\begin{align*}
z&=x^{\xi/2}\\
z_1&=x^{\xi_1/2}\\
z_0&=e^{(\log x)^{1/3}}
\end{align*} with 
\begin{align*}
\xi&=0.265\\
\xi_1&=0.183.
\end{align*}
Recall $\rho(l,z)=\rho(l,z,z_0)\rho(l,z_0)$, as defined in \eqref{rhodef}. We have
\begin{align}\label{majoineq1}
\Lambda(l)&\leq (\log \sqrt{x})\rho(l,z)+E_1(l)\\
&= (\log \sqrt{x}) \rho(l,z,z_0)\rho(l,z_0)+E_1(l),
\end{align}
where $E_1(l)$ accounts for the contribution of primes less than $z$ and their powers. Essential for our application of the switching principle is the following inequality that is based on \cite[Eq. (15)]{pan}. For square free $l\leq \sqrt{x}$ it holds that
\begin{align}\label{panineq}
 \rho(l,z,z_0)\leq \rho(l,z_1,z_0)-\frac{1}{2}\sum_{\substack{z_1< p \leq z\\l=pm}}\rho(m,z_1,z_0)+\frac{1}{2}\sum_{\substack{z_1< p_1<p_2<p_3\leq z\\l=p_1p_2p_3m}}\rho(m,p_1,z_0).
\end{align}
To see this, we first note that both sides of \eqref{panineq} are $0$, if $l$ has a prime divisor in $(z_0,z_1]$. Write now $l=p_1 \ldots p_r R$ with $p_1,\ldots, p_r \in (z_1,z]$ and $\rho(R,z,z_0)=1$. For $r=0$ both sides of \eqref{panineq} are $1$. If $r\geq 1$, the left hand side is $0$ and we have to show that the right hand side is nonnegative. The values of the right hand side are 
\begin{align*}
1-1/2&=1/2, &\text { if } r=1,\\
1-2\times 1/2&=0, &\text { if } r=2,\\
1-3\times 1/2+1/2&=0, &\text { if } r=3,\\
1-4\times 1/2+3\times 1/2&=1/2, &\text { if } r=4,\\
1-5\times 1/2+6\times 1/2&=3/2, &\text { if } r=5.
\end{align*}
By our choice of $\xi_1$ larger values of $r$ are not possible. 

In the next step we apply the sieves defined previously. We now fix the following values
\begin{align*}
D_0&=e^{(\log x)^{2/3}}\\
\delta_0&=10^{-7}\\
D_1&=x^{1/3-\delta_0}.
\end{align*}
We recall \eqref{majoineq1}, and apply an upper bound sieve on $\rho(l,z_0)$ getting
\begin{align*}
\Lambda(l)\leq \theta_I^+(l,D_0,P(z_0))\rho(l,z,z_0)\log \sqrt{x}+E_1(l).
\end{align*}
Next we apply upper and lower bound sieves together with \eqref{panineq} to obtain
\begin{align*}
 \rho(l,z,z_0)\leq \theta_{II}^+(l,D_1,P(z,z_0))-\frac{1}{2}\sum_{\substack{z_1\leq p < z\\l=pm}}\theta_{II}^-(m,D_1/p,P(z_1,z_0))+\frac{1}{2}\sum_{\substack{z_1\leq p_1<p_2<p_3< z\\l=p_1p_2p_3m}}\rho(m,p_1,z_0).
\end{align*}
Note that $\theta_I^+(l)\geq 0$ and that $\theta_{I}^+(l)=\theta_{I}^+(m)$ if $m=(l,P(z_0))$. Thus, combining these two inequalities with the notation in \eqref{mainsieve} gives us
\begin{align*}
\Lambda(l)\leq& \bigl(\theta^+(l;D_1,D_0,z_1,z_0)-\frac{1}{2}\sum_{\substack{z_1\leq p < z\\l=pm}}\theta^-(m;D_1/p,D_0,z_1,z_0)\\
&+\frac{1}{2}\sum_{\substack{z_1\leq p_1<p_2<p_3< z\\l=p_1p_2p_3m}}\rho(m,p_1,z_0)\theta_I^+(m,D_0,P(z_0))\log \sqrt{x}+E_2(l) \\
=&\omega_1(l)+\omega_2(l)+\omega_3(l)+ E_2(l),
\end{align*}
say. Here $E_2(l)$ accounts for the error $E_1(l)$ and $l$ that have a square prime divisor larger than $z_1$. Furthermore we should keep in mind that by the choice of parameters, the weights $\omega_i(l)$ depend on $x$. Note that there is no main sieve component in $\omega_3$. Plugging in this upper bound we get
\begin{align*}
\Lambda^\Lambda(n)\leq \Lambda(n)\sum_{n=k^2+l^2}\bigl(\omega_1(l)+\omega_2(l)+\omega_3(l)+E_2(l)\bigr).
\end{align*}
To prove minor arc bounds for the part that is related to $\omega_3$ we apply a sieve on the outer von Mangoldt function. We have
\begin{align*}
\Lambda(n)&\leq \theta^+(n;x^{1/2-\delta_0},D_0,x^{1/2-2\delta_0},z_0) \log x +E_3(n)\\
&=\Omega(n,x)+E_3(n),
\end{align*}
say. Here $E_3(n)$ accounts for primes less than $x^{1/2-2\delta_0}$ and their powers.
We set 
\begin{align*}
E(n)=\Lambda(n)\sum_{n=k^2+l^2}E_2(l)+E_3(n)\sum_{n=k^2+l^2}\omega_3(l).
\end{align*}
The majorant is defined by
\begin{align}\label{L+def}
&\Lambda^+(n,x) \nonumber\\
=&\Lambda(n)\sum_{n=k^2+l^2}\omega_1(l)+\Lambda(n)\sum_{n=k^2+l^2}\omega_2(l)+\Omega(n,x)\sum_{n=k^2+l^2}\omega_3(l)+E(n)\\
=&\Lambda_1(n,x)+\Lambda_2(n,x)+\Lambda_3(n,x)+E(n),\nonumber
\end{align}
say. We have the bound
\begin{align*} 
\sum_{n\leq x}|E(n)|\ll x^{9/10},
\end{align*}
so that we can ignore this term for all our applications.

Since a natural number $l$ has at most $\log l$ prime divisors and the combinatorical sieve weights are bounded by the divisor function, we note the following bounds for the appearing weights
\begin{align}
|\omega_1(l)|&\ll \tau(l)\log x ,\label{o1bound} \\ 
|\omega_2(l)|&\ll \tau(l)\log l\log x, \label{o2bound} \\
|\omega_3(l)|&\ll (\log l)^3\tau(l)\log x.\label{o3bound}
\end{align}
In contrast to the work on twin primes, we do not require knowledge about the level of linear distribution of the weight $\omega_3$. This gives us the possibility to construct a tighter majorant by using a sharper inequality than \eqref{panineq}. However, this does not lead to new applications for us and so we opt to follow the approach for twin primes.

\subsection{The Majorant in arithmetic progressions}\label{L3subse}
We now take a look at the behavior of $\Lambda^+(n,x)$ in arithmetic progressions. This is used later to understand major arc exponential sums. We have two goals. On the one hand we show strong asymptotics that allow us to make the major arcs thick enough. On the other hand we define the function $\alpha^+(x)$ that is needed in Proposition \ref{prmajo} and prove a weak asymptotics that is sufficient for pseudorandomness in the case $\gamma=0$. 
\begin{lemma}\label{majolem}
For $i\in \{1,2,3\}$ there exist functions $C_i(x)$, $C_i'(x)$ and $\mathcal{F}_i(t,x)$ that are given precisely during the proof, such that the following holds. Let $A>0$. We have uniformly in $q\leq (\log x)^A$ and $T\leq x$ that
\begin{align}\label{majolem1}
\sum_{\substack{n\leq T\\n\equiv a(q)}}\Lambda_i(n,x)=\frac{\Xi(q,a)}{\varphi(q)}C_i(x)\prod_{p}\bigl(1-\frac{\chi(p)}{p-1}\bigr)\int_0^{\sqrt{T}}\mathcal{F}_i(t,x)\sqrt{T-t^2}dt+O_A(x(\log x)^{-A})
\end{align}
Furthermore it holds that
\begin{align}\label{majolem2}
\sum_{\substack{n\leq x\\n\equiv a(q)}}\Lambda_i(n,x)=\frac{H\Xi(q,a)}{\varphi(q)}xC_i(x)C_i'(x)\bigr(1+o(1)\bigr)
\end{align}
with $H$ as in \eqref{Hdef},
\begin{align}\label{majolem3}
\alpha^+(x):=\sum_{i=1}^3 C_i(x)C_i'(x)<2.9739+o(1).
\end{align}
\end{lemma}
\begin{proof}
The proof is mostly standard sieve technique in conjunction with either Corollary \ref{cor} or the Type I sum analysis of Fouvry and Iwaniec that also appears in the proof of Theorem \ref{GFI}.
\subsubsection*{The Case $i=1$}
We now prove \eqref{majolem1} and \eqref{majolem2} for $i=1$. We recall 
\begin{align*}
\Lambda_1(n,x)=\Lambda(n)\sum_{n=k^2+l^2}\omega_1(l),
\end{align*}
where
\begin{align*}
\omega_1(l)=\theta^+(l;D_1,D_0,z_1,z_0)\log\sqrt{x}
\end{align*}
with $\theta^+$ as given by \eqref{mainsieve} and
\begin{align*}
\xi_1&=0.183 
&z_1=x^{\xi_1/2}\\
z_0&=e^{(\log x)^{1/3}}
&D_0=e^{(\log x)^{2/3}}\\
\delta_0&=10^{-7}
&D_1=x^{1/3-\delta_0}.
\end{align*}
We want to apply Corollary \ref{cor}. This requires to understand
\begin{align}\label{32e1}
\sum_{\substack{l\leq T'\\l\equiv a(q)}}\omega_1(l)\psi(lq)=\prod_{p}\bigl(1-\frac{\chi(p)}{p-1}\bigr)\sum_{\substack{l\leq T'\\l\equiv a(q)}}\omega_1(l)\psi'(lq),
\end{align}
where $T'\leq \sqrt{T}$ and, as before, 
\begin{align*}
\psi'(lq)=\prod_{p|lq}\bigl(1-\frac{\chi(p)}{p-1}\bigr)^{-1}.
\end{align*}
We observe that
\begin{align}\label{psiconvo}
\psi'(l)=\sum_{r|l}\psi_0(r),
\end{align}
where $\psi_0(r)$ is the multiplicative function supported on square free $r$ only and given on primes by
\begin{align*}
\psi_0(p)=\frac{\chi(p)}{p-1-\chi(p)}.
\end{align*}

We now introduce the local density function $g(d,q,a)$. It is multiplicative in $d$ and on primes given by
\begin{align*}
g(p,q,a)=\begin{cases}\frac{1+\psi_0(p)}{p+\psi_0(p)} & \text{ if } p\nmid q\\
1 & \text{ if } p|(q,a)\\
0 & \text{ else. }
\end{cases}
\end{align*}
To evaluate \eqref{32e1}, we open the sieve in $w_1(l)$ and apply a routine calculation using \eqref{psiconvo}. This gives us
\begin{align}
\sum_{\substack{l\leq T'\\l\equiv a(q)}}w_1(l)\psi'(lq)=&(\log\sqrt{x}) V^+_1(x,q,a)\frac{T'}{q}\prod_{p|q}\Bigl(\frac{1+\psi_0(p)}{1+\psi_0(p)/p}\Bigr)\prod_{p}\Bigl(1+\frac{\psi_0(p)}{p}\Bigr) \nonumber \\
&+O( D_1D_0\log q\log T'),\label{noideaforname0}
\end{align}
where
\begin{align}
V^+_1(x,q,a)&=\sum_{d|P(z_0)}\lambda_I^+(d)g(d,q,a)\sum_{d|P(z_1,z_0)}\lambda_{II}^+(d)g(d,1,0)\nonumber\\
&=V^+_{1,I}(x,q,a)V^+_{1,II}(x)\label{noideaforname00},
\end{align}
say. Here we used that $z_0>q$ in our range, and so 
\begin{align*}
g(d,q,a)=g(d,1,0)
\end{align*} 
in the range of $V_{1,II}^+$.

We now show that $V_{1,I}^+$ fulfills
\begin{align}\label{V1I}
V_{1,I}^+(x,q,a)=\prod_{p|q}\Bigl(\frac{1+\psi_0(p)}{1+\psi_0(p)/p}\Bigr)^{-1}q\psi'(q)\Bigl(\Theta(q,a)+O(e^{-s_0}) \Bigr)\prod_{\substack{p\leq z_0\\ }}\Bigl(1-g(p,1,0) \Bigr),
\end{align}
where $\Theta(q,a)$ is given by \eqref{2} and 
\begin{align*}
s_0=\frac{\log D_0}{\log z_0}.
\end{align*} If $(a,q)=1$ the function $g(d,q,a)$ fulfills the sieve axioms \eqref{sieveaxiom} with $\kappa=1$, since $\psi_0(p)\ll 1/p$. In this case \eqref{V1I} follows from a straightforward application of the fundamental Lemma \ref{fundlem}. Next we bound the contribution of $a$ with $(a,q)\neq 1$ as follows. We write $g'(d,q)$ for the multplicative function defined by 
\begin{align*}
g'(p,q)=\begin{cases}\frac{1+\psi_0(p)}{p+\psi_0(p)} & \text{ if } p\nmid q \\
\frac{1}{p} &\text{ if } p|q.
\end{cases}
\end{align*}
We have
\begin{align*}
\sum_{a(q)}g(d,q,a)=q g'(d,q)
\end{align*}
and so
\begin{align*}
\frac{1}{q}\sum_{a(q)}\sum_{d|P(z_0)}\lambda_I^+(d)g(d,q,a)=\sum_{d|P(z_0)}\lambda_I^+(d)g'(d,q).
\end{align*}
Note that since we are considering an upper bound sieve, the sums over $d$ on both sides of this equality are always non negative. By the definition of $g'(p,q)$ we can apply fundamental Lemma \ref{fundlem} to the right hand side of this equality. The resulting main term is the same as when first applying fundamental Lemma \ref{fundlem} and then summing only over $a(q)$ with $(a,q)=1$. Hence we have
\begin{align*}
\frac{1}{q}\sum_{\substack{a(q)\\(a,q)\neq 1}}V_{1,I}^+(x,q,a)=O( e^{-s_0})
\end{align*}
and \eqref{V1I} follows in the remaining case $(a,q)\neq 1$. 

We now define \begin{align*}
C_1(x)=\log \sqrt{x}\prod_{p\leq z_0}\Bigl(1-g(p,1,0)\Bigl)V_{1,II}^+(x)\prod_{p}\bigl(1+\frac{\psi_0(p)}{p}\bigr).
\end{align*}
By the size of $s_0$, $D$, and $D_0$ we can bound the error terms sufficiently. By \eqref{noideaforname0}, \eqref{noideaforname00}, and \eqref{V1I} we then have
\begin{align*}
\sum_{\substack{l\leq T'\\l\equiv a(q)}}w_1(l)\psi'(lq)&=\Theta(q,a)\psi'(q)\sum_{\substack{l\leq T'\\}}w_1(l)\psi'(l)+O_A(\sqrt{x}(\log x)^{-A})\\
&=\Theta(q,a)\psi'(q)C_1(x)T'+O_A(\sqrt{x}(\log x)^{-A}).
\end{align*}
We recall that $\psi(l)=\prod_{p}\bigl(1-\frac{\chi(p)}{p-1}\bigr)\psi'(l)$ and are ready to apply Corollary \ref{cor}, giving us
\begin{align}\label{integrallambda1}
\sum_{\substack{n\leq T\\n\equiv a(q)}}\Lambda_1(n,x)=\frac{\Xi(q,a)}{\varphi(q)}C_1(x)\prod_{p}\bigl(1-\frac{\chi(p)}{p-1}\bigr)\int_0^{\sqrt{T}} \sqrt{T-t^2}dt+O_A(x(\log x)^{-A}),
\end{align}
which shows \eqref{majolem1} with $C_1(x)$ as defined and $\mathcal{F}_1(t,x)=1$.

The integral in \eqref{integrallambda1} is exactly $T\pi/4$ and since 
\begin{align*}
H=\frac{\pi}{4}\prod_p\Bigr(1-\frac{\chi(p)}{p-1}\Bigr)
\end{align*}
we can set
\begin{align*}
C'_1(x)=1
\end{align*}
to get a stronger version of
\begin{align*}
\sum_{\substack{n\leq x\\n\equiv a(q)}}\Lambda_1(n,x)=\frac{H\Xi(q,a)}{\varphi(q)}xC_1(x)C_1'(x)\bigr(1+o(1)\bigr),
\end{align*}
which is \eqref{majolem2} for $i=1$.

We conclude the case $i=1$ by bounding $C_1(x)$ with the help of the sieve results. We write
\begin{align*}
s=\frac{\log D_1}{\log z_1}
\end{align*}
and by Lemma \ref{ls} estimate with the use of the upper bound function of the linear sieve
\begin{align*}
V_{1,II}^+(x)\leq \Bigl(F(s)+o(1)\Bigr)\prod_{z_0<p\leq z_1}\Bigl(1-g(p,1,0)\Bigr).
\end{align*}
So we get the bound
\begin{align*}
C_1(x)\leq (\log \sqrt{x})\Bigl(F(s)+o(1)\Bigr)\prod_{p\leq z_1}\Bigl(1-g(p,1,0)\Bigr)\prod_{p}\bigl(1+\frac{\psi_0(p)}{p}\bigr).
\end{align*}
Since 
\begin{align*}
\bigl(1-g(p,1,0)\bigr)\bigl( 1+\frac{\psi_0(p)}{p}\bigr)=1-\frac{1}{p}
\end{align*}
we can apply a Mertens Formula and complete the Euler product to get
\begin{align*}
C_1(x)\leq \frac{1}{\xi_1 e^\gamma}\bigl(F(s)+o(1)\bigr).
\end{align*}
We have \begin{align*}
s=\frac{2/3-2\delta_0}{\xi_1}\in [3,5].
\end{align*}
So by the definition of $F(s)$ in that range (see for example \cite{fgs}) we arrive at the estimate
\begin{align}\label{C_1bound}
C_1(x)C'_1(x)\leq \frac{2}{2/3-2\delta_0}\Bigl(1+\int_2^{\frac{2/3-2\delta_0}{\xi_1}-1}\frac{\log (t-1)}{t}dt\Bigr)+o(1).
\end{align}
\subsubsection*{The Case $i=2$}
For the case $i=2$ we recall
\begin{align*}
\Lambda_2(n,x)=\Lambda(n)\sum_{n=k^2+l^2}\omega_2(l),
\end{align*}
where
\begin{align*}
\omega_2(l)=\frac{-\log \sqrt{x}}{2}\sum_{\substack{z_1\leq p <z\\ l=pm}}\theta^-(m;D_1/p,D_0,z_1,z_0)
\end{align*}
with the same parameters as in the last subsection and
\begin{align*}
z&=x^{\xi/2}, \\ \xi&=0.265.
\end{align*}
We set
\begin{align*}
V_{2,II}^-(x,p)=\sum_{\substack{d|P(z_1,z_0)\\ d\leq D_1/p}}\lambda_{II}^-(d,p)g(d,1,0),
\end{align*}
where $\lambda_{II}^-(d,p)	$ are the sieve weights of the main sieve of $\theta^-(m;D_1/p,D_0,z_1,z_0)$. Furthermore write
\begin{align*}
C_2(x)=\frac{-\log \sqrt {x}}{2}\prod_{p\leq z_0}\bigl(1-g(p,1,0)\bigr)\sum_{z_1\leq p<z}\frac{V^-_{2,II}(x,p)}{p}\prod_{p}\bigl(1+\frac{\psi_0(p)}{p}\bigr).
\end{align*}
Recall that the fundamental lemma range sieve is also an upper bound sieve for $i=2$. With similar steps as for $i=1$ we thus get
\begin{align*}
\sum_{\substack{n\leq T\\n\equiv a(q)}}\Lambda_2(n,x)=\frac{\Xi(q,a)}{\varphi(q)}C_2(x)\prod_p \bigl(1-\frac{\chi(p)}{p-1}\bigr)\int_0^{\sqrt{T}} \sqrt{T-t^2}dt+O_A(x(\log x)^{-A}),
\end{align*}
which is \eqref{majolem1} for $i=2$.

We set $C_2'(x)=1$ and just as before have
\begin{align*}
\sum_{\substack{n\leq x\\n\equiv a(q)}}\Lambda_2(n,x)=\frac{H\Xi(q,a)}{\varphi(q)}xC_2(x)C_2'(x)\bigl(1+o(1)\bigr).
\end{align*}
An application of Mertens Theorem and Lemma \ref{ls} in the lower bounds case now gives us
\begin{align*}
C_2(x)C'_2(x)\leq \frac{-1}{2 \xi_1 e^{\gamma}}\sum_{z_1\leq p<z}\frac{f(s_p)+O((\log x)^{-1/6})}{p},
\end{align*}
where
\begin{align*}
s_p=\frac{\log D_1/p}{\log z_1}=\frac{2/3-2\delta_0-2\log p}{\xi_1}.
\end{align*}
As $\xi_1 \leq 2\log p< \xi$, we have $s_p\in [2,4]$. We apply the Prime Number Theorem and the definition of $f(s_p)$ in the required range to get
\begin{align}
C_2(x)C_2'(x)&\leq -\int_{\xi_1/2}^{\xi/2}\frac{\log\bigl(\frac{2/3-2\delta_0-2t}{\xi_1}-1 \bigr)}{t(2/3-2\delta_0-2t)}dt+o(1)\nonumber \\
&=-\int_{\xi_1}^\xi\frac{\log\bigl(\frac{2/3-2\delta_0-t}{\xi_1}-1 \bigr)}{t(2/3-2\delta_0-t)}dt+o(1). \label{C_2bound}
\end{align}
\subsubsection*{The Case $i=3$}
To show the stated results for $\Lambda_3(n,x)$ we employ a similar approach as in the main term of Theorem \ref{GFI}. We recall
\begin{align*}
\Lambda_3(n,x)=(\log x)\theta^+(n;x^{1/2-\delta_0},D_0,x^{1/2-2\delta_0},z_0)\sum_{n=k^2+l^2}\omega_3(l)
\end{align*}
with
\begin{align*}
\omega_3(l)=\frac{\log \sqrt{x}}{2}\sum_{\substack{z_1\leq p_1<p_2<p_3< z\\l=p_1p_2p_3m}}\rho(m,p_1,z_0)\theta_I^+(m,D_0,P(z_0)).
\end{align*}
As a preparation we first show that instead of $\Lambda_3$ we can consider $\Lambda_3'$ given by
\begin{align*}
\Lambda_3'(n,x)=(\log x)\theta^+(n;x^{1/2-\delta_0},D_0,x^{1/2-2\delta_0},z_0)\sum_{n=k^2+l^2}\omega_3'(l),
\end{align*}
where
\begin{align*}
\omega_3'(l)=\frac{\log \sqrt{x}}{2}\sum_{\substack{z_1\leq p_1<p_2<p_3< z\\l=p_1p_2p_3m}}\rho(m,p_1).
\end{align*}
In other words, the sieve is replaced with a rough number indicator. We have $\omega_3'(l)\leq \omega_3(l)$ and estimate crudely
\begin{align*}
\sum_{\substack{n\leq T \\ n\equiv a(q)}}(\Lambda_3(n,x)-\Lambda_3'(n,x))&\leq (\log x) \sqrt{T}\sum_{l\leq \sqrt{T}}(\omega_3(l)-\omega_3'(l)).
\end{align*}
Furthermore, with the help of an upper bound sieve we have
\begin{align*}
\sum_{l\leq \sqrt{T}}(\omega_3(l)-\omega_3'(l))&=\frac{\log \sqrt{x}}{2}\sum_{l\leq \sqrt{T}}\sum_{\substack{z_1\leq p_1<p_2<p_3< z\\l=p_1p_2p_3m}}\rho(m,p_1,z_0)\bigl\{\theta_I^+(m,D_0,P(z_0))-\rho(m,z_0) \bigr\}\\
&\ll (\log x)\sum_{\substack{l\leq \sqrt{T}\\ }}\rho(l,z_1,z_0)(\theta_I^+(l,D_0,P(z_0))-\rho(l,z_0))\\
&\leq (\log x) \sum_{n\leq \sqrt{T}}\theta_{II}^+(n,T^{1/3},P(z_1,z_0))\bigl\{\theta_I^+(n,D_0,P(z_0))-\theta_I^-(n,D_0,P(z_0))\bigr\}\\
&=(\log x)\sqrt{T}V'^{+}_I(T^{1/3},P(z_1,z_0))\bigl\{V'^{+}_{II}(D_0,P(z_0))-V'^{-}_{II}(D_0,P(z_0))\bigr\}+O(T^{1/3+\epsilon}),
\end{align*}
where $V'^{+}_I$ and $V'^\pm_{II}$ are related to the sieves and the local density function is $g(d)=d^{-1}$. By Lemma \ref{ls} in particular $V'^+_I\ll 1$ and so by the fundamental Lemma \ref{fundlem} we conclude
\begin{align*}
\sum_{\substack{n\leq T \\ n\equiv a(q)}}(\Lambda_3(n,x)-\Lambda_3'(n,x))\ll x(\log x)^{-A}.
\end{align*}
Thus it suffices to prove Lemma \ref{majolem} for $\Lambda_3'$.

We set
\begin{align*}
a_n=\mathbbm{1}_{n\equiv a(q)}\sum_{\substack{n=k^2+l^2\\}}\omega_3'(l)
\end{align*}
and 
\begin{align*}
A_d(T)=\sum_{\substack{n\equiv 0 (d)\\ n\leq T }}a_n.
\end{align*}
As in \eqref{mapprox} we have the approximation 
\begin{align}\label{L3MT}
M_d(T)=\frac{1}{dq}\sum_{k^2+l^2\leq T}\omega_3'(l)\varrho_l(d;q,a)
\end{align}
with $\varrho_l(d;q,a)=0$ except if $(d,q)=1$. If $(d,q)=1$ we have again
\begin{align*}
\varrho_l(d;q,a)=\varrho_l(d)\varrho_l(q,a).
\end{align*}
We require only information for square free $d$ and in that case have
\begin{align*}
\varrho_l(d)&=\prod_{\substack{p|d\\p\nmid l}}\bigl(1+\chi(p)\bigr).
\end{align*}
Writing $M_d(T)=A_d(T)+R_d(T)$ and opening $\theta^+$ we get 
\begin{align}\nonumber
\sum_{\substack{n\leq T\\n\equiv a(q)}}\Lambda_3'(n,x)=(\log x)\sum_{d|P(x^{1/2-2\delta_0})}\lambda^+_d M_d(T)+O\bigl(\sum_{d\leq x^{1/2-\delta_0}D_0}|R_d(T)|\bigr),
\end{align}
where the sieve weights $\lambda^+_d$ are the combination of $\lambda^+_I$ and $\lambda^+_{II}$ as given by \eqref{mainsieve}. By \eqref{linrem} we can bound the remainder and have
\begin{align}
\sum_{\substack{n\leq T\\n\equiv a(q)}}\Lambda_3'(n,x)=(\log x)\sum_{d|P(x^{1/2-2{\delta_0}})}\lambda^+_d M_d(T)+O_A(x(\log x)^{-A})\label{L3eq}
\end{align}

We now consider the main term $M_d(T)$ as given by \eqref{L3MT}. The fact that $\varrho_l(d)$ depends on both $l$ and $d$ is preventing a similar treatment of the sieve as before. However terms with $(l,d)\neq 1$ can be shown to only give negligible contribution to $M_d(T)$. We have
\begin{align*}
\frac{1}{d}\sum_{\substack{l\leq \sqrt{T}\\ (l,d)\neq 1}}\omega_3'(l)\rho_l(d)&=\frac{1}{2 d}\sum_{\substack{l\leq \sqrt{T}\\ (l,d)\neq 1}}\rho_l(d) \sum_{\substack{z_1\leq p_1<p_2<p_3< z\\l=p_1p_2p_3m}}\rho(m,p_1) \\
&\ll_\epsilon \frac{1}{d^{1-\epsilon}} \sum_{\substack{j|d \\ j\neq 1}}\sum_{\substack{l\leq \sqrt{T}\\ j|l}} \sum_{\substack{z_1\leq p_1<p_2<p_3< z\\l=p_1p_2p_3m}}\rho(m,p_1)\\
&\ll_\epsilon \frac{T^{1/2+\epsilon}}{d^{1-\epsilon}} \sum_{\substack{j|d\\ j\geq z_1}} \frac{1}{j} \\
&\ll_\epsilon \frac{T^{1/2+\epsilon}}{d^{1-\epsilon}z_1}  \\
&\ll_A \frac{\sqrt{T}}{d(\log x)^A}
\end{align*}
In the case $(l,d)=1$ we have
\begin{align*}
\varrho_l(d)=\prod_{p|d}\bigl(1+\chi(p)\bigr).
\end{align*}
So we can define the local density function $g_3(d,q,a)$ by
\begin{align*}
g_3(p,q,a)=\begin{cases}\frac{1+\chi(p)}{p} & \text{ if } p\nmid q\\
1 & \text{ if } p|(q,a)\\
0 & \text{ else}.
\end{cases}
\end{align*}
After again adding terms with $(l,d)\neq 1$ we get
\begin{align*}
M_d(T)&=\frac{g_3(d,q,a)}{q}\sum_{k^2+l^2\leq T}\omega_3'(l)\varrho_l(q,a)+O_A\bigl(\frac{T}{d (\log x)^A}\bigr).
\end{align*}
We now return to \eqref{L3eq}. We get uniformly in the required range of $q$
\begin{align*}
\sum_{\substack{n\leq T\\n\equiv a(q)}}\Lambda_3'(n,x)=\frac{\log x}{q}V_3^+(x,q,a)\sum_{k^2+l^2\leq T}\omega_3'(l)\varrho_l(q,a)+O_A(x(\log x)^{-A}),
\end{align*}
where
\begin{align*}
V_3^+(x,q,a)=\sum_{d|P(z_0)}\lambda_I^+(d)g_3(d,q,a)\sum_{d|P(x^{1/2-2\delta_0},z_0)}\lambda_{II}^+(d)g_3(d,1,0).
\end{align*}
The remaining sieve parts can be evaluated as in the cases before. We set 
\begin{align}\label{C3def}
C_3(x)=(\log x)\prod_p \bigl(1-\frac{\chi(p)}{p-1}\bigr)^{-1} \prod_{p\leq z_0}\Bigl(1-\frac{1+\chi(p)}{p}\Bigr)\sum_{d|P(x^{1/2-2\delta_0},z_0)}\lambda_{II}^+(d)g_3(d,1,0)
\end{align}
and get
\begin{align}\label{L3la}
\sum_{\substack{n\leq T\\n\equiv a(q)}}\Lambda_3'(n,x)=\Theta(q,a)\psi'(q)C_3(x)\prod_p \bigl(1-\frac{\chi(p)}{p-1}\bigr)\sum_{k^2+l^2\leq T}\omega_3'(l)\varrho_l(q,a)+O_A(x(\log x)^{-A}).
\end{align}

We now evaluate the main term. Let $T'\leq \sqrt{T}$ and consider
\begin{align*}
\sum_{\substack{l\leq T'\\l\equiv b(q)}}\omega_3'(l)=\frac{\log \sqrt{x}}{2}\sum_{\substack{l\leq T'\\l\equiv b(q)}}\sum_{\substack{z_1\leq p_1<p_2<p_3<z\\ l=p_1p_2p_3m}}\rho(m,p_1).
\end{align*}
By applying the Siegel-Walfisz Theorem, summation by parts, and Lemma \ref{buch} we have
\begin{align*}
\sum_{\substack{l\leq T'\\l\equiv b(q)}}\omega_3'(l)=&\frac{\Theta(q,b)\log \sqrt{x}}{2}\int_0^{T'} \int_{z_1\leq y_1<y_2<y_3<z}\frac{B(t/(y_1y_2y_3),y_1)}{y_1y_2y_3 \log y_1\log y_2 \log y_3}d\bm{y}dt\\
&+O_A(\sqrt{x}(\log x)^{-A})\\
=&\Theta(q,b)\int_0^{T'}\mathcal{F}_3(t,x)dt+O_A(\sqrt{x}(\log x)^{-A}),
\end{align*}
say. Plugging this into \eqref{L3la}, we obtain
\begin{align*}
\sum_{\substack{n\leq T\\n\equiv a(q)}}\Lambda_3'(n,x)=\frac{\Xi(q,a)}{\varphi(q)} C_3(x)\prod_p \bigl(1-\frac{\chi(p)}{p-1}\bigr)\int_0^{\sqrt{T}}\mathcal{F}_3(t,x)\sqrt{T-t^2}dt+O_A(x(\log x)^{-A}),
\end{align*}
which is \eqref{majolem1} for $i=3$.

In the case $T=x$ we use partial summation and the derivative bound \eqref{derivbound} of Lemma \ref{buch} to obtain 
\begin{align*}
\int_0^{\sqrt{T}}\mathcal{F}_3(t,x)\sqrt{T-t^2}dt=\frac{\pi}{4}x \mathcal{F}_3(\sqrt{x},x)\bigl(1+O((\log x)^{-1})\bigr).
\end{align*}
We now set
\begin{align}\label{C'3def}
C'_3(x)=\mathcal{F}_3(\sqrt{x},x)
\end{align}
and have showed
\begin{align*}
\sum_{\substack{n\leq x\\n\equiv a(q)}}\Lambda_3'(n,x)=\frac{H\Xi(q,a)}{\varphi(q)}xC_3(x)C'_3(x)\bigl(1+o(1)\bigr),
\end{align*}
as required for \eqref{majolem2}.

We now bound $C_3(x)$ that is defined in \eqref{C3def}. We observe that as $y\to \infty$
\begin{align*}
\prod_{p\leq y}\Bigl(1-\frac{1+\chi(p)}{p} \Bigr)=\frac{1}{e^{\gamma}\log y}\prod_p\Bigl(1-\frac{\chi(p)}{p-1} \Bigr)\bigl(1+o(1)\bigr).
\end{align*}
By Lemma \ref{ls} together with the definition of $F(s)$ for $1\leq s \leq 2$ we obtain the bound 
\begin{align}\label{C_3bound}
C_3(x)\leq \frac{4 }{1-2\delta_0}\bigl(1+o(1)\bigr).
\end{align}
We complete the case $i=3$ by noting that
\begin{align*}
\mathcal{F}_3(\sqrt{x},x)&=\frac{\log \sqrt{x}}{2} \int_{z_1\leq y_1<y_2<y_3<z}\frac{B(\sqrt{x}/(y_1y_2y_3),y_1)}{y_1y_2y_3 \log y_1\log y_2 \log y_3}d\textbf{y}\\
&=\frac{1}{2}\int_{\xi_1\leq \beta_1\leq \beta_2\leq \beta_3\leq \xi}\frac{\mathcal{B}(\frac{1-\beta_1-\beta_2-\beta_3}{\beta_1})}{\beta_1^2\beta_2\beta_3}d\bm{\beta}.
\end{align*}
So by \eqref{C'3def} and \eqref{C_3bound} we get
\begin{align}\label{C_3cbound}
C_3(x)C_3'(x)\leq \frac{2}{1-2\delta_0}\int_{\xi_1\leq \beta_1\leq \beta_2\leq \beta_3\leq \xi}\frac{\mathcal{B}(\frac{1-\beta_1-\beta_2-\beta_3}{\beta_1})}{\beta_1^2\beta_2\beta_3}d\bm{\beta}+o(1).
\end{align}

\subsubsection*{Estimate of $\alpha^+(x)$}
We have showed $\eqref{majolem1}$ and $\eqref{majolem2}$ for all $i$. To show the estimate \eqref{majolem3} we recall the proved estimates for $C_i(x)C_i'(x)$. By \eqref{C_1bound}, \eqref{C_2bound}, and \eqref{C_3cbound} we have
\begin{align*}
C_1(x)C'_1(x)&\leq \frac{2}{2/3-2\delta_0}\Bigl(1+\int_2^{\frac{2/3-2\delta_0}{\xi_1}-1}\frac{\log (t-1)}{t}dt\Bigr)+o(1)\\
C_2(x)C_2'(x)&\leq-\int_{\xi_1}^\xi\frac{\log\bigl(\frac{2/3-2\delta_0-t}{\xi_1}-1 \bigr)}{t(2/3-2\delta_0-t)}dt+o(1)\\
C_3(x)C_3'(x)&\leq \frac{2}{1-2\delta_0}\int_{\xi_1\leq \beta_1\leq \beta_2\leq \beta_3\leq \xi}\frac{\mathcal{B}(\frac{1-\beta_1-\beta_2-\beta_3}{\beta_1})}{\beta_1^2\beta_2\beta_3}d\bm{\beta}+o(1).
\end{align*}
To estimate the third integral, we note that in the range of integration
\begin{align*}
\frac{1-\beta_1-\beta_2-\beta_3}{\beta_1}\leq \frac{1-3\xi_1}{\xi_1}<2.46
\end{align*}
and that
\begin{align*}
\mathcal{B}(u)=\begin{cases}0 &\text{ if } u<1\\
1/u &\text{ if } 1\leq u <2\\
(1+\log(u-1))/u &\text{ if } 2\leq u<3
\end{cases}
\end{align*}
This is \cite[(1.4.16)]{harb} with our extension of the range. We now numerically estimate with the help of Mathematica\textsuperscript{\textregistered}. We define functions for the first two integrals.
\begin{lstlisting}
f1[d_, x_] :=
2/d (1 + Integrate[Log[t - 1]/t, {t, 2, d/x - 1}])
f2[d_, x_, y_]:=
Integrate[Log[(d - t)/x - 1]/(t (d - t)), {t, x, y}]
\end{lstlisting}
We further define the Buchstab function in the necessary range.
\begin{lstlisting}
B[t_] := Piecewise[{{0, t < 1}, {1/t, 1 <= t <= 2}, 
{(1 + Log[t -1])/t, 2 < t <= 3}}] 
\end{lstlisting}
The third integral can then be expressed as follows.
\begin{lstlisting}
f3[o_,x_, y_] := 1/o NIntegrate[B[(1 - a - b - c)/c]/(a b c^2),
{c, x, y}, {b, c, y}, {a, b, y}]
\end{lstlisting}
We now plug in the values $o=1/2-10^{-7}$, $d=2/3-2*10^{-7}$, $x=0.183$, and $y=0.265$. The numerical calculation gives us 
\begin{align*}
C_1(x)C_1'(x)&\leq 3.21513 +o(1)\\
C_2(x)C_2'(x)&\leq -0.29237+o(1)\\
C_3(x)C_3'(x)&\leq 0.051120+o(1).
\end{align*}
Consequently we arrive at
\begin{align*}
\sum_{i\in\{1,2,3\}}C_i(x)C'_i(x)<2.9739+o(1).
\end{align*}
This completes the proof of Lemma \ref{majolem}.
\end{proof}

\subsection{Combinatorical Decomposition}
The result of the previous subsection is sufficient for the major arc case. To show minor arc case of Proposition \ref{prmajo} we use a dissection of $\Lambda^+$ into Type I and II sums in certain ranges. This decomposition is based on a combinatorical sieve approach of Duke, Friedlander, and Iwaniec in \cite{dfi}. We introduce the following notation. For a complex valued sequence $\mathcal{C}=\{c(n)\}$ with finite support we write
\begin{align*}
S(\mathcal{C},z)=\sum_{n}\rho(n,z)c(n).
\end{align*}
Furthermore we set 
\begin{align*}
\mathcal{C}_d=\{c(dn)\}
\end{align*} 
and denote the corresponding partial sums by
\begin{align*}
|\mathcal{C}_d|=\sum_n c(dn).
\end{align*}
Our dissection into linear and bilinear sums is done by Lemma 2 of \cite{dfi} that has some similarities to Vinogradov's approach. Including all requirements it reads as follows.
\begin{lemma}\label{dfiresult} Let $\mathcal{C}=\{c(n)\}$ be a sequence of complex numbers such that 
\begin{align*}
\sum_{n}|c(n)|<\infty
\end{align*}
and
\begin{align}\label{ccond}
\sum_{n\equiv 0(d)}|c(n)|\leq \gamma(d)X,
\end{align}
for some multiplicative function $\gamma$ with
\begin{align*}
\gamma(p)\ll p^{-1}.
\end{align*}
Let further $3\leq K\leq U_1<U_2<z<D_I$ and set
\begin{align*}
y_k=U_2(U_1/U_2)^{k/K}. 
\end{align*}
Then we have
\begin{align}
S(\mathcal{C},z)-\sum_{U_2\leq p<q<z}S(\mathcal{C}_{pq},p)&={\sum_{\substack{d|P(z)\\d<D_I}}}^{U_1}\mu(d)|\mathcal{C}_d|+\sum_{0\leq k<K} \sum_{y_{k+1}\leq p<y_k<q<z}S(\mathcal{C}_{pq},y_k)\nonumber\\
&+\theta X G(z)^2\bigl(2^{-\frac{\log D_I/z}{\log U_1}}+O(K^{-1}\log U_2)\bigr),\label{dfiresulteq}
\end{align}
where the variables $p$ and $q$ are primes,
\begin{align*}
G(z)&=\prod_{p<z}(1+\gamma(p)),\\
|\theta|&\leq 1,
\end{align*}
and $\sum^{U_1}$ ranges over over $d$ having at most one prime factor $\geq U_1$.
\end{lemma}
As in \cite{dfi}, for a given sequence $\mathcal{C}=\{c(n)\}$ of complex numbers we define the general linear or Type I sums with unspecified coefficients $|\lambda_d|\leq 1$ as
\begin{align}\label{gtypeI}
\mathcal{R}_{I}(D_I;\mathcal{C})=\sum_{d<D_I}\lambda_d \sum_{n}c(dn).
\end{align}
Let $\nu(n)$ denote the number of prime divisors of $n$, counted without multiplicity. For coefficients $|\alpha_n|\leq \nu(n)$, $|\beta_d|\leq 1$, $\beta$ supported on primes only, we define the general bilinear or Type II sums by
\begin{align}\label{gtypeII}
\mathcal{R}_{II}(U_1,U_2;\mathcal{C})=\sum_{U_1\leq d<U_2}\beta_d\sum_{\substack{n\\(d,n)=1}}\alpha_n c(dn).
\end{align}
To apply Lemma \ref{dfiresult} in our case, we define for $i\in\{1,2,3\}$ the sequences 
\begin{align*}
\mathcal{C}_i(\gamma)=\mathcal{C}_i(\gamma;W,b,x)=\{c_{i,W,b,x}(n,\gamma)\},
\end{align*}
where for $n\equiv b(W)$ and $n\leq x$ we set
\begin{align*}
c_i(n,\gamma)=c_{i,W,b,x}(n,\gamma)=e(\gamma (n-b)/W) \sum_{n=k^2+l^2}\omega_i(l)
\end{align*}
with $\omega_i(l)$ as in the construction of $\Lambda^+$. Further we set $c_i(n,\gamma)=0$ in all other cases. With this notation we show the following result. 

\begin{lemma}\label{decomp}
Let $x$, $b$, $W$, and $N$ as in Theorem \ref{MT}. Let further $\gamma\in \mathbbm{R}$, $U_1=x^{(\log \log x)^{-4}}$, $U_2=x^{1/3}$, $z=x^{1/2-(\log \log x)^{-1}}$,  and $D_I=zx^{(\log \log x)^{-2}}$. Let $\mathcal{C}_i=\mathcal{C}_i(\gamma,W,b,x)$ as above and suppose we have for $i \in \{1,2,3\}$  the Type I bound 
\begin{align*}
\mathcal{R}_{I}(D_I,\mathcal{C}_i(\gamma))\ll x(\log x)^{-2}
\end{align*}
and for $i\in \{1,2\}$ the Type II bound
\begin{align*}
\mathcal{R}_{II}(U_1,U_2,\mathcal{C}_i(\gamma))\ll x(\log x)^{-24}.
\end{align*}
Then it holds that
\begin{align*}
|\sum_{n\leq N}\Lambda^+_{W,b}(n)e(\gamma n)|\ll \frac{N}{\log \log x}.
\end{align*}
In particular, the above stated Type I and II bounds for $\gamma\in \mathfrak{m}$ imply Proposition \ref{prmajo} for $\gamma\in \mathfrak{m}$.
\end{lemma}
\begin{proof}
By construction $\Lambda^+$ consists of a part where the outer function is $\Lambda$ and a part where it is a sieve. The part where the outer function is a sieve is $\Lambda_3$.  Recalling that the sieve $\Omega$ has level $x^{1/2-10^{-7}}$, the assumed Type I bound for $i=3$ gives us
\begin{align*}
\sum_{n\leq N}\Lambda_{3;W,b}(n,x)e(\gamma n)&= \sum_{n\leq N}\Omega_{W,b}(n,x)e(\gamma n)\sum_{n=k^2+l^2}\omega_3(l)\\
&\ll \frac{x}{\log x}\\
&\ll \frac{N}{\log \log x},
\end{align*}
as required.

For the rest of the proof we have to consider only $i\in \{1,2\}$. We want to apply Lemma \ref{dfiresult} on the remaining constituents of $\Lambda^+$ that are of the form 
\begin{align*}
\Lambda^{\omega_i}(n)=\Lambda(n)\sum_{n=k^2+l^2}\omega_i(l).
\end{align*}
Before we do so, we remove the contribution of $l$ with many divisors, which simplifies showing a suitable version of \eqref{ccond}. Let $\tau(l)< \mathcal{L}:=(\log x)^{15}$ and write $\omega_i(l)=\omega_i^\sharp(l)+\omega_i^\flat(l)$ with
\begin{align*}
\omega_i^\sharp(l)&:=\begin{cases} \omega_i(l) & \text{if } \tau(l)<\mathcal{L}\\
0,& \text{else} \\
\end{cases} 
\end{align*}
and
\begin{align*}
\omega_i^\flat(l)&:=\begin{cases} 0, & \text{if } \tau(l)<\mathcal{L}\\
\omega_i(l),& \text{else.}
\end{cases}
\end{align*}

By using the bounds \eqref{o1bound}, \eqref{o2bound} we estimate
\begin{align*}
|\sum_{n\leq N}\Lambda^{\omega_i^\flat}_{W,b}e(\gamma n)|&\leq \frac{\varphi(W)}{\Xi(W,b)WH}(\log x)\sum_{k^2+l^2\leq x}|\omega_i^\flat(l)|\nonumber \\
&\leq \frac{\varphi(W)}{\Xi(W,b)WH}(\log x)^3\sqrt{x}\sum_{l\leq \sqrt{x}}\frac{\tau(l)^2}{\mathcal{L}}\nonumber \\
&\ll \frac{x (\log x)^7}{\mathcal{L}}.
\end{align*}
We thus have
\begin{align}\label{t0readd}
|\sum_{n\leq N}\Lambda^{\omega_i}_{W,b}(n)e(\gamma n)|=|\sum_{n\leq N}\Lambda^{\omega_i^\sharp}_{W,b}(n)e(\gamma n)|+O\bigl( \frac{x (\log x)^7}{\mathcal{L}}\bigr).
\end{align}
Assume that $z$ is as given in the Lemma. It follows by our normalisation that for any $\epsilon>0$
\begin{align*}
|\sum_{n\leq N}\Lambda^{\omega_i^\sharp}_{W,b}(n)e(\gamma n)|&=\frac{\varphi(W)}{\Xi(W,b)WH}\sum_{\substack{n\leq x\\ n\equiv b(W)}}\Lambda(n)e(\gamma (n-b)/W)\sum_{n=k^2+l^2}\omega_i^\sharp(l)+O_\epsilon(x^{\epsilon})\\
&=\frac{\varphi(W)}{\Xi(W,b)WH}\sum_{\substack{z\leq p\leq x\\ p\equiv b(W)}}(\log p) e(\gamma (p-b)/W)\sum_{p=k^2+l^2}\omega_i^\sharp(l)+O_\epsilon(x^{1/2+\epsilon}).
\end{align*}
We write $\mathcal{C}_i^\flat(\gamma)$ and $\mathcal{C}_i^\sharp(\gamma)$ for the sequences associated to $\omega_i^\flat$ and $\omega_i^\sharp$ respectively. If we have the bound
\begin{align}\label{demessing}
\sum_{z\leq p \leq x}c_i^\sharp(p,\gamma)\ll \frac{x}{(\log x)^2}+\frac{x\Xi(W,b)}{\varphi(W)(\log x)(\log \log x)},
\end{align}
it follows by summation by parts that
\begin{align*}
|\sum_{n\leq N}\Lambda^{\omega_i^\sharp}_{W,b}(n)e(\gamma n)|\ll \frac{N}{(\log \log x)}.
\end{align*}
Considering \eqref{t0readd} and our choice of $\mathcal{L}$, this is sufficient. The rest of the proof consists in showing that \eqref{demessing} holds.
We now show how we can manipulate the sum in \eqref{demessing} into an object to which Lemma \ref{dfiresult} is applicable. We start by applying Buchstab's identity to get
\begin{align}\label{returntoeq}
S(\mathcal{C}_i^\sharp(\gamma),z)=\sum_{z< p \leq x}c_i^\sharp(p,\gamma)+\sum_{z< p \leq x^{1/2}}\sum_{pj\leq x}\rho(j,p)c_i^\sharp(pj,\gamma).
\end{align}
This is similar to \cite{dfi}. However, due to the sparseness of the sequence, we need an additional argument to bound the second sum. We start by applying the triangle inequality and get
\begin{align*}
\bigl|\sum_{z\leq p <x^{1/2}}\sum_{pj\leq x}\rho(j,p)c_i^\sharp(pj,\gamma)\bigr|\leq \sum_{z\leq p<x^{1/2}}\sum_{\substack{\substack{pj\leq x\\ pj\equiv b(W)}}}\rho(j,p)\sum_{\substack{pj=k^2+l^2}}|\omega_i^\sharp(l)|.
\end{align*}
The double sum over $j$ can be estimated with a sieve, since by \eqref{linrem} we have sufficient level of linear distribution. The process is slightly technical, but mostly standard, so we skip it. The obtained estimate is 
\begin{align*}
\bigl| \sum_{z< p \leq x^{1/2}}\sum_{pj\leq x}\rho(j,p)c_i^\sharp(pj,\gamma)\bigr|\ll \frac{x \Xi(W,b)}{\varphi(W) \log x}\sum_{z< p \leq x^{1/2}}\frac{1}{p}+x^{0.9}.
\end{align*}
By Mertens's Theorem we have
\begin{align*}
\sum_{z\leq p <x^{1/2}}\frac{1}{p}\ll \frac{1}{\log \log x}.
\end{align*}
So by \eqref{returntoeq}
\begin{align*}
\sum_{z\leq p \leq x}c_i^\sharp(p,\gamma)= S(\mathcal{C}_i^\sharp(\gamma),z) +O \Bigl(\frac{\Xi(W,b) x}{\varphi(W)(\log x)(\log \log x)}\Bigr).
\end{align*}

To complete the proof of \eqref{demessing}, we now apply Lemma \ref{dfiresult} and the assumed Type I and II estimates to show
\begin{align}\label{demessing2}
|S(\mathcal{C}_i^\sharp(\gamma),z)|\ll \frac{x}{(\log x)^2}.
\end{align}
Since we have the bound
\begin{align*}
|c_i^\sharp(n,\gamma)|\leq (\log x)^{15} \tau(n),
\end{align*}
condition \eqref{ccond} holds with $\gamma(d)=\tau(d)/d$ and $X=x(\log x)^{15}$. Choose $K=(\log x)^{22}$ and assume that $U_1$, $U_2$, $z$, and $D_I$ are as in the statement of Lemma \ref{decomp}. We get
\begin{align}\nonumber
|X G(z)^2\bigl(2^{-\frac{\log D_I/z}{\log U_1}}+cK^{-1}\log U_2\bigr)|&\ll x(\log x)^{19}\bigl(2^{-\frac{(\log \log x)^{-2}}{(\log \log x)^{-4}}}+(\log x)^{-21} \bigr)\\
&\ll \frac{x}{(\log x)^2}.\label{apply1}
\end{align}
We now consider 
\begin{align*}
\sum_{U_2\leq p<q<z}S(\mathcal{C}_{i;pq}^\sharp(\gamma),p)
\end{align*}
that appears in \eqref{dfiresulteq} for the sequence we consider. Since $U_2=x^{1/3}$ we have
\begin{align*}
|S(\mathcal{C}_{i;pq}^\sharp(\gamma),p)&=|c_i^\sharp(pq,\gamma)|\\
&\leq 2 (\log x)^{14}.
\end{align*}
So we get
\begin{align}
|\sum_{U_2\leq p<q<z}S(\mathcal{C}_{i;pq}^\sharp(\gamma),p)|&\ll z^2 (\log x)^{14}\nonumber\\
&\ll \frac{x}{(\log x)^2}\label{apply2}.
\end{align}
We use \eqref{apply1} and \eqref{apply2} in the main statement of Lemma \ref{dfiresult}. So \eqref{dfiresulteq} becomes
\begin{align}
S(\mathcal{C}_i^\sharp(\gamma),z)=\sum_{\substack{d|P(z)\\ d<D_I}}^{U_1}\mu(d)|\mathcal{C}_{i;d}^\sharp(\gamma)|+\sum_{0\leq k<K} \sum_{y_{k+1}\leq p<y_k<q<z}S(\mathcal{C}_{i;pq}^\sharp(\gamma),y_k)+O\bigl(\frac{x}{(\log x)^2}\bigr).\label{decompdfiapplied}
\end{align}
The sum over $d$ is related to the Type I estimate, the double sum over $p$ and $q$ to the Type II estimate.

To apply the assumed Type I and II bounds, we need to go back from $\mathcal{C}_i^\sharp(\gamma)$ to $\mathcal{C}_i(\gamma)$. We achieve this by bounding the contribution of $\mathcal{C}_i^\flat(\gamma)$. For the Type I sum we have
\begin{align*}
\bigl|\sum_{\substack{d|P(z)\\ d<D_I}}^{U_1}\mu(d)|\mathcal{C}_{i;d}^\flat(\gamma)|\bigr|&\leq \sum_{dn\leq x}|c_i^\flat(dn,\gamma)|\\
&\leq \sum_{n\leq x}\tau(n)|c_i^\flat(n,\gamma)|\\
&\leq \bigl(\sum_{n\leq x}\tau(n)^3\bigr)^{1/2}\bigl(\sum_{n\leq x}\frac{|c_i^\flat(n,\gamma)|^2}{\tau(n)}\bigr)^{1/2}.
\end{align*}
By Cauchy's inequality it holds that
\begin{align*}
|c_i^\flat(n,\gamma)|^2&=|\sum_{n=k^2+l^2}\omega_i^\flat(l)|^2\\
&\leq \tau(n)\sum_{n=k^2+l^2}|\omega_i^\flat(l)|^2.
\end{align*}
We get
\begin{align*}
\sum_{n\leq x}\frac{|c_i^\flat(n,\gamma)|^2}{\tau(n)}&\leq (\log x)^4\sqrt{x}\sum_{l\leq \sqrt{x}}\frac{\tau(l)^4}{\mathcal{L}^2}\\
&\ll \frac{x (\log x)^{19}}{\mathcal{L}^2}.
\end{align*}
Using the assumed Type I bound for $\mathcal{C}_i$ and our choice of $\mathcal{L}$, this gives us
\begin{align}
\sum_{\substack{d|P(z)\\ d<D_I}}^{U_1}\mu(d)|\mathcal{C}_{i;d}^\sharp(\gamma)|&=\sum_{\substack{d|P(z)\\ d<D_I}}^{U_1}\mu(d)|\mathcal{C}_{i;d}(\gamma)|+O\Bigl(\frac{x(\log x)^{13}}{\mathcal{L}}\Bigr)\nonumber\\
&\ll \frac{x}{(\log x)^2}.\label{decompIbound}
\end{align}
For the Type II related object we note 
\begin{align*}
\sum_{0\leq k<K} \sum_{y_{k+1}\leq p<y_k<q<z}|S(\mathcal{C}_{i;pq}^\flat(\gamma),y_k)|&\leq \sum_{pqn\leq x}|c_i^\flat(pqn,\gamma)|\nonumber \\
&\leq (\log x)^2 \sum_{n\leq x}|c_i^\flat(n,\gamma)|\nonumber \\
&\leq (\log x)^5 \sum_{k^2+l^2\leq x}\frac{\tau(l)^2}{\mathcal{L}}\nonumber \\
&\leq \frac{x (\log x)^8}{\mathcal{L}}. 
\end{align*}
Furthermore
\begin{align*}
\sum_{y_{k+1}\leq p<y_k<q<z}S(\mathcal{C}_{i;pq}(\gamma),y_k)
\end{align*}
can be interpreted as a Type II sum with coefficients and ranges as required for \eqref{gtypeII}. So we get by our assumed Type II bound 
\begin{align*}
&\sum_{0\leq k<K} \sum_{y_{k+1}\leq p<y_k<q<z}S(\mathcal{C}_{i;pq}^\sharp(\gamma),y_k)\\
\leq& \sum_{0\leq k<K} \bigl|\sum_{y_{k+1}\leq p<y_k<q<z}S(\mathcal{C}_{i;pq}(\gamma),y_k)\bigr|+O\bigl(\frac{x (\log x)^8}{\mathcal{L}} \bigr)\\
\ll& K \frac{x}{(\log x)^{24}}+\frac{x (\log x)^8}{\mathcal{L}}\\
\ll& \frac{x}{(\log x)^2}, 
\end{align*}
which together with \eqref{decompdfiapplied} and \eqref{decompIbound} is sufficient for \eqref{demessing2} and so completes the proof of Lemma \ref{decomp}.
\end{proof}

\section{The Major Arcs} \label{sec4}
In this section we prove that Proposition \ref{prmajo} holds for $\gamma\in \mathfrak{M}$. Here $\mathfrak{M}$ is the set of major arcs that is given by \eqref{Mganzdef}.

Let $f$ be an arithmetic function that fulfills for any fixed $A>0$ uniformly in $q\leq (\log x)^A$ the distribution law
\begin{align}\label{gvert}
\sum_{\substack{n\leq x\\ n\equiv a(q)}}f(n)=\frac{\Xi(q,a)}{\varphi(q)}\sum_{n\leq x}f(n)+O_A(x(\log x)^{-A})
\end{align}
and for any $\epsilon>0$ the bound 
\begin{align}\label{gvert2}
f(n)\ll_\epsilon x^\epsilon.
\end{align}
By Lemma \ref{majolem} and \eqref{o1bound}, \eqref{o2bound}, \eqref{o3bound},  this covers the majorant. The exponential sum we are interested in is given by
\begin{align*}
S(\gamma,T)=\sum_{n\leq T}f(Wn+b)e(\gamma n).
\end{align*}
We start by considering the case $\gamma=a/q$.
\begin{lemma}\label{marc1}
Let $x$, $W$, $N$, and $b$ as in Theorem \ref{MT}. Let further $A>0$, $T\leq N$, $S(\gamma,T)$ as above, and assume $f$ fulfills \eqref{gvert} and \eqref{gvert2}. It holds uniformly in $q\leq (\log x)^{A}$ and $a$ with $(a,q)=1$ that
\begin{align}\label{marc1eq}
S(a/q,T)=\frac{\epsilon(a,q,W,b)}{\varphi(q)}\frac{\Xi(W,b)}{\varphi(W)}\sum_{n\leq TW}f(n)+O_A(x(\log x)^{-A}),
\end{align}
where $\epsilon=1$ if $q=1$, $\epsilon=0$, if $(q,W)>1$ and $|\epsilon|\ll \tau(q)$ if $(q,W)=1$.
\end{lemma}
\begin{proof}
By sorting into residue classes modulo $q$, \eqref{gvert}, and \eqref{gvert2} we have
\begin{align*}
S(a/q,T)&=\sum_{n\leq T}f(Wn+b)e(\frac{an}{q})\\
&=\sum_{c(q)}e(\frac{ac}{q})\sum_{\substack{n\leq T\\ n\equiv c(q)}}f(Wn+b)\\
&=\sum_{c(q)}e(\frac{ac}{q})\sum_{\substack{n\leq TW+b\\ n\equiv Wc+b(Wq)}}f(n)\\
&=\frac{\sum_{n\leq TW}f(n)}{\varphi(Wq)}\sum_{c(q)}e(\frac{ac}{q})\Xi(Wq,Wc+b)+O_A(x (\log x)^{-A}).
\end{align*}
We now split the considerations depending on whether $q$ and $W$ have a common divisor. 

Let $(q,W)=1$. In that case $W$ is invertible mod $q$ and we denote its inverse by $\overline W$. Using the multiplicativity of $\Xi$ we get
\begin{align*}
\sum_{c(q)}e(\frac{ac}{q})\Xi(Wq,Wc+b)&=\Xi(W,b)\sum_{c(q)}e(\frac{ac}{q})\Xi(q,Wc+b)\\
&=\Xi(W,b)\sum_{c'(q)}e(\frac{a(c'-b)\overline W}{q})\Xi(q,c')\\
&=\Xi(W,b)e(\frac{-ab\overline W}{q})\sum_{c(q)}e(\frac{a\overline W c}{q})\Xi(q,c)
\end{align*}
We set 
\begin{align*}
\epsilon(a,q,W,b)=e(\frac{-ab\overline{W}}{q})\sum_{c(q)}e(\frac{a\overline W c}{q})\Xi(q,c)
\end{align*}
and get \eqref{marc1eq}. Since $(q,W)=1$, in particular $q$ is odd.  A short calculation using the multiplicativiy of $\Xi$ shows
\begin{align*}
|\epsilon(a,q,W,b)|&=\prod_{p|q}\bigl(1+O(p^{-1})\bigr)\\
&\ll \tau(q).
\end{align*} 

Let now $(q,W)>1$ and write $q=q_Wq_r$ with $$q_W=\prod_{\substack{p^\alpha||q\\p|W}}p^\alpha.$$ By sorting into classes modulo $q_W$ we have 
\begin{align}
\sum_{c(q)}e(\frac{ac}{q})\Xi(Wq,Wc+b)&=\sum_{\substack{c'(q_W)\\ d(q_r)}}e(\frac{a(c'q_r+dq_W)}{q})\Xi(Wq_rq_W,Wc'q_r+Wdq_W+b)\nonumber \\
&=\sum_{\substack{c'(q_W)\\ d(q_r)}}e(\frac{ad}{q_r})e(\frac{a c'}{q_W})\Xi(Wq_rq_W,Wc'q_r+Wdq_W+b) \label{c'sum}.
\end{align}
By construction $(q_r,q_W W)=1$ and for odd primes we have $p|W$ if and only if $p|q_W W$. We further recall that $4|W$. By using the properties of $\Xi$ given in Corollary \ref{cor} we have
\begin{align*}
\Xi(Wq_rq_W,Wc'q_r+Wdq_W+b)&=\Xi(q_r,Wdq_W+b)\Xi(W,b)
\end{align*}
and the right hand side is independent of $c'$. Consequently the sum over $c'$ in \eqref{c'sum} vanishes and we have $\epsilon(a,q,W,b)=0$ in this case.
\end{proof}

As usual we handle the remaining $\gamma\in \mathfrak{M}(q,a)$ by summation and integration by parts. This requires an asymptotics for $f(n)$ with good enough error term. We assume that for any fixed $A>0$ we have uniformly in $T\leq x$
\begin{align}\label{partsumasymp}
\sum_{n\leq T}f(n)=\int_0^{\sqrt{T}}\mathcal{F}(y)\sqrt{T-y^2}dy+O_A(x(\log x)^{-A})
\end{align}
for some function $\mathcal{F}(y)$ that is piece-wise continuously differentiable and so admissible for integration by parts. Again by Lemma \ref{majolem}, this covers all required cases for us. We set
\begin{align*}
S_0(\gamma,N)=\sum_{n\leq N}e(\gamma n)
\end{align*}
and prove the following result.
\begin{lemma}\label{marc2}
Let $x$, $W$, $N$ and $b$ as in Theorem \ref{MT}. Let further $A>0$ and 
assume $f$ fulfills the condition of Lemma \ref{marc1} and \eqref{partsumasymp}. We then have for $\gamma\in \mathfrak{M}(q,a)$ that
\begin{align*}
S(\gamma,N)=&\frac{\pi}{4}\frac{\epsilon(a,q,W,b)}{\varphi(q)} \frac{\Xi(W,b)W}{\varphi(W)} \mathcal{F}(\sqrt{x})S_0(\gamma-a/q,N)\\
&+O\Bigl(\frac{\Xi(W,b)W}{\varphi(W)}\bigl(\sqrt{WN}\int_0^{\sqrt{N}}t|(\mathcal{F}'(t\sqrt{W}))|dt+\int_0^{N+1} \sqrt{Wt}|(\mathcal{F}'(\sqrt{tW})|dt\bigr)\Bigr)\\
&+O\Bigl(\int_N^{N+1}|\mathcal{F}(\sqrt{tW})|dt\Bigr)+O_A\bigl(x(\log x)^{-A}\bigr).
\end{align*}

\end{lemma}
\begin{proof} We start by writing $\gamma=a/q+\beta$ and use summation by parts to get
\begin{align*}
&S(\gamma,N)\\
=&\sum_{n\leq N}f(Wn+b)e(na/q)e(\beta n)\\
=&e(\beta N)\sum_{n\leq N}f(Wn+b)e(na/q)-2 \pi \text{i} \beta \int_0^N e(\beta T)\sum_{n\leq T}f(Wn+b)e(na/q)dT.
\end{align*}
The previously proved Lemma \ref{marc1} can now be applied on both appearing sums to obtain
\begin{align*}
S(\gamma,N)=&\frac{\epsilon(a,q,W,b)\Xi(W,b)}{\varphi(q)\varphi(W)}\Bigl(e(\beta N)\sum_{n\leq WN}f(n)-2 \pi \text{i} \beta \int_0^N e(\beta T)\sum_{n\leq TW}f(n)dT \Bigr)\\
&+O_A\bigl(x(1+|\beta| N)(\log x)^{-A}\bigr).
\end{align*}
By the definition of $\mathfrak{M}(q,a)$ given in \eqref{Mqadef}, after choosing $A$ sufficiently large in terms of $A_\mathfrak{M}$, the error is acceptable for the Lemma. Using the assumption \eqref{partsumasymp} we have for $T\leq N$
\begin{align*}
\sum_{n\leq TW}f(n)&=\int_0^{\sqrt{TW}}\mathcal{F}(y)\sqrt{TW-y^2}dy+O_A\bigl(x(\log x)^{-A})\\
&=W\int_0^{\sqrt{T}}\mathcal{F}(\sqrt{W}y)\sqrt{T-y^2}dy+O_A\bigl(x(\log x)^{-A}).
\end{align*}
We get
\begin{align*}
&e(\beta N)\sum_{n\leq WN}f(n)-2 \pi \text{i} \beta \int_0^N e(\beta T)\sum_{n\leq TW}f(n)dT\\
=&We(\beta N)\int_0^{\sqrt{N}} \mathcal{F}(\sqrt{W}y)\sqrt{N-y^2}dy\\
&-2W\pi \text{i}\beta \int_0^N e(\beta T)\int_0^{\sqrt{T}}\mathcal{F}(\sqrt{W}y)\sqrt{T-y^2}dydT\Bigr)+O_A\bigl(x(1+|\beta| N)(\log x)^{-A}\bigr).
\end{align*}
The error is again admissible. Next we interchange the order of integration and rewrite  
\begin{align*}
&e(\beta N)\int_0^{\sqrt{N}} \mathcal{F}(\sqrt{W}y)\sqrt{N-y^2}dy-2\pi \text{i}\beta \int_0^N e(\beta T)\int_0^{\sqrt{T}}\mathcal{F}(\sqrt{W}y)\sqrt{T-y^2}dydT\\
=&\int_0^{\sqrt{N}}\mathcal{F}(\sqrt{W}y)\Bigl(e(\beta N)\sqrt{N-y^2}-2\pi \text{i}\beta \int_{y^2}^N e(\beta T)\sqrt{T-y^2}dT \Bigr)dy\\
=&\int_0^{\sqrt{N}}\mathcal{F}(\sqrt{W}y)\int_{y^2}^N\frac{e(\beta T)}{2\sqrt{T-y^2}}dTdy.
\end{align*}
We once more interchange order of integration to get
\begin{align*}
\int_0^{\sqrt{N}}\mathcal{F}(\sqrt{W}y)\int_{y^2 }^N\frac{e(\beta T)}{2\sqrt{T-y^2}}dTdy&=\int_{0}^Ne(\beta T)\int_0^{\sqrt{T}}\frac{\mathcal{F}(\sqrt{W}y)}{2\sqrt{T-y^2}}dydT.
\end{align*}
Since
\begin{align*}
\int_{0}^{\sqrt{T}}\frac{1}{2\sqrt{T-y^2}}dy=\frac{\pi}{4}
\end{align*}
and for $t\leq \sqrt{T}$ 
\begin{align*}
\int_{0}^t\frac{1}{2\sqrt{T-y^2}}dy=O(\frac{t}{\sqrt{T}}),
\end{align*}
we get by integration by parts
\begin{align*}
\int_0^{\sqrt{T}}\frac{\mathcal{F}(\sqrt{W}y)}{2\sqrt{T-y^2}}dy=\frac{\pi \mathcal{F}(\sqrt{TW})}{4}+\sqrt{W}\int_0^{\sqrt{T}}(\mathcal{F}'(\sqrt{W}t))O\bigl(\frac{t}{\sqrt{T}} \bigr)dt.
\end{align*}
This gives us
\begin{align*}
&\int_{0}^Ne(\beta T)\int_0^{\sqrt{T}}\frac{\mathcal{F}(\sqrt{W}y)}{2\sqrt{T-y^2}}dydT\\
=&\frac{\pi}{4}\int_0^N e(\beta T)\mathcal{F}(\sqrt{TW})dT+\sqrt{W}\int_0^N \int_0^{\sqrt{T}}(\mathcal{F}'(\sqrt{W}t))O\bigl(\frac{t}{\sqrt{T}} \bigr)dtdT\\
=&\frac{\pi}{4}\int_0^N e(\beta T)\mathcal{F}(\sqrt{TW})dT+O\Bigl(\sqrt{NW}\int_0^{\sqrt{N}}t\bigl|(\mathcal{F}'(\sqrt{W}t))\bigr|dt \Bigr).
\end{align*}
Recall that $N=\floor{\frac{x}{W}}$ and set $N'=\frac{x}{W}$. We then have
\begin{align*}
\int_1^N e(\beta T)\mathcal{F}(\sqrt{TW})dT=\int_0^{N'}e(\beta T)\mathcal{F}(\sqrt{TW})dT+O\bigl(\int_N^{N+1}|\mathcal{F}(\sqrt{Wt})|dt\bigr).
\end{align*}
We integrate one final time by parts to get
\begin{align*}
\int_0^{N'} e(\beta T)\mathcal{F}(\sqrt{TW})dT=\mathcal{F}(\sqrt{x})\int_0^{N'} e(\beta T)dT+O\Bigl(\int_0^{N+1} \sqrt{Wt} \bigl|(\mathcal{F}'(\sqrt{Wt})\bigr|dt\Bigr).
\end{align*}
As
\begin{align*}
\int_0^{N'} e(\beta T)dT&=\sum_{n\leq N}e(\beta n)+O(N|\beta|+1)\\
&=S_0(\gamma-a/q,N)+O((\log x)^{A_\mathfrak{M}}),
\end{align*}
this completes the proof.
\end{proof}
\subsection{Major Arc Pseudorandomness}
We conclude this section by showing that the majorant is pseudorandom on the major arcs.
\begin{lemma}\label{prmajomajor}
Proposition \ref{prmajo} holds for all $\gamma\in \mathfrak{M}$.
\end{lemma}
\begin{proof}
Let $\gamma\in \mathfrak{M}(q,a)$. 

We apply Lemma \ref{marc2} in conjunction with \eqref{majolem1} of Lemma \ref{majolem} to get
\begin{align} \nonumber
\sum_{n\leq N}\Lambda_{i;W,b}(n,x)e(\gamma n)=&\frac{ \epsilon(a,q,W,b)}{\varphi(q)}C_i(x)C_i'(x)S_0(\gamma-a/q,N)\\
&+O\Bigl(\sqrt{WN}\int_0^{\sqrt{N}}t|(\mathcal{F}_i'(t\sqrt{W},x))|dt+\int_0^{N+1} \sqrt{Wt}|(\mathcal{F}'_i(\sqrt{tW,x})|dt\Bigr)\label{eq_new} \\
&+O\Bigl(\frac{\varphi(W)}{\Xi(W,b)W}\int_N^{N+1}|\mathcal{F}_i(\sqrt{tW})|dt \Bigr)+O_A\bigl(x(\log x)^{-A}\bigr)\nonumber.
\end{align}
For $i=1,2$ we have $\mathcal{F}_i(t,x)=1$ and 
\begin{align*}
\mathcal{F}_3(t,x)= \int_{z_1\leq y_1<y_2<y_3<z}\frac{B(t/(y_1y_2y_3),y_1)}{y_1y_2y_3 \log y_1\log y_2 \log y_3}d\bm{y}.
\end{align*}
By Lemma \ref{buch} we have
\begin{align*}
\mathcal{F}_3'(t,x)&\leq  \frac{1}{t} \int_{z_1\leq y_1<y_2<y_3<z}\frac{B(t/(y_1y_2y_3),y_1)}{y_1y_2y_3 \log y_1\log y_2 \log y_3 \log \bigr(t/(y_1 y_2 y_3)\bigl)}d\bm{y}
\end{align*}
and so we can estimate the error terms in \eqref{eq_new} as 
\begin{align*}
O\Bigl(\frac{N}{\log N}\Bigr).
\end{align*}

Thus,
\begin{align}\label{marcpreq}
\sum_{n\leq N}\Lambda^+_{W,b}(n,x)e(\gamma n)=\frac{\epsilon(a,q,W,b)}{\varphi(q)}\alpha^+(x)S_0(\gamma-a/q,N)+O\Bigl(\frac{N}{\log N}\Bigr).
\end{align}
We note that since $\Lambda^+$ is an upper bound for the Fouvry Iwaniec primes, $\alpha^+(x)\gg 1$ and so, together with the bound \eqref{majolem3} and for sufficiently large $x$, we have $\alpha^+(x)=O(1)$.

Let first $\gamma\in \mathfrak{M}(1,0)$. We have then
\begin{align*}
\sum_{n\leq N}\Lambda^+_{W,b}(n,x)e(\gamma n)=\alpha^+(x) S_0(\gamma,N)+O\Bigl(\frac{N}{\log N}\Bigr),
\end{align*}
which is sufficient.

If  $\gamma\not \in \mathfrak{M}(1,0)$ we have
\begin{align*}
\alpha^+(x)S_0(\gamma,N)=o(N).
\end{align*}
To prove Proposition \ref{prmajo} in those cases it suffices to show
\begin{align*}
\sum_{n\leq N}\Lambda^+_{W,b}e(\gamma n)=o(N).
\end{align*}
Let $(q,W)\neq 1$. By Lemma \ref{marc1} we have $\epsilon=0$ in \eqref{marcpreq} and the statement follows. If $(q,W)=1$ we have
\begin{align*}
\Bigl|\frac{ \epsilon(a,q,W,b)}{\varphi(q)}\alpha^+(x)S_0(\gamma-a/q,N)\Bigr|\ll_\epsilon Nq^{-1+\epsilon}.
\end{align*}
We recall $W=2\prod_{p\leq 0.1 \log \log x}p$. Consequently $(q,W)=1$ implies $q\geq 0.1 \log \log x$ and this completes the proof.
\end{proof}

\section{The Minor Arcs}\label{sec5}
In this section minor arc bounds for Type I and II sums are proved that by Lemma \ref{decomp} are sufficient to show that Proposition \ref{prmajo} holds for $\gamma\in \mathfrak{m}$. The appearing Type I and II sums are of the form given in \eqref{gtypeI} and \eqref{gtypeII}. We begin with the harder case of Type II sums. The strategy is as follows. First we transfer the sum into the Gaussian integers, following the work of Fouvry and Iwaniec. We now introduce the required notation. For a Gaussian integer $m\in \mathbbm{Z}[i]$ with real part $m_R$ and imaginary part $m_I$ we have its squared norm \begin{align*}
|m|^2=m_R^2+m_I^2.
\end{align*}
Sums over Gaussian integers are written as sums over $|m|^2$ in this section. In particular, given a function $f: \mathbbm{Z}[i]\to \mathbbm{C}$, we study
\begin{align*}
\sum_{|m|^2\sim M}f(m),
\end{align*}
where the condition $|m|^2\sim M$ is to be interpreted as $M<  |m|^2\leq M'$. Here $M'$ is is a parameter in the range $M<M'\leq 2M$ whose value will be clear from the context. We call $m$ primitive, if $(m_R,m_I)=1$ and denote by
\begin{align*}
\sum_{|m|^2\sim M}^*f(m)
\end{align*}
the above sum restricted to primitive Gaussian integers. For $m,l\in \mathbbm{Z}[i]$ we write $m*l=\Re(m\overline{l})$. The central point of this section is a nontrivial estimation of
\begin{align}\label{basissumme}
\sum_{\substack{|m|^2\sim M\\ d_j|m*l_j}}e(\gamma |m|^2),
\end{align}
where $W\gamma\in \mathfrak{m}$ and for $j\in\{1,2\}$ we are given squarefree $d_j\in \mathbbm{N}$ and primitive $l_j\in \mathbbm{Z}[i]$. The divisibility condition is to be interpreted as holding for both $j$. This task becomes easier the smaller  $d_j$ in relation to $M$ are. We need to consider the case of $d_j$ being almost $\sqrt{M}$. The divisibility condition can be interpreted as $m$ lying in a two dimensional lattice. For that reason, before starting with the proof of the bound, in the next subsection we note some observations on the appearing lattices that help us estimating \eqref{basissumme}. Afterwards we prepare the Type II sums as described above and apply the lattice considerations. This gives us a sufficient bound for the Type II sums. At the end of this section we bound the Type I sums and both results are then combined to complete the proof of Proposition \ref{prmajo}. Throughout this section we do not aim to give the best bounds with respect to the number of appearing logarithms. In particular  we use Cauchy's inequality instead of Hölder's.

We remark that a different way of applying Cauchy's inequality would yield sums of the form
\begin{align*}
\sum_{\substack{|m|^2\sim M\\ d|m*l_j}}e(\gamma |m|^2),
\end{align*}
where only one $d$ appears. By the homogenity of $m*l_j$ the range of $d$ for which nontrivial summation is possible is similar to our approach. This also means that, in contrast to twin prime related problems (see for example Matomäki \cite{mbv}), we do not benefit from well-factorable sieve weights.

\subsection{Lattices}

Assume we are given for $j\in\{1,2\}$ primitive Gaussian integers $l_j\in \mathbbm{Z}[i]$ and square-free $d_j\in \mathbbm{N}$. We set
\begin{align*}
\Gamma(l_1,d_1,l_2,d_2)=\{m\in \mathbbm{Z}[i]: d_j|m*l_j \}.
\end{align*}
This is a two dimensional lattice whose discriminant we call
\begin{align*}
\Delta=\Delta(l_1,d_1,l_2,d_2).
\end{align*}
It can be calculated as follows. By the Chinese Remainder Theorem it is enough to understand the condition
\begin{align*}
p|m*l_i
\end{align*}
for $p|d_i$. If $p \nmid (d_1,d_2)$ this gives a contribution of $p$, since the $l_i$ are primitive. Let now $p|(d_1,d_2)$. For these primes we have two simultaneous congruence conditions
\begin{align*}
p&|m*l_1\\
p&|m*l_2.
\end{align*}
These condition can either define the same set of solutions or not. The first case happens if and only if 
\begin{align*}
p|\Im(l_1\overline{l_2})
\end{align*}
and then gives a contribution of $p$ to the discriminant. In the complementary case the only solutions are $m$ with $p|m_R$ and $p|m_I$. The contribution to the discriminant then is $p^2$. Combining these considerations we get
\begin{align}\label{discr}
\Delta=\frac{d_1d_2}{(d_1,d_2,|\Im(l_1\overline{l_2})|)}.
\end{align}
Of central interest is the intersection of lattices of the above form with an annulus 
\begin{align*}
\Gamma(M;l_1,d_1,l_2,d_2)=\{|m|^2\sim M: d_i|m*l_i \},
\end{align*}
which is the range of summation of \eqref{basissumme}. To sum over the lattice one can use a basis and a choice could look like
\begin{align*}
 \Bigl( \begin{pmatrix}
 0\\
 \Delta
 \end{pmatrix} , 
  \begin{pmatrix}
 1\\
 \tilde{l}
 \end{pmatrix}\Bigr),
\end{align*}
where $\tilde{l}$ depends on $l_i$ and $d_i$. This choice is sufficient as long as $\Delta$ is somewhat smaller than $\sqrt{M}$. However, we need a result for $\Delta$ being almost $M$. This can be accomplished by choosing a different basis for $\Gamma$.

In the two dimensional case we can choose a basis $(b_1,b_2)$ such that the length of $b_i$ is equal to the $i$-th successive minimum of $\Gamma$ (see for example Siegel's lecture notes on the geometry of numbers \cite[X.6]{sc}). We then have
\begin{align}\label{sucmini}
|b_1||b_2|\asymp \Delta.
\end{align}
Let $\mathfrak{B}$ be a function that maps every quadruplet $(l_1,d_1,l_2,d_2)$ to such a basis. Assume we are given $(b_1,b_2)=\mathfrak{B}(l_1,d_1,l_2,d_2)$. We want to parametrise $\Gamma(M;l_1,d_1,l_2,d_2)$ with this basis and define the following sets
\begin{align}
\mathfrak{L}(M,b_1,b_2)&=\{(\lambda_1,\lambda_2)\in \mathbbm{Z}^2: |\lambda_1b_1+\lambda_2b_2|^2\sim M \}\label{Ldef}\\
\mathfrak{L}_2(M,b_1,b_2)&=\pi_2(\mathfrak{L}(M,b_1,b_2))\label{L2def}\\
\mathfrak{L}_1(\lambda_2,M,b_1,b_2)&=\{\lambda_1:(\lambda_1,\lambda_2)\in \mathfrak{L}(M,b_1,b_2)\},\label{L1def}
\end{align}
where $\pi_2$ is the projection on the $2$nd coordinate. Clearly the number of points in the lattice annulus can be counted by
\begin{align*}
|\Gamma(M;l_1,d_1,l_2,d_2)|=|\mathfrak{L}(M,b_1,b_2)|=\sum_{\lambda_2\in \mathfrak{L}_2(M,b_1,b_2)}\sum_{\lambda_1\in \mathcal{L}_1(\lambda_2,M,b_1,b_2)}1,
\end{align*}
which is true for any basis. As the $b_i$ are successive minima, we also can bound the number of lattice points by
\begin{align}\label{latticepoints}
|\Gamma(M;l_1,d_1,l_2,d_2)|\ll \begin{cases} \sqrt{M}|b_1|^{-1}, &\text{if } |b_1|< \sqrt{M}\leq |b_2|\\
M\Delta^{-1}, &\text{if } |b_2|< \sqrt{M}.
\end{cases}
\end{align}
To see this, let $\Gamma'=\frac{1}{\sqrt{\Delta}}\Gamma$ be a rescaled version of $\Gamma$ that is of discriminant $1$. If $R_i$ denote the successive minima of $\Gamma'$ we have $R_i \sqrt{\Delta}=|b_i|$. Let $N(R,\Gamma')$ denote the number of points $\bm{x}$ of $ \Gamma'$ with $|\bm{x}|\leq R$. Then
\begin{align*}
|\Gamma(M;l_1,d_1,l_2,d_2)|=N(\sqrt{2M}/\sqrt{\Delta},\Gamma')-N(\sqrt{M}/\sqrt{\Delta},\Gamma')
\end{align*}
and \eqref{latticepoints} follows from Davenport \cite[Lemma 12.4]{dav}.

We note for later that we have the bounds
\begin{align}\label{L1}
|\mathfrak{L}_1(0,M,b_1,b_2)|\ll \frac{\sqrt{M}}{|b_1|}+1\\
|\mathfrak{L}_2(M,b_1,b_2)|\ll \frac{\sqrt{M}}{|b_2|}+1\label{L2},
\end{align}
and for any $\lambda_2$
\begin{align}\label{L3}
|\mathfrak{L}_1(\lambda_2,M,b_1,b_2)|\ll |\mathfrak{L}_1(0,M,b_1,b_2)|.
\end{align}
We conclude this subsection with a lemma that is useful to estimate the number of lattices of the type we consider to which certain vectors can belong.
\begin{lemma}\label{X} Let $k\in \mathbbm{N}$. It holds that
\begin{align*}
\sum_{\substack{|v_i|^2\sim V_i\\ v_1*v_2\neq 0}}\tau(|v_1*v_2|)^k\ll V_1V_2(\log V_1V_2)^{2^{2k-1}+5/2}.
\end{align*}
\end{lemma}
\begin{proof}
We write $v_{i,R}$ and $v_{i,I}$ for the real and imaginary part of $v_i$ and use this to overcount the sum by
\begin{align*}&\sum_{\substack{|v_i|^2\sim V_i\\ v_1*v_2\neq 0}}\tau(|v_1*v_2|)^k
\\
\leq& \sum_{\substack{|v_{1,R}|\leq \sqrt{2 V_1}\\ |v_{2,R}|\leq \sqrt{2 V_2}}}\text{ }\sum_{\substack{|v_{1,I}|\leq \sqrt{2 V_1}\\ \substack{|v_{2,I}|\leq \sqrt{2 V_2}\\ v_1*v_2\neq 0}}}\tau(|v_{1,R}v_{2,R}+v_{1,I}v_{2,I}|)^k\\=&\sum_{\substack{|c_i|\leq 2\sqrt{V_1V_2}\\c_1\neq -c_2}}\tau(|c_1+c_2|)^k \sum_{\substack{|v_{1,R}|\leq \sqrt{2 V_1}\\ \substack{|v_{2,R}|\leq \sqrt{2 V_2}\\v_{1,R}v_{2,R}=c_1}}}\text{ }\sum_{\substack{|v_{1,I}|\leq \sqrt{2 V_1}\\ |\substack{v_{2,I}|\leq \sqrt{2 V_2}\\v_{1,I}v_{2,I}=c_2}}}1\\
\ll& \sum_{\substack{c_i\leq 2\sqrt{V_1V_2}\\\substack{c_i\neq 0\\ c_1\neq -c_2}}}\tau(|c_1+c_2|)^k\tau(|c_1|)\tau(|c_2|)+(V_1\sqrt{V_2}+V_2\sqrt{V_1})(\log V_1V_2)^{2^{k+1}-1}.
\end{align*}
An application of Cauchy's inequality on the remaining sum with the standard divisor bound shows the proposed estimate.
\end{proof}
\subsection{Type II Sums}
To show minor arc pseudorandomness we now assume we are given $\gamma \in \mathfrak{m}$ and parameters $x$, $N$, $W$, and $b$ as stated in Theorem \ref{MT}. We further recall the definition of $\omega_i$ for $i\in \{1,2\}$ given in subsection \ref{subsecconst}. The Type II sums given by Lemma \ref{decomp} are of the form
\begin{align*}
\mathcal{R}_{II}(U_1,U_2,\mathcal{C}_i(\gamma))=\sum_{U_1\leq d<U_2}\beta_d\sum_{\substack{n\\(d,n)=1}}\alpha_n c_i(dn,\gamma),
\end{align*}
where for $n\equiv b(W)$ and $n\leq x$ 
\begin{align*}
c_i(n,\gamma)=e(\gamma (n-b)/W)\sum_{n=k^2+l^2}\omega_i(l)
\end{align*}
and $c_i(n,\gamma)=0$ else. To deal with the effect of the $W$ trick on the minor arc, we import the following estimate that is a slight variant of the standard techniques for minor arc bounds.
\begin{lemma}\label{minsumW}
Let $K$ and $q$ be positive integers and suppose that $q\geq 100K$. Let $(a,q)=1$ and suppose that $|\theta-a/q|\leq Kq^{-2}$. It holds that
\begin{align*}
\sum_{1\leq n \leq X}\min\{Y,||\theta n||^{-1} \}\ll  K \bigl(\frac{XY}{q}+Y+(X+q)\log q \bigr)
\end{align*}
\end{lemma}
\begin{proof}
This is \cite[Lemma 1]{greb}.
\end{proof}

\begin{lemma}\label{T2boundl}
Let $U_1$, $U_2$ as in Lemma \ref{decomp} and $\gamma \in \mathfrak{m}$. We have for $i\in \{1,2\}$ 
\begin{align} \label{T2bound}
\mathcal{R}_{II}(U_1,U_2,\mathcal{C}_i(\gamma))\ll x(\log x)^{-24}.
\end{align}
\end{lemma}
\begin{proof}
We show the estimate for arbitrary sequences of complex numbers with $|\alpha_n|\leq \nu(m)$, $|\beta_d|\leq 1$ and $\beta$ supported on primes. To simplify the tracking of the appearing powers of $\log x$ we use the notation $\mathcal{L}$. We start by setting 
\begin{align*}
\mathcal{L}_1=(\log x)^{29}.
\end{align*}

The proof is split into three parts. In the first part we transfer the sums into the Gaussian integers and trivially bound the contribution of certain terms. The second part consists in applying the lattice considerations of the previous subsection and Weyl differencing. Finally, in the third part, we step back into the rational integers and apply estimates that are standard for minor arc bounds.

\subsubsection*{Step 1 - Gaussian Integers and Preparation}

Starting point for the proof of \eqref{T2bound} is the transference of $\mathcal{R}_{II}$ into the Gaussian integers. We closely follow chapter 8 of \cite{foi} and for parameters $M$ and $L$ set $M'=e^{\mathcal{L}_1^{-1}}M$ and $L'=2L$. Recall that $d\sim L$ denotes the range $L<d \leq L'$ and similarly for other parameters. We introduce the short sums
\begin{align*}
\mathcal{R}_1(L,M,i)=\sum_{d\sim L}\sum_{\substack{n\sim M\\(d,n)=1}}\beta_d\alpha_n c_i(dn,\gamma).
\end{align*}
Similar to the work of Fouvry and Iwaniec, we use these sums for $M=e^{\frac{k}{\mathcal{L}_1}}x/U_2$ and $L=2^jU_1$. We get 
\begin{align*}
\mathcal{R}_{II}(U_1,U_2,\mathcal{C}_i(\gamma))\leq \sum_{\substack{x/\mathcal{L}_1<ML<x \\ U_1\leq L < U_2}}|\mathcal{R}_1(L,M,i)|+O(\frac{x(\log x)^{5}}{\mathcal{L}_1}),
\end{align*}
where the error term comes from $mn\leq 2x/\mathcal{L}_1$ or $e^{-2/\mathcal{L}_1}x<mn\leq x$, which are not covered exactly. There are fewer than $2\mathcal{L}_1 (\log x)^2$ short sums in $\mathcal{R}_{II}$, so it suffices to show
\begin{align*}
\mathcal{R}_1(L,M,i)\ll \frac{x(\log x)^3}{\mathcal{L}_1^2}
\end{align*}
for all sequence $\alpha$, $\beta$ as before and all
\begin{align*}
U_1&\leq L < U_2\\
x\mathcal{L}_1^{-1}&<ML\leq x.
\end{align*}
We continue following the argument in \cite{foi} and get for $(d,n)=1$ and $dn\equiv b(W)$ 
\begin{align*}
c_i(dn,\gamma)=e(-\frac{\gamma b}{W})\frac{1}{4}\sum_{|l|^2=d}\sum_{|m|^2=n}\omega_i(m*l)e( |m|^2 |l|^2 \gamma/W),
\end{align*}
Thus it holds that
\begin{align*}
4 |\mathcal{R}_1(L,M,i)|= \bigl|\sum_{|l|^2\sim L}\sum_{\substack{|m|^2\sim M \\ \substack{(|l|^2,|m|^2)=1\\ |m|^2|l|^2\equiv b(W)}}}\beta_{|l|^2}\alpha_{|m|^2}\omega_i(m*l)e( |m|^2 |l|^2 \gamma/W) \bigr|.
\end{align*}
We now remove the condition $|m|^2|l|^2\equiv b(W)$ by noting that
\begin{align*}
4 |\mathcal{R}_1(L,M,i)|= \bigl|\sum_{\substack{a_1,a_2(W)\\a_1a_2\equiv b(W)}}\sum_{\substack{|l|^2\sim L\\ |l|^2\equiv a_1(W)}}\sum_{\substack{|m|^2\sim M \\ \substack{(|l|^2,|m|^2)=1\\ |m|^2\equiv a_2(W)}}}\beta_{|l|^2}\alpha_{|m|^2}\omega_i(m*l)e(|m|^2 |l|^2 \gamma/W)\bigr|.
\end{align*}
We consider fixed values of $a_1$, $a_2$, and absorb the remaining congruence conditions into the sequences $\alpha$ and $\beta$. This leads to\begin{align*}
\mathcal{R}_2(L,M,i):=\sum_{\substack{|l|^2\sim L\\ }}\sum_{\substack{|m|^2\sim M \\ \substack{(|l|^2,|m|^2)=1\\ }}}\beta_{|l|^2}\alpha_{|m|^2}\omega_i(m*l)e(|m|^2 |l|^2 \gamma/W).
\end{align*}
Since $W\leq \log x$, it is enough to show
\begin{align*}
\mathcal{R}_2(L,M,i)\ll \frac{x(\log x)^2}{\mathcal{L}_1^2}
\end{align*}
in the same range and for the same class of sequences $\alpha$ and $\beta$ as before. The next step is to remove the condition $(|l|^2,|m|^2)=1$. In our case this can be easily done, because $\beta$ is supported on primes only. We use Cauchy's inequality to obtain the estimate
\begin{align*}
&\sum_{|l|^2\sim L}\sum_{\substack{|m|^2\sim M\\|l|^2| |m|^2}}|\beta_{|l|^2} \alpha_{|m|^2} \omega_i(m*l)|\\
\ll& (\log x)\Bigl(\sum_{|l|^2\sim L}|\beta_{|l|^2}|^2\sum_{\substack{|m|^2\sim M\\|l|^2| |m|^2}}1\Bigr)^{1/2} \Bigl(\sum_{|l|^2\sim L}\sum_{\substack{|m|^2\sim M\\}}|\omega_i(m*l)|^2 \Bigr)^{1/2}\\ 
\ll& (\log x)^2 \sqrt{M}\Bigl(\sum_{|l|^2\sim L}\sum_{\substack{|m|^2\sim M\\}}|\omega_i(m*l)|^2 \Bigr)^{1/2}.
\end{align*}
By using the bounds \eqref{o1bound}, \eqref{o2bound}, and Lemma \ref{X} on the remaining sum, we get
\begin{align*}
\sum_{|l|^2\sim L}\sum_{\substack{|m|^2\sim M\\|l|^2| |m|^2}}|\beta_{|l|^2} \alpha_{|m|^2} \omega_i(m*l)|\ll \sqrt{L}M (\log x)^9.
\end{align*}
Further, since only primitive $l$ can give prime $|l|^2$, we have
\begin{align*}
\mathcal{R}_2(L,M,i)=\sum_{\substack{|l|^2\sim L\\ }}^*\sum_{\substack{|m|^2\sim M \\ \substack{\\ }}}\beta_{|l|^2}\alpha_{|m|^2}\omega_i(m*l)e(|m|^2 |l|^2 \gamma/W)+O\bigl( \sqrt{L}M (\log x)^9 \bigr).
\end{align*}
The error is admissible, because $L\geq U_1=x^{(\log \log x)^{-3}}$ and $LM\leq x$. This completes the transfer into Gaussian integers and also the part of our argument that is analogous to the work of Fouvry and Iwaniec. 

We note that the $\omega_i$ are supported on nonzero integers only, we can thus include the condition $m*l\neq 0$ in the above sum. To open the functions $\omega_i$ we recall 
\begin{align*}
\omega_1(l)=\theta^+(l,D_1,D_0,z_1,z_0)\log \sqrt{x},
\end{align*}
where $\theta^+$ is a sieve of level 
\begin{align*}
D_1D_0&=x^{1/3-10^{-7}}e^{(\log x)^{2/3}}\\
&\ll x^{1/3-10^{-8}}.
\end{align*}
Furthermore we have for $l>0$
\begin{align*}
\omega_2(l)=\frac{\log \sqrt{x}}{2}\sum_{\substack{z_1\leq p <z\\ l=pm}}\theta^-(m,m,D_1/p,D_0,z_1,z_0)
&=\frac{\log \sqrt{x}}{2}\sum_{\substack{z_1\leq p <z\\ l=pm}}\sum_{\substack{d\leq D_1D_0/p\\ d|m}}\lambda_{d}^-\\
&=\frac{\log \sqrt{x}}{2}\sum_{\substack{d'\leq D_1D_0\\ d'|l}}\lambda'_{d'}
\end{align*}
for some coefficients $|\lambda'_{d'}|\leq \log d'$. We now discard the condition that the sieves are supported on positive integers only. To get the required estimate for $\mathcal{R}_2$ it is consequently sufficient to show
\begin{align}\label{dsumdissect}
\sum_{\substack{|l|^2\sim L\\ }}^*\sum_{\substack{|m|^2\sim M \\ \substack{m*l\neq 0\\ }}}\beta_{|l|^2}\alpha_{|m|^2}\sum_{\substack{d\leq D_1D_0\\ d|m*l}}\lambda_d e(|m|^2 |l|^2 \gamma/W)\ll \frac{x \log x}{\mathcal{L}_1^2}
\end{align}
for all sequences of complex numbers with $|\lambda_d|\leq \log d$. With these considerations we no longer have to discern between the cases $i=1$ and $i=2$. We dissect the sum over $d$ in \eqref{dsumdissect} into fewer than $\log x$ dyadic intervals $D< d\leq D'=2D$ and set
\begin{align*}
\mathcal{R}_3(D,L,M)=\sum_{\substack{|l|^2\sim L\\ }}^*\sum_{\substack{|m|^2\sim M \\ \substack{m*l\neq 0\\ }}}\beta_{|l|^2}\alpha_{|m|^2}\sum_{\substack{d\sim D\\  d|m*l}}\lambda_d e(|m|^2 |l|^2 \gamma/W).
\end{align*}
We have to show
\begin{align*}
|\mathcal{R}_3(D,L,M)|\ll \frac{x}{\mathcal{L}_1^2 } 
\end{align*}
for all 
\begin{align*}
D\ll x^{1/3-10^{-8}}
\end{align*}
and in the same range of $L$ and $M$ as before. An application of Cauchy's inequality shows
\begin{align*}
\bigl|\mathcal{R}_3(D,L,M)\bigr|^2&\leq \Bigl(\sum_{|m|^2\sim M}|\alpha_{|m|^2}|^2\Bigr)\times \Bigl(\sum_{|m|^2\sim M}\bigl|\sum_{\substack{|l|^2\sim L\\ m*l\neq 0}}^*\beta_{|l|^2} \sum_{\substack{d\sim D\\  d|m*l}}\lambda_d e(|m|^2 |l|^2 \gamma/W) \bigr|^2 \Bigr)\\
&=\Bigl(\sum_{|m|^2\sim M}|\alpha_{|m|^2}|^2\Bigr)\times\mathcal{S}_1(D,L,M),
\end{align*}
say. As $|\alpha|^2\ll (\log x)^2$ our goal is now to estimate
\begin{align*}
|\mathcal{S}_1(D,L,M)|\ll \frac{M L^2}{\mathcal{L}_1^4 (\log x)^2}.
\end{align*}
We open the square in the definition of $\mathcal{S}_1$ and from now on use the index $j$ to denote the two copies of the variables $l$ and $d$. We have
\begin{align*}
\mathcal{S}_1(D,L,M)=\sum_{\substack{|l_j|^2\sim L\\}}^*\beta_{|l_1|^2}\overline{\beta_{|l_2|^2}} \sum_{\substack{d_j\sim D\\  }}\lambda_{d_1}\overline{\lambda_{d_2}} \sum_{\substack{|m|^2\sim M \\ \substack{m*l_j\neq 0 \\ d_j|m*l_j }}}e( |m|^2[|l_1|^2-|l_2|^2]\gamma/W).
\end{align*}
For convenience of notation we from now on write $\xi=[|l_1|^2-|l_2|^2]\gamma/W$. Since $|\beta|\leq 1$ and $|\lambda|\leq \log x$ the triangle inequality gives us 
\begin{align*}
\bigl|\mathcal{S}_1(D,L,M)\bigr|\leq (\log x)^2 \sum_{\substack{|l_j|^2\sim L\\ \\ } }^* \sum_{d_j\sim D}\bigl|\sum_{\substack{|m|^2\sim M \\  \substack{d_j|m*l_j \\ m*l_j\neq 0 } }} e(\xi|m|^2) \bigr|.
\end{align*}
The innermost sum is of the type that was considered in the last subsection. Before we can apply the lattice considerations, we prepare the sums by removing the contribution of terms that are not suitable for this. We note the trivial bound
\begin{align*}
\bigl|\mathcal{S}_1(D,L,M)\bigr|&\leq  (\log x)^2 \sum_{\substack{|l_j|^2\sim L\\ \\ } }^*\sum_{\substack{|m|^2\sim M \\  \substack{\\ m*l_j\neq 0 } }}\tau(|m*l_1|)\tau(|m*l_2|)\\ 
\end{align*}
that can be combined with Cauchy's inequality and Lemma \ref{X}. We start by removing diagonal terms with $|l_1|^2=|l_2|^2$ from $\mathcal{S}_1$. We have
\begin{align*}
(\log x)^2 \sum_{\substack{|l_j|^2\sim L\\ |l_1|^2=|l_2|^2 \\ } }^*\sum_{\substack{|m|^2\sim M \\  \substack{\\ m*l_j\neq 0 } }}\tau(|m*l_1|)\tau(|m*l_2|)\leq 4 (\log x)^2\sum_{\substack{|l|^2\sim L\\ \substack{|m|^2\sim M \\ m*l\neq 0 } }}\tau(|m*l|)^2.
\end{align*}
By Lemma \ref{X} and the range condition of $L$ this is negligible. We set
\begin{align}\label{S2}
\mathcal{S}_2(D,L,M)=\sum_{\substack{|l_j|^2\sim L\\ |l_1|^2\neq |l_2|^2 \\ } }^* \sum_{d_j\sim D}\bigl|\sum_{\substack{|m|^2\sim M \\  \substack{d_j|m*l_j \\ m*l_j\neq 0 } }} e(\xi|m|^2) \bigr|
\end{align}
and it is sufficient to show the bound
\begin{align*}
\mathcal{S}_2(D,L,M)\ll \frac{M L^2}{\mathcal{L}_1^4 (\log x)^4}.
\end{align*}
in the same range of $D$, $L$, and $M$ as before. We expect the sum over $m$ in \eqref{S2} to range over $\asymp M/D^2$ many elements, but this is not true pointwise for every quadruplet $(l_1,d_1,l_2,d_2)$. There are two reasons for this. One obstacle is a large value of $(d_1,d_2,|\Im(l_1\overline{l_2})|)$. The second one is the possibility that the second successive minimum of the associated lattice $\Gamma(l_1,d_1,l_2,d_2)$ is of size larger than $\sqrt{M}$. We now deal with the first hurdle. For a technical reason we do more than required and remove terms for which $(d_j,\Im(l_1\overline{l_2}))$ is large for at least one $j$. We start with $j=1$. To do this, we temporarily remove terms for which 
\begin{align*}
\tau(|m*l_1|)>\mathcal{L}_2.
\end{align*}
We have
\begin{align*}
\sum_{\substack{|l_j|^2\sim L\\ |l_1|^2\neq |l_2|^2 \\ } }^* \sum_{d_j\sim D}\bigl|\sum_{\substack{|m|^2\sim M \\  \substack{d_j|m*l_j \\ \substack{ m*l_j\neq 0 \\ \tau(|m*l_1|)> \mathcal{L}_2} } }} e(\xi|m|^2) \bigr|& \leq \sum_{\substack{|l_j|^2\sim L\\ |l_1|^2\neq |l_2|^2 \\ } }^*\sum_{\substack{|m|^2\sim M \\  \substack{m*l_j\neq 0  \\ \tau(|m*l_1|)>\mathcal{L}_2 } }}\tau(|m*l_1|)\tau(|m*l_2|) \\
&< \mathcal{L}_2^{-1}\sum_{\substack{|l_j|^2\sim L\\ |l_1|^2\neq |l_2|^2 \\ } }^*\sum_{\substack{|m|^2\sim M \\  \substack{m*l_j\neq 0  \\ } }}\tau(|m*l_1|)^2\tau(|m*l_2|)\\
&\ll \frac{M L^2 (\log x)^{71} }{\mathcal{L}_2}
\end{align*}
by Cauchy's inequality and Lemma \ref{X} again. The contribution of terms with $$(d_1,\Im(l_1\overline{l_2}))=k$$ among the remaining ones with $\tau(|m*l_1|)\leq \mathcal{L}_2$ can be estimated by
\begin{align*}
\sum_{\substack{|l_j|^2\sim L\\ |l_1|^2\neq |l_2|^2 \\ } }^* \sum_{\substack{d_j\sim D\\ (d_1,\Im(l_1\overline{l_2}))=k}}\bigl|\sum_{\substack{|m|^2\sim M \\  \substack{d_j|m*l_j \\ \substack{ m*l_j\neq 0 \\\tau(|m*l_1|)\leq\mathcal{L}_2}} }} e(\xi|m|^2) \bigr|&\leq \sum_{\substack{|l_j|^2\sim L\\ k| \Im(l_1 \overline{l_2}) \\ } }^* \sum_{\substack{|m|^2\sim M \\  \substack{ \\ \substack{ m*l_j\neq 0 \\ k|m*l_1}} }}\mathcal{L}_2^2 .
\end{align*}
Here the sum over $l_j$ is empty, if $k>4L$. Since $L\leq \sqrt{M}$, by the lattice considerations of the previous subsection, the sum over $m$ ranges over $O(M/k)$ elements. We get
\begin{align*}
\sum_{\substack{|l_j|^2\sim L\\ k| \Im(l_1 \overline{l_2}) \\ } }^* \sum_{\substack{|m|^2\sim M \\  \substack{ \\ \substack{ m*l_j\neq 0 \\ k|m*l_1}} }}\mathcal{L}_2^2
&\ll \frac{\mathcal{L}_2^2 M}{k}\sum_{\substack{|l_j|^2\sim L\\k|\Im(l_1\overline{l_2})}}1\\
&\ll \frac{\mathcal{L}_2^2 M}{k}\sum_{\substack{0\neq |c_j|\leq 2L\\ c_1\equiv c_2(k)}}\tau(|c_1|)\tau(|c_2|)+\frac{\mathcal{L}_2^2M L^{3/2}}{k}\\
&\ll_\epsilon \frac{\mathcal{L}_2^2ML^2 (\log x)^{2}}{k^2}+\mathbbm{1}_{k\geq L^{1/10}}\frac{ML^{2+\epsilon} }{k^2}+\frac{\mathcal{L}_2^2M L^{3/2}}{k},
\end{align*}
where we applied Shiu's Theorem (see \cite{shiu}) in the range $k\leq L^{1/10}$ and estimated the divisor function by $O(L^\epsilon)$ in the remaining range. Summing over $k$ in the range $\mathcal{L}_3<k\ll L$ shows
\begin{align*}
\sum_{\substack{|l_j|^2\sim L\\ |l_1|^2\neq |l_2|^2 \\ } }^* \sum_{\substack{d_j\sim D\\ (d_1,\Im(l_1\overline{l_2}))>\mathcal{L}_3}}\bigl|\sum_{\substack{|m|^2\sim M \\  \substack{d_j|m*l_j \\ \substack{ m*l_j\neq 0 \\\tau(|m*l_1|)\leq \mathcal{L}_2}} }} e(\xi|m|^2) \bigr|\ll \frac{\mathcal{L}_2^2 M L^2 (\log x)^2}{\mathcal{L}_3}+\mathcal{L}_2^2 M L^{3/2} \log x.
\end{align*}
Plugging in these results in the definition of $\mathcal{S}_2$ \eqref{S2} we arrive at
\begin{align*}
\mathcal{S}_2(D,L,M)=&\sum_{\substack{|l_j|^2\sim L\\ |l_1|^2\neq |l_2|^2 \\ } }^* \sum_{\substack{ d_j\sim D\\ (d_1,\Im(l_1\overline{l_2}))\leq \mathcal{L}_3}}\bigl|\sum_{\substack{|m|^2\sim M \\  \substack{d_j|m*l_j \\ \substack{ m*l_j\neq 0 \\\tau(|m*l_1|)\leq \mathcal{L}_2}} }} e(\xi|m|^2) \bigr|\\
&+O\bigl(\frac{\mathcal{L}_2^2 M L^2 (\log x)^2}{\mathcal{L}_3}+\mathcal{L}_2^2 M L^{3/2} \log x+\frac{M L^2 (\log x)^{71} }{\mathcal{L}_2}\bigr).
\end{align*}
We now set 
\begin{align*}
\mathcal{L}_2&=\mathcal{L}_1^4 (\log x)^{75}=(\log x)^{191}\\
\mathcal{L}_3&=\mathcal{L}_2^2 \mathcal{L}_1^4 (\log x)^6=(\log x)^{504}
\end{align*}
and the error term becomes admissible. Completely analogously as above we can also remove the contribution of terms with $(d_2,\Im(l_1\overline{l_2}))>\mathcal{L}_3$ or with $(d_j,\Im(l_1l_2))>\mathcal{L}_3$, which is useful in a moment. After readding the missing terms with $\tau(|m*l_j|)>\mathcal{L}_2$ in the same way they were removed, we get
\begin{align*}
\mathcal{S}_2(D,L,M)&=\sum_{\substack{|l_j|^2\sim L\\ |l_1|^2\neq |l_2|^2 \\ } }^* \sum_{\substack{ d_j\sim D\\ \substack{(d_j,\Im(l_1\overline{l_2}))\leq \mathcal{L}_3\\(d_j,\Im(l_1l_2))\leq \mathcal{L}_3}}}\bigl|\sum_{\substack{|m|^2\sim M \\  \substack{d_j|m*l_j \\ m*l_j\neq 0 } }} e(\xi|m|^2) \bigr|+O(\frac{M L^2}{\mathcal{L}_1^4 (\log x)^4})\\
&=\mathcal{S}_3(D,L,M)+O\bigl(\frac{M L^2}{\mathcal{L}_1^4 (\log x)^4}\bigr),
\end{align*}
say. 

We do not need the condition $m*l_j\neq 0$ anymore and it is hindering later, so we remove it by noting that in $\mathcal{S}_3$ there are no terms with $m*l_1=m*l_2=0$, since the $l_j$ are primitive and $|l_1|^2\neq |l_2|^2$. Let us consider the contribution of terms with 
\begin{align*}
m*l_1&=0\\
m*l_2&\neq 0.
\end{align*}
In that case $m=c i \overline{l_1}$ for some $c\in \mathbbm{Z}_{\neq 0}$, so we have
\begin{align*}
\sum_{\substack{|l_j|^2\sim L\\ |l_1|^2\neq |l_2|^2 \\ } }^* \sum_{\substack{ d_j\sim D\\ \substack{(d_j,\Im(l_1\overline{l_2}))\leq \mathcal{L}_3\\(d_j,\Im(l_1l_2))\leq \mathcal{L}_3}}}\bigl|\sum_{\substack{|m|^2\sim M \\  \substack{d_j|m*l_j \\ \substack{m*l_1= 0 \\ m*l_2\neq 0}} }} e(\xi|m|^2) \bigr|& \leq D \sum_{\substack{|l_j|^2\sim L\\ |l_1|^2\neq |l_2|^2 \\ } }^* \sum_{0<c\leq \sqrt{M/L}}\tau(c (i \overline{l_1})*l_2)\\
&\ll \frac{M L^2}{\mathcal{L}_1^4 (\log x)^4},
\end{align*}
by Cauchy's inequality and our range conditions on $D$, $M$, and $L$. This gives us 
\begin{align*}
\mathcal{S}_3(D,L,M)=\sum_{\substack{|l_j|^2\sim L\\ |l_1|^2\neq |l_2|^2 \\ } }^* \sum_{\substack{ d_j\sim D\\ \substack{(d_j,\Im(l_1\overline{l_2}))\leq \mathcal{L}_3\\(d_j,\Im(l_1l_2))\leq \mathcal{L}_3}}}\bigl|\sum_{\substack{|m|^2\sim M \\  \substack{d_j|m*l_j \\ } }} e(\xi|m|^2) \bigr|+O\bigl(\frac{M L^2}{\mathcal{L}_1^4 (\log x)^4}\bigr)
\end{align*}
and concludes the preparation.
\subsubsection*{Part 2 - Lattice and Weyl Differencing}
We are now ready to apply the lattice considerations on the sum over $m$. Let $\mathfrak{B}$ and $\mathfrak{L}$ \eqref{Ldef} as in the previous subsection. We have
\begin{align*}
&\sum_{\substack{|l_j|^2\sim L\\ |l_1|^2\neq |l_2|^2 \\ } }^* \sum_{\substack{ d_j\sim D\\ \substack{(d_j,\Im(l_1\overline{l_2}))\leq \mathcal{L}_3\\(d_j,\Im(l_1l_2))\leq \mathcal{L}_3}}}\bigl|\sum_{\substack{|m|^2\sim M \\  \substack{d_j|m*l_j \\ } }} e(\xi|m|^2) \bigr|\\=&\sum_{\substack{|l_j|^2\sim L\\ |l_1|^2\neq |l_2|^2 \\ } }^* \sum_{\substack{ d_j\sim D\\ \substack{(d_j,\Im(l_1\overline{l_2}))\leq \mathcal{L}_3\\(d_j,\Im(l_1l_2))\leq \mathcal{L}_3}}} \sum_{\mathfrak{B}(l_j,d_j)=(b_1,b_2)}\bigl|\sum_{\substack{m=\lambda_1 b_1+\lambda_2 b_2\\  \substack{ (\lambda_1,\lambda_2)\in \mathfrak{L}(M,b_1,b_2) \\  } }} e(\xi|m|^2) \bigr|.
\end{align*}
We split the sum over $b_1$ into intervals of the form $B< |b_1|^2\leq B'=2B$ and write
\begin{align*}
\mathcal{S}_4(B,D,L,M)=\sum_{\substack{|l_j|^2\sim L\\ |l_1|^2\neq |l_2|^2 \\ } }^* \sum_{\substack{ d_j\sim D\\ \substack{(d_j,\Im(l_1\overline{l_2}))\leq \mathcal{L}_3\\(d_j,\Im(l_1l_2))\leq \mathcal{L}_3}}} \sum_{\substack{\mathfrak{B}(l_j,d_j)=(b_1,b_2)\\ |b_1|^2\sim B}}\bigl|\sum_{\substack{m=\lambda_1 b_1+\lambda_2 b_2\\  \substack{ (\lambda_1,\lambda_2)\in \mathfrak{L}(M,b_1,b_2) \\ } }} e(\xi|m|^2) \bigr|.
\end{align*}
Since there are fewer than $\log x$ such sums required, our goal has become to show
\begin{align}\label{s4bound}
\mathcal{S}_4(B,D,L,M)\ll \frac{ML^2}{\mathcal{L}_1^4 (\log x)^{5}}
\end{align}
for all
\begin{align*}
B\leq D^2
\end{align*}
in the same range of $L$, $M$, and $D$ as before. We want to later change order of summation and for a fixed $j$ estimate 
\begin{align*}
\sum_{\substack{d\sim D\\ d|b_1*l_j}}1
\end{align*}
by $\tau(|b_1*l_1|)$. To enable this we now show that, with our intended size of parameters, there are no terms in $\mathcal{S}_4$ with $b_1*l_j=0$ for any $j$. Similar as above, we start by noting that there are no terms with
\begin{align*}
b_1*l_1=b_1*l_2=0,
\end{align*}
as the primitivity of $l_j$ would again imply $|l_1|^2=|l_2|^2$ in that case. We now consider the case
\begin{align*}
b_1*l_1&=0\\
b_1*l_2&\neq 0.
\end{align*}
We then have for some $c\in \mathbbm{Z}_{\neq 0}$
\begin{align*}
b_1=ci\overline{l_1}.
\end{align*}
By the summation range of $b_1$ it necessarily holds that
\begin{align}\label{cineq1}
c\leq \frac{2 D}{\sqrt{L}}.
\end{align}
Furthermore, by weakening the condition of $b_1$ being the shortest vector of the lattice to $b_1$ being any element of it, we have
\begin{align}\label{dcong}
d_2|c[(i\overline{l_1})*l_2].
\end{align}
Again by the primitivity of the $l_i$ and the range of summation only including terms with $|l_1|^2\neq |l_2|^2$ we have
\begin{align*}
(i\overline{l_1})*l_2\neq 0.
\end{align*}
Since $c\neq 0$ we see that \eqref{dcong} implies
\begin{align}\label{cineq2}
c\geq \frac{d_2}{([(i\overline{l_1})*l_2],d_2)}.
\end{align}
Combining \eqref{cineq1} and \eqref{cineq2}, we get that the sum over $b_1$ with the above conditions is empty except if
\begin{align*}
([(i\overline{l_1})*l_2],d_2)\geq \frac{\sqrt{L}}{2}.
\end{align*}
Note that $(i\overline{l_1})*l_2=-\Im(l_1 l_2)$. We only are summing over terms for which
\begin{align*}
(d_2,\Im(l_1 l_2))\leq \mathcal{L}_3,
\end{align*}
so as long as
\begin{align}\label{tau_1bed}
\mathcal{L}_3 < \frac{\sqrt{L}}{2}
\end{align}
there are no terms with $b_1*l_1=0$ in $\mathcal{S}_3$. By the range condition on $L$, \eqref{tau_1bed} is true for all sufficiently large $x$. By symmetry we can also deal with the case $b_1*l_2=0$, so for all sufficiently large $x$ it holds that
\begin{align*}
\mathcal{S}_4(B,D,L,M)=\sum_{\substack{|l_j|^2\sim L\\ |l_1|^2\neq |l_2|^2 \\ } }^* \sum_{\substack{ d_j\sim D\\ \substack{(d_j,\Im(l_1\overline{l_2}))\leq \mathcal{L}_3\\(d_j,\Im(l_1l_2))\leq \mathcal{L}_3}}}  \sum_{\substack{\mathfrak{B}(l_j,d_j)=(b_1,b_2)\\ \substack{|b_1|^2\sim B\\b_1*l_j\neq 0 }}}\bigl|\sum_{\substack{m=\lambda_1 b_1+\lambda_2 b_2\\  \substack{ (\lambda_1,\lambda_2)\in \mathfrak{L}(M,b_1,b_2) \\} }} e(\xi|m|^2) \bigr|.
\end{align*} 
In the next steps the summation conditions on $l_j$, $d_j$ and $(b_1,b_2)$ do not change. To simplify the notation we write $\sum'$ for these sums, so that
\begin{align}\label{s4}
\mathcal{S}_4(B,D,L,M)=\sideset{}{'}\sum \bigl|\sum_{\substack{m=\lambda_1 b_1+\lambda_2 b_2\\  \substack{ (\lambda_1,\lambda_2)\in \mathfrak{L}(M,b_1,b_2) \\} }} e(\xi|m|^2) \bigr|.
\end{align}

To allow application of a point wise bound on the number of points in $\Gamma(l_1,d_1,l_2,d_2)$ we now show that lattices with large second successive minimum $b_2$ only give negligible contribution to $\mathcal{S}_4$. By \eqref{sucmini} $|b_2|$ is large if and only if $|b_1|$ is small. Using \eqref{latticepoints}, $(d_1,d_2,\Im(l_1\overline{l_2}))\leq \mathcal{L}_3$, Cauchy's inequality, and Lemma \ref{X} we get the trivial bound
\begin{align*}
\mathcal{S}_4(B,D,L,M)&\ll \max\{\frac{\sqrt{M}}{\sqrt{B}},\frac{M\mathcal{L}_3}{D^2} \} \sum_{\substack{|l_j|^2\sim L\\ |l_1|^2\neq |l_2|^2 \\ } }^* \sum_{\substack{ d_j\sim D\\ \substack{(d_j,\Im(l_1\overline{l_2}))\leq \mathcal{L}_3\\(d_j,\Im(l_1l_2))\leq \mathcal{L}_3}}} \sum_{\substack{\mathfrak{B}(l_j,d_j)=(b_1,b_2)\\ \substack{|b_1|^2\sim B\\b_1*l_j\neq 0 }}}1\\
&\leq \max\{\frac{\sqrt{M}}{\sqrt{B}},\frac{M\mathcal{L}_3}{D^2} \}\sum_{\substack{|l_j|^2\sim L\\ |l_1|^2\neq |l_2|^2 \\ } }^*\sum_{\substack{ \\ \substack{|b_1|^2\sim B\\b_1*l_j\neq 0 }}}\tau(|b_1*l_1|)\tau(|b_1*l_2|) \\
&\ll \max\{\frac{\sqrt{M}}{\sqrt{B}},\frac{M\mathcal{L}_3}{D^2} \} L^2 B (\log x)^{11}.
\end{align*}
By the range conditions for $B$ and $D$, the contribution of the first argument of the maximum is admissable for the required estimate \eqref{s4bound}. If
\begin{align*}
B\leq \frac{D^2}{\mathcal{L}_1^4 \mathcal{L}_3(\log x)^{16}},
\end{align*}
then it holds that
\begin{align*}
\frac{M\mathcal{L}_3 L^2 B (\log x)^{11}}{D^2}\leq \frac{ML^2}{\mathcal{L}_1^4(\log x)^{5}},
\end{align*}
which is sufficient for \eqref{s4bound}. So it remains to show the estimate in the range
\begin{align*}
x^{(\log \log x)^{-3}}\leq L &< x^{1/3} \\
x\mathcal{L}_1^{-1} < ML &\leq x \\
D&\leq x^{1/3-10^{-8}}\\
\frac{D^2}{\mathcal{L}_1^4 \mathcal{L}_3(\log x)^{16}} < B &\leq D^2.
\end{align*}

The remaining steps are standard and include a Weyl differencing step. We use the parametrisation of $m$ and apply the triangle inequality on the right hand side of \eqref{s4} to get
\begin{align*}
\mathcal{S}_4(B,D,L,M)\leq \sideset{}{'}\sum\sum_{\lambda_2\in \mathcal{L}_2 }\bigl|\sum_{\lambda_1\in \mathfrak{L}_1(\lambda_2)} e(\xi[\lambda_1^2 |b_1|^2+2 \lambda_1 \lambda_2 b_1*b_2]) \bigr|,
\end{align*}
where $\mathfrak{L}_1(\lambda_2)=\mathfrak{L}_1(\lambda_2,M,b_1,b_2)$ and $\mathfrak{L}_2=\mathfrak{L}_2(M,b_1,b_2)$ are as defined in \eqref{L2def} and \eqref{L1def}. To sum nontrivially over $\lambda_1$ we apply Cauchy's inequality as follows. We have
\begin{align*}
&(\mathcal{S}_4(B,D,L,M))^2\\
\leq& \Bigl( \sideset{}{'}\sum\sum_{\lambda_2\in \mathfrak{L}_2 }1 \Bigr)\times \Bigl( \sideset{}{'}\sum\sum_{\lambda_2\in \mathfrak{L}_2 }\bigl|\sum_{\lambda_1\in \mathfrak{L}_1(\lambda_2)} e(\xi[\lambda_1^2 |b_1|^2+2 \lambda_1 \lambda_2 b_1*b_2]) \bigr|^2\Bigr)\\
=&\Bigl( \sideset{}{'}\sum\sum_{\lambda_2\in \mathfrak{L}_2 }1 \Bigr) \times \mathcal{G}_1(B,D,L,M),
\end{align*}
say. To estimate the first part of the above, we note that
\begin{align*}
|\mathfrak{L}_2 |&\ll \frac{\sqrt{M}}{|b_2|}\\
&\ll \frac{\sqrt{M}|b_1|}{\Delta}\\
&\ll \frac{\sqrt{M}\mathcal{L}_3}{D}.
\end{align*}
We get
\begin{align*}
\sideset{}{'}\sum\sum_{\lambda_2\in \mathfrak{L}_2 }1 &= \sum_{\substack{|l_j|^2\sim L\\ |l_1|^2\neq |l_2|^2 \\ } }^* \sum_{\substack{ d_j\sim D\\ \substack{(d_j,\Im(l_1\overline{l_2}))\leq \mathcal{L}_3\\(d_j,\Im(l_1l_2))\leq \mathcal{L}_3}}}  \sum_{\substack{\mathfrak{B}(l_j,d_j)=(b_1,b_2)\\ \substack{|b_1|^2\sim B\\b_1*l_j\neq 0 }}}\sum_{\lambda_2\in \mathcal{L}_2 }1 \\ &\ll \frac{\sqrt{M}\mathcal{L}_3}{D}\sum_{\substack{|l_j|^2\sim L\\ |l_1|^2\neq |l_2|^2 \\ } }^*\sum_{|b_1|^2\sim B}\tau(|b_1*l_1|)\tau(|b_1*l_2|)\\
&\ll \sqrt{M}L^2 D \mathcal{L}_3 (\log x)^{11}
\end{align*}
and so have to show
\begin{align*}
\mathcal{G}_1(B,D,L,M)\ll \frac{M^{3/2}L^2}{D\mathcal{L}_3 \mathcal{L}_1^8 (\log x)^{21}}
\end{align*}
in the same ranges as before. Expanding the square in the definition of $\mathcal{G}_1$ and denoting the new variables by $\lambda_{1,k}$ gives us
\begin{align*}
&\mathcal{G}_1(B,D,L,M)\\
=&\sideset{}{'}\sum\sum_{\lambda_2\in \mathfrak{L}_2 } \sum_{\lambda_{1,k}\in \mathfrak{L}_1(\lambda_2)} e(\xi[(\lambda_{1,1}^2-\lambda_{1,2}^2)|b_1|^2+2(\lambda_{1,1}-\lambda_{1,2})\lambda_2b_1*b_2]).
\end{align*}
We write $\lambda_{1,2}=\lambda_{1,1}+h$ and want to sum with cancellation over $\lambda_{1,1}$, so we need a notation for the resulting summation ranges. Let
\begin{align*}
\mathcal{H}(\lambda_2)&=\{h: \exists \lambda_{1,1}\in \mathfrak{L}_1(\lambda_2) \text{ s.t. } \lambda_{1,1}+h\in \mathfrak{L}_1(\lambda_2) \}\\
\mathcal{I}(\lambda_2,h)&=\{\lambda_{1,1}\in \mathfrak{L}_1(\lambda_2) :\lambda_{1,1}+h\in \mathfrak{L}_1(\lambda_2) \}
\end{align*}
and note
\begin{align}
|\mathcal{H}(\lambda_2)|&\ll  |\mathfrak{L}_1(\lambda_2)|\nonumber \\
|\mathcal{I}(\lambda_2,h)|&\ll |\mathfrak{L}_1(\lambda_2)|\label{I}.
\end{align}
Using this notation we have
\begin{align*}
\mathcal{G}_1(B,D,L,M)&=\sideset{}{'}\sum\sum_{\lambda_2\in \mathcal{L}_2 } \sum_{\substack{h\in \mathcal{H}(\lambda_2)\\\lambda_{1,1}\in \mathcal{I}(\lambda_2,h)}}e(\xi[|b_1|^2(2h\lambda_{1,1}-h^2)+2h\lambda_{2,1}*b_2])\\
&\leq \sideset{}{'}\sum\sum_{\substack{\lambda_2\in \mathcal{L}_2 \\ h\in \mathcal{H}(\lambda_2)}}\bigl|\sum_{\lambda_{1,1}\in \mathcal{I}(\lambda_2,h)}e(2 \xi |b_1|^2 h \lambda_{1,1}) \bigr|. 
\end{align*}
By \eqref{L1}, \eqref{L3}, and \eqref{I} we have
\begin{align*}
|\mathcal{I}(\lambda_2,h)|\ll \frac{\sqrt{M}}{\sqrt{B}}
\end{align*}
and the right hand side of above is always greater than $1$ in the ranges we have to consider. An application of the usual geometric series bound consequently gives us
\begin{align*}
\mathcal{G}_1(B,D,L,M)\ll\sideset{}{'}\sum\sum_{\substack{\lambda_2\in \mathfrak{L}_2 \\ h\in \mathcal{H}(\lambda_2)}}\min \{\frac{\sqrt{M}}{\sqrt{B}},|| 2\xi |b_1|^2 h||^{-1} \}.
\end{align*}
For all $\lambda_2$ we have
\begin{align*}
\mathcal{H}(\lambda_2)\subset [-\sqrt{M/B},\sqrt{M/B}]
\end{align*} 
and by \eqref{discr}, \eqref{sucmini}, and \eqref{L2} further can bound the number of elements in $\mathfrak{L}_2$ by
\begin{align*}
|\mathfrak{L}_2|&\ll \frac{(d_1,d_2,\Im(l_1\overline{l_2}))\sqrt{M}\sqrt{B}}{d_1d_2}\\
&\ll  \frac{\mathcal{L}_3\sqrt{M}\sqrt{B}}{D^2},
\end{align*}
where the right hand side of above is always greater than $1$ for sufficiently large $x$ in the ranges we need to consider. Plugging in these considerations gives us
\begin{align*}
&\mathcal{G}_1(B,D,L,M)\\
\ll& \frac{\mathcal{L}_3\sqrt{M}\sqrt{B}}{D^2} \sideset{}{'}\sum
\sum_{\substack{ \\ |h|\leq \sqrt{M/B}}}\min \{\frac{\sqrt{M}}{\sqrt{B}},|| 2\xi |b_1|^2 h||^{-1} \} \\
=& \frac{\mathcal{L}_3\sqrt{M}\sqrt{B}}{D^2}\sum_{\substack{|l_j|^2\sim L\\ |l_1|^2\neq |l_2|^2 \\ } }^* \sum_{\substack{ d_j\sim D\\ \substack{(d_j,\Im(l_1\overline{l_2}))\leq \mathcal{L}_3\\(d_j,\Im(l_1l_2))\leq \mathcal{L}_3}}}\!\! \sum_{\substack{\mathfrak{B}(l_j,d_j)=(b_1,b_2)\\ \substack{|b_1|^2\sim B\\b_1*l_j\neq 0 }}}\sum_{\substack{ \\ |h|\leq \sqrt{M/B}}}\!\! \min \{\frac{\sqrt{M}}{\sqrt{B}},|| 2\xi |b_1|^2 h||^{-1} \}.
\end{align*}
There is no longer any dependence on $b_2$ and since we choose only one basis for each lattice we can discard the summation over it. With this we are ready for the last step.

\subsubsection*{Step 3 - Back to Rational Integers}
We interchange the order of summation and relax the condition on $b_1$ to be any element of $\Gamma(l_1,d_1,l_2,d_2)$. This shows
\begin{align*}
&\mathcal{G}_1(B,D,L,M)\\
\ll& \frac{\mathcal{L}_3\sqrt{M}\sqrt{B}}{D^2}\sum_{\substack{|l_j|^2\sim L\\ |l_1|^2\neq |l_2|^2 \\ } }^*\sum_{\substack{\\ \substack{|b_1|^2\sim B\\b_1*l_j\neq 0 }}}\sum_{\substack{ \\ |h|\leq \sqrt{M/B}}}\!\!\min \{\frac{\sqrt{M}}{\sqrt{B}},|| 2\xi |b_1|^2 h||^{-1} \}\sum_{\substack{d_j\sim D\\d_j|b_1*l_j}}1\\
\leq& \frac{\mathcal{L}_3\sqrt{M}\sqrt{B}}{D^2}\sum_{\substack{|l_j|^2\sim L\\ |l_1|^2\neq |l_2|^2 \\ } }^*\sum_{\substack{\\ \substack{|b_1|^2\sim B\\b_1*l_j\neq 0 }}}\sum_{\substack{ \\ |h|\leq \sqrt{M/B}}}\!\! \min \{\frac{\sqrt{M}}{\sqrt{B}},|| 2\xi |b_1|^2 h||^{-1} \}\tau(|b_1*l_1|)\tau(|b_1*l_2|)\\
=&\frac{\mathcal{L}_3\sqrt{M}\sqrt{B}}{D^2}\mathcal{G}_2(B,L,M),
\end{align*}
say. Since 
\begin{align*}
B\leq D^2,
\end{align*}
it suffices to show
\begin{align*}
\mathcal{G}_2(B,L,M)\ll \frac{ML^2}{\mathcal{L}_3^2\mathcal{L}_1^8(\log x)^{21}}.
\end{align*}

The remaining steps are to transfer $\mathcal{G}_2$ back into a sum over integers and deal with the appearing multiplicities. Our tool of choice for this is Cauchy's inequality, which we now apply multiple times. We start with
\begin{align*}
\bigl(\mathcal{G}_2(B,L,M)\bigl)^2\ll& \Bigl(\frac{\sqrt{M}}{\sqrt{B}}\sum_{\substack{|l_j|^2\sim L\\ |l_1|^2\neq |l_2|^2 \\ } }^*\sum_{\substack{\\ \substack{|b_1|^2\sim B\\b_1*l_j\neq 0 }}}\sum_{\substack{ \\ |h|\leq \sqrt{M/B}}}\tau(|b_1*l_1|)^2\tau(|b_1*l_2|)^2 \Bigr)\\
&\times \Bigl(\sum_{\substack{|l_j|^2\sim L\\ |l_1|^2\neq |l_2|^2 \\ } }^*\sum_{\substack{\\ \substack{|b_1|^2\sim B\\b_1*l_j\neq 0 }}}\sum_{\substack{ \\ |h|\leq \sqrt{M/B}}} \min \{\frac{\sqrt{M}}{\sqrt{B}},|| 2\xi |b_1|^2 h||^{-1} \}\Bigr)\\
=&\Bigl(\frac{\sqrt{M}}{\sqrt{B}}\sum_{\substack{|l_j|^2\sim L\\ |l_1|^2\neq |l_2|^2 \\ } }^*\sum_{\substack{\\ \substack{|b_1|^2\sim B\\b_1*l_j\neq 0 }}}\sum_{\substack{ \\ |h|\leq \sqrt{M/B}}}\tau(|b_1*l_1|)^2\tau(|b_1*l_2|)^2 \Bigr)\\
&\times \mathcal{G}_3(B,L,M),
\end{align*}
say. We use Lemma \ref{X} to estimate
\begin{align*}
\frac{\sqrt{M}}{\sqrt{B}}\sum_{\substack{|l_j|^2\sim L\\ |l_1|^2\neq |l_2|^2 \\ } }^*\sum_{\substack{\\ \substack{|b_1|^2\sim B\\b_1*l_j\neq 0 }}}\sum_{\substack{ \\ |h|\leq \sqrt{M/B}}}\tau(|b_1*l_1|)^2\tau(|b_1*l_2|)^2\ll M L^2 (\log x)^{131}
\end{align*}
and have thus to show
\begin{align*}
\mathcal{G}_3(B,L,M)\ll \frac{ML^2}{\mathcal{L}_3^4 \mathcal{L}_1^{16} (\log x)^{173}}.
\end{align*}
We recall $\xi=(|l_1|^2-|l_2|^2)\gamma/W$ and write
\begin{align*}
u(j)=\sum_{\substack{\mathfrak{B}(l_j,d_j)=(b_1,b_2)\\ \substack{|b_1|^2\sim B\\b_1*l_j\neq 0 }}}\sum_{\substack{ \\ |h|\leq \sqrt{M/B}}}1_{(j=[|l_1|^2-|l_2|^2]|b_1|^2 h)}
\end{align*}
so that
\begin{align*}
\mathcal{G}_3(B,L,M)&=\sum_{j\leq 4L \sqrt{MB}}u(j) \min\{\frac{\sqrt{M}}{\sqrt{B}},||2\gamma j||^{-1}\}\\
&\leq \bigl(\sum_{j\leq 4L \sqrt{MB}}u(j)^2 \bigr)^{1/2} \bigl( \frac{\sqrt{M}}{\sqrt{B}}\sum_{j\leq 4L \sqrt{MB}} \min\{\frac{\sqrt{M}}{\sqrt{B}},||2\gamma/W j||^{-1}  \bigr)^{1/2}
\end{align*}
To estimate the first sum, we crudely bound the number of representations by the divisor function and get
\begin{align*}
u(j)&\leq \sum_{\substack{t_j\sim L \\ t_3 \sim B}}\sum_{|h|\leq \sqrt{M/B}}\tau(t_1)\tau(t_2)\tau(t_3) 1_{j=(t_1-t_2)t_3h}\\
&\leq \sum_{\substack{t'\leq 2L\\ t_3 \sim B}}\sum_{|h|\leq \sqrt{M/B}}1_{j=t't_3h} \tau(t_3) \sum_{t_2\sim L}\tau(t'+t_2)\tau(t_2).
\end{align*}
After another application of Cauchy's inequality, the sum over $t_2$ can be estimated uniformly in $t'$ by $L (\log x)^3$ and so
\begin{align*}
u(j)&\ll L(\log x)^3\sum_{\substack{t'\leq 2L\\ t_3 \sim B}}\sum_{|h|\leq \sqrt{M/B}}1_{j=t't_3h}\tau(t_3) \\
&\leq L (\log x)^3  \tau_4(j).
\end{align*}
Thus
\begin{align*}
\sum_{j\leq 4L\sqrt{MB}} u(j)^2 &\ll L^2 (\log x)^3 \sum_{4L\sqrt{M/B}}\tau_4(j)^2 \\
&\ll L^3 \sqrt{MB} (\log x)^{18}.
\end{align*}
and to complete the proof we need to show
\begin{align}\label{eq_minbound}
\sum_{j\leq 4L \sqrt{MB}} \min\{\frac{\sqrt{M}}{\sqrt{B}},||2\gamma/W j||^{-1}\}
 &\ll \frac{LM}{\mathcal{L}_3^8 \mathcal{L}_1^{32}(\log x)^{364}}.
\end{align}
By Dirichlet's approximation theorem, if $\gamma\in \mathfrak{m}$ then we can find a rational approximation 
\begin{align*}
|\gamma-\frac{a}{q}|\leq q^{-2}
\end{align*}
for some 
\begin{align}\label{qrangeAM}
(\log x)^{A_\mathfrak{M}}\leq q \leq x (\log x)^{-A_\mathfrak{M}}
\end{align}
and $(a,q)=1$. So for some $a', q'$ with $(a',q')=1$ and $q\leq q' \leq qW$ we have
\begin{align*}
|\frac{\gamma}{W}-\frac{a'}{q'}|\leq \frac{W}{q'^2}.
\end{align*}
After bounding the contribution of $j=0$ by $\sqrt{M}/\sqrt{B}$, an application of Lemma \ref{minsumW} gives us 
\begin{align}\label{estimate}
\sum_{j\leq 4L \sqrt{MB}} \min\{\frac{\sqrt{M}}{\sqrt{B}},||2\gamma/W j||^{-1}\}&\ll W^2\bigl(\frac{LM}{q}+\frac{\sqrt{M}}{\sqrt{B}}+L\sqrt{MB}+q \bigr)\log q.
\end{align}
The term $\sqrt{M}/\sqrt{B}$ is negligible. By our range conditions we have
\begin{align*}
B&\leq D^2 \ll x^{2/3-10^{-8}},\\
M&>\frac{x}{\mathcal{L}_1 L}\geq \frac{x^{2/3}}{\mathcal{L}_1},
\end{align*}
so that the term 
\begin{align*}
L\sqrt{MB}
\end{align*}
is small enough for \eqref{eq_minbound}. We recall
\begin{align*}
\mathcal{L}_1&=(\log x)^{29}\\
\mathcal{L}_3&=(\log x)^{504}\\
W&\ll \log x
\end{align*}
so the above estimate \eqref{estimate} is sufficient for \eqref{eq_minbound}, if
\begin{align*}
(\log x)^{5327}\leq q\leq x (\log x)^{-5327},
\end{align*}
which is fulfilled by the range in \eqref{qrangeAM}, since $A_\mathfrak{M}=10^4$. This completes the Type II estimate.
\end{proof}

\subsection{Type I Sums and proof of Proposition \ref{prmajo}}
We now show minor arc Type I sum estimates. Together with the results of the previous subsections this means the conditions of Lemma \ref{decomp} are met and we can use it to deduce that Proposition \ref{prmajo} holds for $\gamma \in \mathfrak{m}$. Since some additional flexibility is convenient for the $L^q$ restriction estimate in the next section, we state the Type I result in dependence of a parameter $B$.
\begin{lemma}\label{typIminor}
Let $x$, $W$ and $b$ as in Theorem \ref{MT}, $B>0$, and $\omega$ be a complex valued weight with $|\omega(l)|\leq (\log l)^{c_1}\tau(l)^{c_2}$. Let further $|\gamma-a/q|\leq q^{-2}$ with $(a,q)=1$, 
\begin{align*}
(\log x)^B\leq q \leq x(\log x)^{-B},
\end{align*}
and assume that 
\begin{align}\label{D_Irange}
D_I\leq \sqrt{x}(\log x)^{-B}.
\end{align}
Then it holds that
\begin{align*}
\mathcal{R}^\omega(D_I):=\sum_{d\leq D_I}\bigl| \sum_{\substack{dn\leq x\\dn\equiv b(W)}}\sum_{dn=k^2+l^2}\omega(l)e( dn\gamma/W)\bigr|\ll x(\log x)^{c_1+2^{c_2}+5-B/8}.
\end{align*}
\end{lemma}
\begin{proof}
We start by rewriting
\begin{align*}
\mathcal{R}^\omega(D_I)=\sum_{d\leq D_I}\bigl|\sum_{\substack{a_1,a_2(W)\\ a_1^2+a_2^2\equiv b(W)}}\sum_{\substack{b_1,b_2(d)\\ b_1^2+b_2^2\equiv 0(d)}}\sum_{\substack{k\leq \sqrt{x}\\ \substack{ k\equiv b_1(d)\\ k\equiv a_1 (W) }}}e(k^2\gamma/W) \sum_{\substack{l\leq \sqrt{x-k^2}\\ \substack{ l\equiv b_2(d)\\ l\equiv a_2 (W) }}}\omega(l)e( l^2 \gamma/W)\bigr|.
\end{align*}
Similar to the Type II approach we dissect $\mathcal{R}^\omega(D_I)$ into short sums. Assume we are given $D$, $K$, $L$, and $\mathcal{L}$. We set
\begin{align*}
D'&=2D\\
K'&=2K\\
L'&=e^{\mathcal{L}^{-1}}L
\end{align*}
and, as before, write $d \sim D$ for the range $D< d \leq D'$. With this notation we define the short sums
\begin{align*}
\mathcal{R}^\omega(D,K,L)=\sum_{d\sim D}\bigl|\sum_{\substack{a_1,a_2(W)\\ a_1a_2\equiv b(W)}}\sum_{\substack{b_1,b_2(d)\\ b_1^2+b_2^2\equiv 0(d)}}\sum_{\substack{k\sim K\\ \substack{ k\equiv b_1(d)\\ k\equiv a_1 (W) }}}e( k^2\gamma/W) \sum_{\substack{l\sim L \\ \substack{ l\equiv b_2(d)\\ l\equiv a_2 (W) }}}\omega(l)e(l^2 \gamma/W)\bigr|.
\end{align*}
The contribution of terms with 
\begin{align*}
k^2+l^2&\leq x \\
(le^{\mathcal{L}^{-1}})^2+(2k)^2&>x
\end{align*}
to $\mathcal{R}^\omega(D_I)$, that may not be covered exactly, can be bound trivially by
\begin{align}\label{TItrivial}
O(\frac{x(\log x)^{A_1}}{\mathcal{L}}),
\end{align}
where
\begin{align*}
A_1\leq c_1+2^{c_2}+5.
\end{align*}
We note for later that there are fewer than 
\begin{align}\label{T1sumnumber}
2\mathcal{L}(\log x)^2
\end{align}  short sums in $\mathcal{R}^\omega(D_I)$ and that it suffices to consider the range
\begin{align}
D&\leq D_I \label{Drange}\\
K&\leq \sqrt{x}\label{Krange}\\
L&\leq \sqrt{x}\label{Lrange}.
\end{align}
We proceed by applying the triangle inequality and estimate the sum over $l$ by Shiu's Theorem (see \cite{shiu}) to get
\begin{align*}
\mathcal{R}^\omega(D,K,L)\leq (\log x)^{A_2}\bigl(\frac{L }{D}+1\bigr)\sum_{d\sim D}\tau(d)\sum_{b_1(d)}\sum_{a_1(W)} \bigl|\sum_{\substack{k\sim K \\ \substack{ k\equiv b_1(d)\\ k\equiv a_1 (W) }}}e( k^2 \gamma/W)\bigr|,
\end{align*}
where
\begin{align*}
A_2=c_1+2^{c_2}-1.
\end{align*}
We use Cauchy's inequality and the bound $W\leq \log x$ to get
\begin{align*}
\Bigl(\sum_{d\sim D}\sum_{b_1(d)}\tau(d)\sum_{a_1(W)} \bigl|\sum_{\substack{k\sim K\\ \substack{ k\equiv b_1(d)\\ k\equiv a_1 (W) }}}e( k^2\gamma/W)\bigr|\Bigr)^2\leq D^2 (\log x)^4 \sum_{d\sim D} \sum_{\substack{b_1(d)\\ a_1(W)}}\bigl|\sum_{\substack{k\sim K \\ \substack{ k\equiv b_1(d)\\ k\equiv a_1 (W) }}}e(k^2 \gamma/W)\bigr|^2.
\end{align*}
Expanding the square with new variables $k_1$ and $k_2$ gives us
\begin{align*}
\sum_{d\sim D} \sum_{\substack{b_1(d)\\ a_1(W)}}\bigl|\sum_{\substack{k\sim K \\ \substack{ k\equiv b_1(d)\\ k\equiv a_1 (W) }}}e( k^2 \gamma/W)\bigr|^2&=\sum_{d\sim D} \sum_{\substack{b_1(d)\\ a_1(W)}} \sum_{\substack{k_j\sim K \\ \substack{ k_j\equiv b_1(d)\\ k\equiv a_1 (W) }}}e((k_1^2-k_2^2)\gamma/W)\\
&=\sum_{d\sim D}\sum_{\substack{k_j\sim K \\ k_1\equiv k_2 ([d,W])}}e( (k_1^2-k_2^2)\gamma/W).
\end{align*}
We write $k_2=k_1+k'$ and note that the congruence condition is then equivalent to $[d,W]|k'$. We set 
\begin{align*}
K(k')=\{k_1\sim K : k_1+k'\sim K \}
\end{align*} 
and get by separating the diagonal contribution of terms with $k'=0$ and the geometric series bound
\begin{align*}
\sum_{d\sim D}\sum_{\substack{k_j\sim K \\ k_1\equiv k_2 ([d,W])}}e(\gamma (k_1^2-k_2^2))&\leq \sum_{d\sim D}\sum_{\substack{|k'|\leq 4 K\\ [d,W]|k'}} \bigl|\sum_{k_1\in K(k')}e(2 k_1 k' \gamma/W)  \bigr|\\
&\ll \sum_{d\sim D}\sum_{\substack{0<k'\leq 4 K\\ [d,W]|k'}} \bigl|\sum_{k_1\in K(k')}e(2 k_1 k' \gamma/W)  \bigr|+DK\\
&\ll \sum_{d\sim D} \sum_{\substack{0<k'\leq 4 K\\ [d,W]|k'}}\min\{K,||2 k'\gamma/W||^{-1} \}+DK.
\end{align*}
We now use the assumed rational approximation $|\gamma-a/q|\leq q^{-2}$ for some coprime $a$ and $q$. By relaxing $[d,W]|k'$ to $d|k'$ and applying Lemma \ref{minsumW} just like in \eqref{estimate}, we estimate
\begin{align*}
\sum_{d\sim D} \sum_{\substack{0<k'\leq 4 K\\ [d,W]|k'}}\min\{K,||2k'\gamma/W||^{-1} \}&\ll \sum_{0< k' \leq 4K}\tau(k')\min\{K,||2k'\gamma/W||^{-1} \} \\
&\ll W^2 \bigl(K^{3/2}+K^2 q^{-1/2}+q^{1/2}K)(\log x)^3.
\end{align*}
Combining this with $W \leq \log x$ and the range conditions \eqref{Drange}, \eqref{Krange}, \eqref{Lrange}, we get
\begin{align*}
R^\omega(D,K,l)\ll (\log x)^{A_2+6} \bigl(x^{7/8}+xq^{-1/4}+q^{1/4}x^{3/4}+D_I^{1/2}x^{3/4}\Bigr)
\end{align*}
We now choose $\mathcal{L}=(\log x)^{B/8}$. By  \eqref{D_Irange}, \eqref{TItrivial}, and \eqref{T1sumnumber} get
\begin{align*}
R^\omega(D)\ll x(\log x)^{A_2+6-B/8},
\end{align*}
as stated.
\end{proof}
We conclude that our majorant behaves indeed pseudorandomly.
\begin{proof}[Proof of Proposition \ref{prmajo}]
By Lemma \ref{prmajomajor} the Theorem holds for $\gamma\in \mathfrak{M}$. We want to use Lemma \ref{decomp} to deal with the remaining $\gamma\in \mathfrak{m}$. By Lemma \ref{T2boundl} the required Type II bounds holds. 
For the Type I estimate we recall the bounds \eqref{o1bound}, \eqref{o2bound}, and \eqref{o3bound} that give finite values for $c_1$ and $c_2$ in the application of Lemma \ref{typIminor}. Since $A_\mathfrak{M}$ is large enough and $D_I\ll_A x^{1/2}(\log x)^{-A}$ for any fixed $A>0$, we use Lemma \ref{typIminor} to get the Type I bound required for Lemma \ref{decomp}.
\end{proof}

\section{$L^q$ Estimate.}\label{sec6}
We now prove that condition II. of Theorem \ref{MT} holds. Starting with Green's work the $L^q$ restriction theory is a central element of the transference principle. In many interesting cases it can be proved using the work of Green and Tao on the Selberg sieve \cite{gt} or the similar results of Ramare and Rusza \cite{rr}, see for example \cite{mms} or \cite{te}. However, this seems not to be applicable in our case and we instead use an older idea of Bourgain \cite{bou}. That idea was applied also in \cite{har}, \cite{bp}, and \cite{cho}. We roughly follow  the proofs in \cite{har} and \cite{bp}. The main ingredient is the following idea from Bourgain that is also stated on p. 16 of \cite{har}.
\begin{lemma}\label{Bou}
Let $X$ be large, $Q\geq 1$ and $1/X \leq \Delta \leq 1/2$. Let further $\epsilon>0$ and $A>0$ be arbitrary. Define
\begin{align*}
G(\theta):=\sum_{q\leq Q}\frac{1}{q}\sum_{\substack{a(q)\\||\theta-a/q||\leq \Delta}}\frac{1}{1+X||\theta-a/q||}.
\end{align*}
Then for $1/X$ well spaced points $\theta_1,\ldots,\theta_R$ we have
\begin{align*}
\sum_{1\leq r,r'\leq R}G(\theta_r-\theta_{r'})\ll_{\epsilon,A}RQ^\epsilon \log(1+\Delta X)+\frac{R^2Q \log(1+\Delta X)}{X}+\frac{R^2\log(1+\Delta X)}{Q^A}.
\end{align*}
\end{lemma}
To apply this result we require a majorant for the Fouvry-Iwaniec primes for which we have strong minor arc bounds. This is necessary because we need to overcome the logarithmic loss of the $L^2$ bound. Since Lemma \ref{decomp} only saves a factor of $\log \log x $ over the trivial estimate, $\Lambda^+$ is not suitable for this application. We instead use a more simple approach that is similar to $\Lambda_3$.

Let \begin{align*}
z_0&=e^{(\log x)^{1/3}}\\
D_0&=e^{(\log x)^{2/3}}\\
z&=x^{1/4}\\
D_1&=x^{1/3}
\end{align*}
and
\begin{align*}
\theta^+(n,x)=\theta^+(n,D_1,D_0,z,z_0)
\end{align*}
as defined in \eqref{mainsieve}. Then we have
\begin{align*}
\Lambda(n)\leq (\log x)( \theta^+(n,x)+E'(n)),
\end{align*}
where $E'(n)$ accounts for primes less than $z$ and their powers. We note
\begin{align}\label{E'}
\sum_{n\leq x}|E'(n)|\sum_{n=y^2+l^2}\Lambda(l)\ll x^{1/3}.
\end{align}
The majorant we use for the $L^q$ estimate is
\begin{align}\label{nudef}
\nu(n,x)=(\theta^+(n,x)+E'(n))\log x\sum_{n=l^2+k^2}\Lambda(l).
\end{align}
For any complex valued function $f$ and integers $W$, $b$ recall the notation
\begin{align*}
\hat{f}_{W,b}(\gamma)=\sum_{n\leq N}f_{W,b}(n)e(\gamma n).
\end{align*}
We show the following result that implies condition II. of Theorem \ref{MT}.  
\begin{lemma}\label{Lq}
Assume we are given $x$, $N$, $W$, and $b$ as in Theorem \ref{MT}. Let $\nu(n,x)$ as defined in \eqref{nudef} and $2<\mathfrak{q}<3$. There exists $C_\mathfrak{q}$ such that the following holds. For any $\phi: \mathbbm{Z}\to \mathbbm{C}$ with $|\phi|\leq |\nu|$ we have
\begin{align*}
\int_0^1 \bigl|\hat{\phi}_{W,b}(\gamma) d\gamma\bigr|^\mathfrak{q}\leq C_\mathfrak{q} N^{\mathfrak{q}-1}.
\end{align*}
\end{lemma}
\begin{proof}
By \eqref{E'} the contribution of $E'(n)$ is negligible. We start by considering the distribution of the majorant in arithmetic progressions. Using the approach for $i=3$ in subsection \ref{L3subse} we have for any fixed $A>0$ uniformly in $T\leq x$, $q\leq (\log x)^A$ and $a(q)$ that
\begin{align}\label{Lqmajorantcongruence}
\sum_{\substack{n\leq T\\ n\equiv a(q)}}\nu(n)=\frac{\Xi(q,a)}{\varphi(q)}C(x)\int_0^{\sqrt{T}}\sqrt{T-t^2}+O_A(x(\log x)^{-A})
\end{align}
for some function $C(x)\asymp 1$ that depends on the sieve. Recall that $\Xi(q,a)=0$ if $(a,q)\neq 1$. 

To apply Lemma \ref{Bou}, we require major and minor arc information for $\hat{\nu}_{W,b}(\gamma)$. Let $A_\mathfrak{q}>0$ and define $\mathfrak{M}(q,a)$ and $\mathfrak{M}$ as in \eqref{Mqadef} and \eqref{Mganzdef}, except with $A_\mathfrak{q}$ taking the role of $A_\mathfrak{M}$. We start by considering the case $\gamma\in \mathfrak{M}(q,a)$. As we have the asypmtotics \eqref{Lqmajorantcongruence}, we can apply Lemma \ref{marc2} and note that $\mathcal{F}$ is constant. This gives us for any $\epsilon_0>0$ the estimate
\begin{align*}
|\sum_{n\leq N}\nu(Wn+b)e(\gamma n)|\ll_{\epsilon_0} \frac{\Xi(W,b)W}{\varphi(W)q^{1-\epsilon_0}}\min\{N,||\gamma-a/q||^{-1}\}+x(\log x)^{-A_\mathfrak{q}}.
\end{align*}
With the notation of the Lemma and our choice of normalisation \eqref{normal} it follows that
\begin{align}\label{numajor}
|\hat{\nu}_{W,b}(\gamma)|\ll_{\epsilon_0} q^{-1+\epsilon_0}\frac{N}{1+N|\gamma-a/q|}+x(\log x)^{-A_\mathfrak{q}}.
\end{align}
Let now $\gamma\in \mathfrak{m}$. Since $\theta$ is a sieve, a Type I estimate is sufficient to estimate the minor arc case. By Lemma \ref{typIminor} we have
\begin{align}\label{numinor}
|\hat{\nu}_{W,b}(\gamma)|\ll x(\log x)^{7-A_q/8}.
\end{align}

Before we can use this information, we obtain a crude estimate for the second moment. We have
\begin{align*}
\int_0^1 |\hat{\phi}_{W,b}(\gamma)|^2 d\gamma &= \sum_{n\leq N}|\phi_{W,b}(n)|^2 \\
&\leq \sum_{n\leq N}|\nu_{W,b}(n)|^2 \\
&\leq \bigl(\frac{\varphi(W)}{\Xi(W,b)WH}\bigr)^2 \sum_{n\leq x} |\nu(n)|^2.
\end{align*}
By Cauchy's inequality and the bound $|\theta^+(n)|\leq \tau(n)$ we have
\begin{align*}
|\nu(n)|^2&=(\log x)^2 (\theta^+(n,x)+E'(n))^2 |\sum_{n=y^2+l^2}\Lambda(l)|^2\\
&\ll (\log x)^4 \tau(n)^4.
\end{align*}
This gives us the $L^2$ estimate
\begin{align}\label{lq2}
\int_0^1 |\hat{\phi}_{W,b}(\gamma)|^2 d\gamma \ll N (\log x)^{20}.
\end{align}
We could obtain a more sharp estimate with respect to appearing power of $\log x$. However, we cannot completely prevent this logarithmic loss.

Let $0<\delta<1$ be a parameter and set
\begin{align}
N'=\sum_{n\leq N}\nu_{W,b}(n).
\end{align} 
By \eqref{Lqmajorantcongruence} we have $N'\asymp N$. We define the set of large Fourier coefficients of $\phi_{W,b}$ by
\begin{align*}
\mathfrak{R}_\delta=\Bigl\{\gamma\in [0,1]:|\hat{\phi}_{W,b}(\gamma)|\geq \delta N' \Bigr\}.
\end{align*}
Our goal is to show for any $\epsilon_1>0$ the estimate 
\begin{align}\label{Rd}
\textsc{meas}(\mathfrak{R}_\delta)\leq \frac{C_{\epsilon_1}}{\delta^{2+\epsilon_1}N'}.
\end{align}
This is sufficient for proving Lemma \ref{Lq}. To see this, consider
 \begin{align*}
\int_0^1|\hat{\phi}_{W,b}(\gamma)|^q d\gamma&\leq \sum_{j\geq 0}\Bigl(\frac{N'}{2^{j-1}} \Bigr)^q \textsc{meas}\{\gamma\in [0,1]: \frac{N'}{2^j}\leq |\hat{\phi}_{W,b}(\gamma)|\leq \frac{N'}{2^{j-1}} \}\\
&\leq C_{\epsilon_1} 2^\mathfrak{q} N'^{\mathfrak{q}-1} \sum_{j\geq 0}(2^{2+\epsilon_1-\mathfrak{q}})^j
\end{align*}
and choose $\epsilon_1$ such that $2+\epsilon_1-\mathfrak{q}<0$. 

For proving \eqref{Rd} we first use \eqref{lq2} to crudely bound
\begin{align*}
\textsc{meas}(\mathfrak{R}_\delta)\ll  \frac{(\log N')^{20}}{\delta^2 N'}.
\end{align*}
If $\delta \leq  (\log N')^{-20/ \epsilon_1}$, \eqref{Rd} follows. We can thus assume from now
\begin{align}\label{delta}
\delta >  (\log N')^{-20/ \epsilon_1}.
\end{align}
Let $\gamma_1,\ldots,\gamma_R$ be $\frac{1}{N'}$ spaced points in $\mathfrak{R}_\delta$. To show \eqref{Rd} it suffices to show 
\begin{align}\label{Rbound}
R\leq \frac{C_{\epsilon_1}}{\delta^{2+\epsilon_1}}.
\end{align}
Following the argument in \cite[Section 6]{bp}, we let $a_n \in \mathbbm{C}$ for $1\leq n \leq N$ with $|a_n|\leq 1$ such that
\begin{align*}
\phi_{W,b}(n)=a_n \nu_{W,b}(n).
\end{align*}
Let further $c_r\in \mathbbm{C}$ for $1\leq r \leq R$ such that $|c_r|=1$ and
\begin{align*}
c_r \hat{\phi}_{W,b}(\gamma_r)=|\hat{\phi}_{W,b}(\gamma_r)|.
\end{align*}
By Cauchy's inequality we have
\begin{align*}
\delta^2N^2 R^2&\leq \Bigl(\sum_{ r\leq R}|\hat{\phi}_{W,b}(\gamma_r)| \Bigr)^2\\
&=\Bigl(\sum_{r \leq R}c_r \sum_{n\leq N} a_n \nu_{W,b}(n) e(\gamma_r n)\Bigr)^2\\
&\leq \Bigl(\sum_{n\leq N}\nu_{W,b}(n) \Bigr) \Bigl(\sum_{n\leq N}\nu_{W,b}(n)\sum_{r\leq R}c_r e(\gamma_r n)|^2\Bigr).
\end{align*}
Hence,
\begin{align} \label{lqmain}
\delta^2N R^2\ll \sum_{r,r'\leq R}|\hat{\nu}_{W,b}(\gamma_r-\gamma_{r'})|.
\end{align}
We put $\gamma=\gamma_r-\gamma_{r'}$ and by \eqref{numinor} can bound the contribution of $\gamma\in \mathfrak{m}$ to the sum in \eqref{lqmain} by
\begin{align*}
R^2 N (\log N)^{7-A_\mathfrak{q}/8}.
\end{align*}
This is negligible by \eqref{delta}, if $A_\mathfrak{q}$ is large enough in terms of $\epsilon_1$. Let now $\gamma\in \mathfrak{M}(q,a)$. By \eqref{numajor} we have
\begin{align*}
\hat{\nu}_{W,b}(\gamma)\ll_{\epsilon_0} q^{-1+\epsilon_0}N(1+N|\gamma-a/q|)^{-1}+x(\log x)^{-A_\mathfrak{q}}.
\end{align*}
Again if $A_\mathfrak{q}$ is sufficiently large in terms of $\epsilon_1$, the last term is admissible. We estimate the contribution of terms with $q>\delta^{-3}$ to the right hand side of \eqref{lqmain} by
\begin{align*}
O_\epsilon(R^2 N \delta^{3-\epsilon}).
\end{align*}
Similarly, terms with $|\gamma-a/q|>\delta^{-3}/N$ contribute at most
\begin{align*}
O(R^2N \delta^3).
\end{align*}
Both bounds are acceptable. We can now apply Lemma \ref{Bou} on the remaining terms of \eqref{lqmain} to get
\begin{align*}
\delta^2 R^2 &\ll_{\epsilon_0} \sum_{q\leq \delta^{-3}}q^{-1+\epsilon_0} \sum_{a(q)^*} \sum_{\substack{r,r'\leq R\\ |\gamma-a/q|\leq \delta^{-3}/N}}\bigl(1+N|\gamma_r-\gamma_{r'}-a/q| \bigr)^{-1} \\
&\ll_{\epsilon_0} R \delta^{-6\epsilon_0} \log(1+\delta^{-3})+\frac{R^2 \delta^{-3} \log(1+\delta^{-3})}{N}+R^2 \log(1+\delta^{-3})\delta^3.
\end{align*}
Since $\delta<1$ and $3>2$ the contribution of the third term can be ignored. Similarly, using \eqref{delta}, the second term is negligible. We have thus showed
\begin{align*}
R\ll_{\epsilon_0}\delta^{-2-5\epsilon_0},
\end{align*}
which is sufficient for \eqref{Rbound} after setting $\epsilon_0=\epsilon_1/5$.
\end{proof}

\section{Conclusion}
\begin{proof}[Proof of Theorem \ref{Goal}]
We apply Theorem \ref{MT} with $\Lambda^+$ as the majorant, $\alpha=\alpha^+(x)$ as given in \eqref{majolem3} and $\alpha^-=0.999$. For any fixed $\eta>0$ by Proposition \ref{prmajo} condition I. is fulfilled, if $x$ is sufficiently large.  The $L^q$ estimate holds for any fixed $2<q<3$ by Lemma \ref{Lq}, so condition II. also is true. 

Since $\alpha^+(x)<2.9739+o(1)$ we have $\alpha^+(x)<2.974$ for all sufficiently large $x$. For those $x$ we have that
\begin{align*}
3\alpha^- -\alpha^+(x)>0.02,
\end{align*}
so that we only need to consider some fixed $\eta$ for condition III. For any such $\eta$, by Corollary \ref{condc}, the condition is true. Thus $x$ can be written as the sum of three Fouvry-Iwaniec primes.
\end{proof}

We end this paper by proving Roth's Theorem in the Fouvry-Iwaniec primes as stated in the introduction.
\begin{proof}[Proof of Theorem \ref{3apcor}]
Let $X$ be a subset of the Fouvry-Iwaniec primes with indicator function $\mathbbm{1}_X(n)$ such that
\begin{align*}
\liminf_{x \to \infty}\frac{\sum_{n\leq x}\mathbbm{1}_X(n) \Lambda^\Lambda(n)}{\sum_{n\leq x}\Lambda^\Lambda(n)}=\delta_X>0.
\end{align*}
Let $w$ be a large integer and set $W=2\prod_{p\leq w}p$. By the pigeonhole principle there exists a residue class $b(W)$, $b\equiv 1(4)$, such that
\begin{align}\label{3apdensity}
\liminf_{x \to \infty}\frac{\sum_{n\leq x}\mathbbm{1}_X(Wn+b) \Lambda^\Lambda_{W,b}(n)}{\sum_{n\leq x}\Lambda^\Lambda_{W,b}(n)}\geq \delta_X.
\end{align}
We want to apply \cite[Proposition 5.1]{gt}. Assume we are given $x>0$ and $\eta>0$. We set $N=\floor{x/W}$ and for $n\in [N/4,N/2]$
\begin{align*}
f(n)&=\Lambda^\Lambda_{W,b}(n)\mathbbm{1}_X(Wn+b)\\
\nu(n)&=\Lambda^+_{W,b}(n).
\end{align*}
Clearly 
\begin{align*}
0\leq f(n)\leq \nu(n)
\end{align*}
for all $n\leq N$. Furthermore, by \eqref{3apdensity} the required mean condition  \cite[(5.4)]{gt} holds for some $\delta>0$ and sufficiently large $x$. If $W$ is large enough in terms of $\eta$, the majorant $\nu$ fulfills the pseudorandomness condition \cite[(5.5)]{gt}. Finally, by Lemma \ref{Lq} we get \cite[(5.6)]{gt} for some $M>0$ and $2<q<3$. Since we restricted ourselves to $n\in [N/4,N/2]$, we get
\begin{align*}
\sum_{n,d \leq N}f(n)f(n+d)f(n+2d)\geq N^2\bigl(\frac{1}{2}c(\delta)-O_{\delta,M,q}(\eta)\bigr).
\end{align*}
This completes the proof, if $\eta$ is chosen to be small enough.
\end{proof}

By slightly modifying the arguments, Theorem \ref{3apcor} could be extended to cover more cases for which \cite[Theorem 1]{foi} gives a nontrivial result.

\section*{Acknowledgments}
The author was a Ph.D. student at Utrecht University during this project and expresses his gratitude to his supervisor D. Schindler for helpful comments and discussion. He furthermore thanks the anonymous referee for carefully reading earlier versions of the manuscript.

\end{document}